\documentclass[manuscript,screen,nonacm]{acmart}

\setcopyright{none}
\AtBeginDocument{%
  }
\usepackage{microtype}
\usepackage{enumitem}
\usepackage{array} 

\usepackage{multirow} 
\usepackage[utf8]{inputenc}
\usepackage[T1]{fontenc}    
\usepackage{hyperref}       
\usepackage{url}            
\usepackage{booktabs}       
\usepackage{amsfonts}       
\usepackage{nicefrac}       
\usepackage{microtype}     
\usepackage[table]{xcolor}
\usepackage{comment}
\usepackage{caption}
\usepackage{natbib}
\usepackage{doi}
\usepackage{listings}
\usepackage[most]{tcolorbox}
\tcbuselibrary{listings,breakable}
\usepackage{tcolorbox}
 \usepackage{etoolbox} 
\usepackage{amsmath}
\usepackage{booktabs}
\usepackage{graphicx}
\usepackage[caption=false]{subfig} 
\tcbuselibrary{listings, skins, breakable}

\definecolor{codebg}{HTML}{FFF9F7}
\definecolor{codeframe}{HTML}{E6DAD6}
\definecolor{codekeyword}{HTML}{8250DF}
\definecolor{codestring}{HTML}{116329}
\definecolor{codecomment}{HTML}{6B6563}
\definecolor{codeidentifier}{HTML}{2B2726}
\definecolor{codenumber}{HTML}{9A8F8B}
\definecolor{codebuiltin}{HTML}{0550AE}

\definecolor{outputbg}{HTML}{F7F7F8}
\definecolor{outputtext}{HTML}{2F3337}
\definecolor{outputborder}{HTML}{B8BDC3}
\definecolor{outputcaption}{HTML}{202124}

\lstdefinestyle{prettyPython}{
    language=Python,
    basicstyle=\ttfamily\footnotesize\color{codeidentifier},
    identifierstyle=\color{codeidentifier},
    keywordstyle=\bfseries\color{codekeyword},
    keywordstyle=[2]\color{codebuiltin},
    stringstyle=\color{codestring},
    commentstyle=\itshape\color{codecomment},
    morekeywords=[2]{self,cls,True,False,None,NotImplemented,Ellipsis},
    numbers=left,
    numberstyle=\ttfamily\scriptsize\color{codenumber},
    stepnumber=1,
    numbersep=8pt,
    numberblanklines=false,
    showstringspaces=false,
    showspaces=false,
    showtabs=false,
    tabsize=4,
    keepspaces=true,
    columns=fullflexible,
    upquote=true,
    breaklines=true,
    breakatwhitespace=true,
    breakautoindent=true,
    captionpos=t,
    abovecaptionskip=0pt,
    belowcaptionskip=5pt,
    frame=none
}
\newtcblisting[
  auto counter,
  number within=section
]{pythoncode}[2][]{
  listing only,
  breakable,
  enhanced,
  colback=codebg,
  frame hidden,
  boxrule=0pt,
  arc=0pt,
  outer arc=0pt,
  boxsep=0pt,
  left=8mm,
  right=2mm,
  top=1.2mm,
  bottom=1.2mm,
  before skip=6pt,
  after skip=6pt,
  title={Listing~\thetcbcounter: #2},
  listing options={
    style=prettyPython
  },
  #1
}
\lstdefinestyle{prettyOutput}{
    language={},
    basicstyle=\ttfamily\footnotesize\color{outputtext},
    identifierstyle=\color{outputtext},
    keywordstyle=\color{outputtext},
    stringstyle=\color{outputtext},
    commentstyle=\color{outputtext},
    numbers=none,
    showstringspaces=false,
    showspaces=false,
    showtabs=false,
    tabsize=4,
    keepspaces=true,
    columns=fullflexible,
    upquote=true,
    breaklines=true,
    breakatwhitespace=false,
    breakautoindent=true,
    captionpos=t,
    abovecaptionskip=0pt,
    belowcaptionskip=5pt,
    frame=none
}

\newtcblisting{codeoutput}[2][]{
    listing only,
    breakable,
    enhanced,
    colback=outputbg,
    frame hidden,
    boxrule=0pt,
    arc=0pt,
    outer arc=0pt,
    boxsep=0pt,
    left=3mm,
    right=2mm,
    top=1.5mm,
    bottom=1.5mm,
    before skip=6pt,
    after skip=6pt,
    borderline west={
        0.6pt
    }{
        0pt
    }{
        outputborder
    },
    listing options={
        style=prettyOutput,
        caption={#2},
        #1
    }
}

\usepackage{tikz}
\usetikzlibrary{shapes.geometric, arrows.meta, positioning, backgrounds, calc, shadows.blur, fit}

\definecolor{boxGray}{RGB}{150, 150, 150}
\definecolor{boxWhite}{RGB}{255, 255, 255}

\definecolor{colorFP64}{RGB}{189, 215, 238}
\definecolor{colorFP32}{RGB}{248, 203, 173}
\definecolor{colorFP16}{RGB}{198, 194, 229}
\definecolor{colorE5M2}{RGB}{255, 242, 204}
\definecolor{colorBg}{RGB}{248, 250, 245}
\definecolor{lineGray}{RGB}{80, 80, 80}

\usetikzlibrary{positioning, fit, backgrounds, shadows, backgrounds}
\usetikzlibrary{decorations.pathreplacing}

\usepackage{color} 
\usepackage{amsmath,amsthm}
\usepackage{algorithm}
\usepackage{algpseudocode}
\usepackage{graphicx} 

\usepackage[caption=false]{subfig} 
\newtheorem{theorem}{Theorem}
\newtheorem{lemma}[theorem]{Lemma}

\newcommand{\xyy}[1]{{\color{black}{#1}}}

\begin{document}
\title{Automated Numerical Stability Analysis of Deep Learning Operators}
% Numerical Stability Detection for Deep Learning Operators
% Automated Numerical Stability Analysis of Deep Learning Operators
% Operator-Level Numerical Stability Detection for Deep Neural Networks
% A Unified Framework for Detecting Numerical Instability in Deep Learning Operators
\author{Xinye Chen}
\orcid{0000-0003-1778-393X}
\email{xinye.chen@lip6.fr}
\affiliation{%
  \institution{Sorbonne Université, CNRS, LIP6}
  \city{Paris}
  \country{France}
  \thanks{The work on this paper is still ongoing, and the paper is expected to be submitted soon. This work was supported by the France 2030 NumPEx Exa-MA project (ANR-22-EXNU-0002) managed by the French National Research Agency (ANR)}
}

\renewcommand{\shortauthors}{Chen}

\begin{abstract}
Finite-precision arithmetic unavoidably introduces numerical approximation errors. Numerical computations may use insufficient precision or an improper formulation, which leads to numerical instability.  In this paper, we introduce a unified software tool for stochastic numerical validation of deep-learning operators. The tool follows CESTAC on supported exposed operations and uses an operator-level data-perturbation approximation for GEMM-like kernels. Our developed software not only enables numerical validation with a single computation pass but also detects the sources of numerical instability and provides numerical stability monitoring during deep learning training and inference. We verified its effectiveness on the detection of polluted operators with injected numerical instabilities across various tasks. We believe that our developed method and tools provide valuable insights into developing numerically stable computing kernels, which are particularly critical for numerically stable and efficient deep learning training and inference. Our software is publicly available at \url{https://github.com/chenxinye/noisefloat}. 
\end{abstract}




% 向量的最大最小值digits
\maketitle

% the difficulty of the integrating CESTAC
% The instability is related to the value and formulation
% arbtrary precision emphasis
\section{Introduction}

Deep learning increasingly uses reduced-precision hardware to accelerate training and inference. Scientists and engineers employ reduced-precision arithmetic formats (partially presented in \tablename~\ref{tab:ftp}) defined by standards such as IEEE 754 \cite{ieee754-2019} and the emerging IEEE P3109 standard \cite{ieee_p3109_interim}. Compared with IEEE float64 format (FP64), these formats can significantly improve computational throughput while reducing energy consumption and memory footprint. However, reduced precision can increase numerical errors and expose numerical instabilities.

To quantify the numerical stability and reliability of numerical computations, one approach is rigorous error analysis. The round-off errors of fixed-point arithmetic can be modeled as white noise, independent of the input signal while the floating-point error depends strongly on the signal magnitude at runtime, which complicates error analysis.  Alternatively, tools and techniques such as Discrete Stochastic Arithmetic \citep{Vignes2004}, Monte Carlo Arithmetic \citep{parker1997mca}, CADNA \citep{Jez2008}, MCALIB \citep{frechtling2015mcalib}, and  Verificarlo \citep{denis2016verificarlo} support numerical validation at runtime. These approaches generally repeat a computation under randomized arithmetic or perturbations and therefore incur additional runtime and memory costs.

Deep learning has achieved remarkable success across a wide range of applications, including computer vision \cite{krizhevsky2012imagenet, he2016deep}, natural language processing~\cite{vaswani2017attention,devlin2019bert}, speech recognition~\cite{hinton2012deep,hannun2014deep}, and pattern recognition~\cite{lecun1998gradient,lecun2015deep}. Despite continued improvements in model accuracy, the reliability of these models depends not only on their robustness to input perturbations but also on their numerical stability throughout the computational process. With growing hardware support, mixed-precision computing is widely used in deep-learning training and inference \cite{micikevicius2018mixed, ferro2022neural, ioualalen2019precision}. Reduced- and mixed-precision formats reduce memory traffic and often improve throughput, but the resulting execution combines format conversions, operator-specific accumulation precisions, and fused backend kernels. Consequently, an observed loss of accuracy may depend on several interacting precision choices and can be difficult to attribute to a particular operation. Selecting suitable precisions therefore requires precision-tuning tools and studies \cite{10387857, Ben2022, GRAILLAT2019101017, 10820739, chen2026} or analyses tailored to the numerical kernels involved \cite{doi:10.1137/17M1122918, doi:10.1137/18M1226312, car2024kmeans, arar2025mixedprecisionaccumulationneural}. A survey of reduced-precision tools is provided in \cite{10.1145/3381039}. Rounding-error amplification, catastrophic cancellation, ill-conditioned reductions, unstable normalization, and sensitivity in attention or recurrent computations may arise inside a training or inference pipeline without causing an immediate program failure. Numerical validation is therefore needed to assess the accuracy of computed results and to support justified claims about the reliability of training and inference.

Deep-learning models compose many tensor operators whose input, accumulation, and output precisions may differ. Improper low-precision choices can produce unstable losses or catastrophic loss growth during training \cite{qiu2026lowprecision, wortsman2023stable, lee2024fp8back, liu2026grokking}. Localizing the responsible operator and selecting an appropriate analysis granularity remain difficult. Developers commonly mitigate such failures by evaluating sensitive intermediate expressions in a higher precision or by reformulating them. Automatic-differentiation frameworks and machine-learning libraries such as PyTorch \cite{paszke2019pytorch}, TensorFlow \cite{abadi2016tensorflow}, JAX \citep{jax2018github} and scikit-learn \cite{scikit-learn} use such internal precision conversions in selected computations, with pairwise Euclidean distance as one example \cite{car2024kmeans}. A comparative study reports that the tested TensorFlow operators exhibit more floating-point errors than their PyTorch counterparts when the implementations omit these conversions \cite{10.1145/3773992}.

Numerical stability of deep-learning operators has been studied through deterministic and probabilistic error analysis \cite{budzinskiy20, 10.1093/imanum/draf130, arar2025mixedprecisionaccumulationneural}, precision tuning \cite{ferro2022neural, ioualalen2019precision, lauter2020framework}, stochastic-arithmetic experiments \cite{faraone2019montecarlo, pepe2026fuzzy}, and fixed-point or quantization methods \cite{khalifa2024fixedpoint, lohar2023sound, gernigon2024adaqat}. These approaches establish important error bounds or evaluate numerical variability, but automated operator-level localization of the source of accuracy loss remains limited. Because stability can depend on both the algorithm and its data, exhaustive analysis may be costly in practice. Understanding how numerical errors propagate through deep networks and affect training is therefore important to both academia and industry.

\begin{table}[ht]
\centering\small
\setlength{\tabcolsep}{4pt}
\caption{Comparison of floating-point types. Here $p$ is the precision, including the implicit leading significand bit for normal numbers, and the unit roundoff is $u=2^{-p}$. % Floating-point numbers have at most $\left\lfloor p\log_{10}2 \right\rfloor$ exact decimal digits.
}
\label{tab:ftp}
\begin{tabular}{l c c c c c}
\toprule
Floating-Point Type
& Total Bits
& Exponent Bits ($e$)
& Significand Bits ($p$)
& Unit Roundoff
& Decimal Significant Digits \\
\midrule
BF16
& 16 & 8 & 8
& $2^{-8} \approx 3.91 \times 10^{-3}$
& $\approx 2.41$ \\

FP16
& 16 & 5 & 11
& $2^{-11} \approx 4.88 \times 10^{-4}$
& $\approx 3.31$ \\

TF32
& 19 & 8 & 11
& $2^{-11} \approx 4.88 \times 10^{-4}$
& $\approx 3.31$ \\

FP32
& 32 & 8 & 24
& $2^{-24} \approx 5.96 \times 10^{-8}$
& $\approx 7.22$ \\

FP64
& 64 & 11 & 53
& $2^{-53} \approx 1.11 \times 10^{-16}$
& $\approx 15.95$ \\
\bottomrule
\end{tabular}
\end{table}

Numerical validation plays an important role in scientific computing by assessing the reliability of specific finite-precision computations. CESTAC, introduced by Vignes and La Porte \citep{vignes1974error} and subsequently developed as stochastic arithmetic \citep{VIGNES1993233}, can identify instability sources such as severe cancellation and unstable branching when its stochastic samples are evaluated synchronously. Internally, CESTAC estimates the reliability of floating-point results from synchronous stochastic executions with directed rounding.  Existing implementations such as CADNA primarily expose overloaded scalar types for C, C++, and Fortran, but it does not natively instrument Python tensor operators, accelerator-resident tensors, or automatic-differentiation graphs. As an Alternative approach, Monte Carlo Arithmetic instead provides an asynchronous variability analysis across repeated executions \citep{parker1997mca}, but it does not directly attribute instability sources. The resulting diagnostics can guide algorithm reformulation and improve the reliability of the final computation.

In this paper, we present the framework \texttt{noisefloat} that integrates stochastic numerical validation with deep learning auto-differentiation frameworks. Internally, the framework examines the numerical stability through overloaded values, wrapped operators, and user-defined instrumentation boundaries. In detail, \texttt{noisefloat} implements \xyy{arithmetic-level instrumentation with CESTAC-style stochastic propagation on supported exposed operations.} The framework records significant-digit estimates by operator and iteration during training, validation, or inference on CPUs and GPUs. Users select the execution phase, instrumentation granularity, and whether the straight-through estimator is enabled. To the best of our knowledge, this is among the first frameworks to incorporate CESTAC-inspired stochastic numerical validation into deep-learning training and inference. This integration is nontrivial since \xyy{directly applying instruction-level CESTAC throughout deep-learning operators} would replicate stochastic samples through large tensor graphs{--}incurring substantial runtime and memory overhead, and introduce randomized quantization steps that are not differentiable in the usual automatic-differentiation sense. Therefore, the framework combines synchronized stochastic representatives with a straight-through estimator for gradient-based training, and treats GEMM operators through a data-perturbation approximation when their stochastic inputs are insufficiently separated, following the trusted-operator rationale of stochastic validation for dense linear algebra.

% Arithmetic-level instrumentation follows CESTAC on supported exposed operations, whereas operator-boundary instrumentation provides a coarser MCA-style approximation

% For example, users can analyze the numerical stability of an entire layer, a combination of a convolutional and pooling layer, CESTAC information such as significant digits, or the stability of a single matrix multiplication operation.

Our contributions are summarized as follows:
\begin{itemize}
    \item We develop a Python stochastic-validation tool that integrates with modern deep-learning training and inference. It batches multiple stochastic representatives within one instrumented forward evaluation and monitors selected operators over training and inference iterations. \xyy{To the best of our knowledge, this is the first work on CESTAC Python library and first framework that integrates CESTAC idea in Python deep learning deployments for numerical stochastic validation.}
    \item Our developed method and tools enable efficient automated detection and localization of numerical instabilities in deep learning operators, offering insights into the development of numerically stable computing kernels, which are particularly critical for efficient deep learning deployment.
    \item We provide an empirical study of deep learning operators and case studies on controlled pathological operators based on our developed framework. The results indicate how unstable operators affect training, in which situations their significant digits will fluctuate, and how.
\end{itemize}

Our work contributes to a framework for understanding numerical reliability of numerical computation in scientific computing and deep learning, enabling automated numerical stability detection and offering valuable insights into designing numerically stable and robust numerical algorithms and neural architectures.

% \paragraph{Paper organization.} 
The remainder of the paper is organized as follows: Section~\ref{sec:related_work} reviews prior work on numerical validation, stochastic arithmetic, and numerical reliability in deep learning. Section~\ref{sec:method} formulates our approach and software framework. Specifically, we describe our methodology and explain how we incorporate stochastic validation into numerical-stability detection for deep-learning operators. Section~\ref{sec:experiments} first validates the framework on stable and unstable numerical computations, then examines representative deep-learning operators, and finally evaluates the method on end-to-end deep-learning tasks. The final section concludes the paper and outlines future work.

\section{Related Work}\label{sec:related_work}
We review work on numerical validation and stability analysis in deep learning, including approaches closely related to the stochastic operator analysis proposed here. Although automated numerical validation at the level of deep-learning operators remains limited, numerical validation has a substantial history in scientific computing. We therefore focus on the studies and software most relevant to our formulation.

The probabilistic method  for numerical validation relies on repeated computation, either with different rounding modes or data perturbation. The CESTAC assumes that the digits common to the results obtained under all rounding modes are significant. DSA \citep{Vignes2004}, the subsequent work of CESTAC method, the notion of computational zero, and discrete stochastic relations. This framework enables the monitoring of scientific computations, the detection of numerical instabilities, and the verification of the underlying assumptions, thereby enhancing the robustness and reliability of numerical results.  Similarly to CESTAC, Monte Carlo Arithmetic (MCA) \citep{852391}, is capable of tracking the rounding error at runtime via the randomization. MCA collects a few trials of computations so that the numerical computations become a Monte Carlo simulation, where the statistics on the effect of rounding errors are obtained. The CADNA \citep{Jez2008} and Verificarlo \citep{DBLP:journals/corr/DenisCP15} serve as high-performance tools for the implementation of  DSA  and MCA.

In the area of deep learning, mixed-precision computing \citep{micikevicius2018mixed} is widely adopted and is receiving increasing attention in academia and industry. How to design numerically stable kernels and operators requires both theoretical and engineering insights. \cite{10.1145/3773992} systematically investigates precision casting in the implementation of deep learning operators. Rather than focusing on model architectures or mixed-precision training strategies, the paper examines how the deep learning  frameworks internally change the precision of selected floating-point computations to improve numerical stability. Fuzzy PyTorch \cite{pepe2026fuzzy} enables numerical variability in deep learnin models by integrating stochastic arithmetic into PyTorch through Probabilistic Rounding interfacing with Verificarlo \cite{DBLP:journals/corr/DenisCP15}.  However, it is limited to CPU implementation, which restricts its scalability. \cite{10.1145/3773992} investigates the precision cast inside the deep learning operators in the deep learning  framework of both Tensorflow and PyTorch.  The correctness of numerical computations often relies on the algorithm's inherent stability rather than careful error analysis.

The software framework proposed in this paper closes this integration gap by representing numerical values and tensors as stochastic samples and by lifting arithmetic and neural-network operators to act sample-wise on these samples.  It works as a runtime transformation of scalars, arrays, and tensor operators rather than by intercepting the hardware rounding mode.

%The framework in this paper closes the software-integration gap by treating numerical values and tensors as stochastic samples, then lifting arithmetic and neural-network operators to operate sample-wise on those samples. It works as a runtime transformation of scalars, arrays and tensor operators rather than by intercepting the hardware rounding mode.

\section{Automated Numerical Stability Analysis of Deep Learning Operators}
\label{sec:method}

This section formulates \texttt{noisefloat} as a runtime program transformation for numerical-stability analysis of deep-learning software. The object of analysis is not an isolated floating-point expression only, but a typed tensor program whose nodes include scalar arithmetic, array operations, and neural-network operators such as linear maps, convolutions, normalizations, activations, attention, recurrent layers, reductions, and losses.  The goal is to estimate, during training or inference, how many significant digits remain reliable at the outputs of selected operators and to expose these estimates through structured software diagnostics. Randomized directed rounding is applied at the instrumentation boundaries defined below. Consequently, the interpretation of the diagnostics depends on whether arithmetic-level or operator-boundary instrumentation is selected.

We model a numerical program as a typed computational graph $\mathcal{P}=(\mathcal{G},\Theta)$, where $\mathcal{G}=(\mathcal{V},\mathcal{E})$, $\mathcal{V}$ is the set of operation nodes, and $\mathcal{E}$ is the set of data-dependency edges. Each node $v\in\mathcal{V}$ denotes a scalar arithmetic operation, an array operation, or a neural-network operator, and $\Theta$ denotes immutable program state such as learned parameters and buffers. For every node $v$, $\mathcal{X}_v$ denotes the typed value space associated with its output. The node is associated with an operator
\begin{equation}
f_v: \mathcal{X}_{u_1}\times\cdots\times\mathcal{X}_{u_m}\times\Theta_v
\rightarrow \mathcal{X}_v,
\end{equation}
where $u_1,\ldots,u_m$ are predecessor nodes in the graph and $\Theta_v$ is the state used by node $v$. The deterministic execution computes
\begin{equation}
x_v=f_v(x_{u_1},\ldots,x_{u_m};\Theta_v).
\end{equation}

The stochastic program transformation implemented by \texttt{noisefloat} maps $\mathcal{P}$ to a transformed program
\begin{equation}
\mathcal{T}_{s,Q}(\mathcal{P}),
\end{equation}
where $s$ is the number of stochastic samples and $Q$ is a backend-native software quantizer. Each deterministic value $x_v$ is replaced by a stochastic state $X_v=(x_v^{(1)},\ldots,x_v^{(s)})$ and selected operators are lifted sample-wise. The transformed graph preserves the dependency structure of $\mathcal{G}$ but replaces each selected node $v$ by a lifted operator $\widehat{f}_v$, in other words, the structure of the original program is preserved while numerical-reliability reports is attached to chosen operators and iterations.

The design is governed by four principles. First, stochastic quantization must be backend-native so that NumPy arrays, PyTorch tensors, JAX arrays, and TensorFlow tensors remain in their original execution ecosystem. Second, the differentiable tensor path must coexist with automatic differentiation through a straight-through estimator. Third, diagnostics must be reported at the level of deep-learning operators, because this is the level at which developers decide whether to reformulate, replace, or precision-cast a component. Fourth, the same instrumentation must be usable during training, validation, and inference, so that numerical reliability can be tracked over iterations. Figure~\ref{fig:nfloat-method-overview} summarizes this formulation.

\begin{figure}[t]
\centering
\resizebox{\linewidth}{!}{%
\begin{tikzpicture}[
    >=Stealth,
    font=\sffamily\huge,
    stage/.style={rectangle, rounded corners=2pt, draw=black!55, fill=white, thick,
        minimum height=2.20cm, text width=3.35cm, align=center},
    mode/.style={rectangle, rounded corners=2pt, draw=blue!60!black, fill=blue!4,
        thick, minimum height=2.20cm, text width=3.80cm, align=center},
    report/.style={rectangle, rounded corners=2pt, draw=green!45!black, fill=green!5,
        thick, minimum height=2.20cm, text width=3.60cm, align=center},
    guard/.style={rectangle, rounded corners=2pt, draw=orange!70!black, fill=orange!7,
        thick, minimum height=2.20cm, text width=4.10cm, align=center},
    arr/.style={->, thick, draw=black!65},
    darr/.style={->, thick, dashed, draw=black!55},
    label/.style={font=\sffamily\LARGE, align=center, text=black!75}
]

\node[stage] (program) at (0, 0) {User numerical or neural-network program\\[-0.4mm]\Large scalar, tensor, operator graph};
\node[stage] (samples) at (8.0, 0) {Stochastic representation\\[-0.4mm]\Large $X=(x^{(1)},\ldots,x^{(s)})$};

\node[mode] (operator) at (14.3, 2.0) {Operator-boundary lifting\\[-0.4mm]\Large evaluate each sample through backend operator};
\node[mode] (arithmetic) at (14.3, -2.0) {Full/arithmetic CESTAC mode\\[-0.4mm]\Large propagate \texttt{NFloatTensor} through exposed primitives};

\node[guard] (gemm) at (20.7, -2.0) {GEMM-like trusted operator\\[-0.4mm]\Large use data perturbation if sample separation $<\eta$};
\node[stage] (quant) at (20.7, 2.0) {Backend-native stochastic quantization\\[-0.4mm]\Large $Q_r$ after selected operator output};

\node[stage] (rep) at (26.1, 1.35) {Representative value\\[-0.4mm]\Large $\overline{Y}=s^{-1}\sum_i y^{(i)}$};
\node[stage] (digits) at (26.1, -1.35) {Significant-digit field\\[-0.4mm]\Large $D_Y=C_{\overline{Y}}$};
\node[report] (reports) at (31.7, 0) {Structured diagnostics\\[-0.4mm]\Large \texttt{KernelReport}, iteration tracker, optional source counters};

\draw[arr] (program) -- (samples) node[midway, above=4pt, label, fill=white, inner sep=2pt] {\texttt{NFloat}/\texttt{NFloatTensor}};

\draw[arr] (samples.east) -- ++(0.7,0) |- (operator.west);
\draw[arr] (samples.east) -- ++(0.7,0) |- (arithmetic.west);
\draw[arr] (operator) -- (quant);
\draw[arr] (arithmetic) -- (gemm);
\draw[darr] (gemm.north) -- ++(0,1.2) -| (quant.south);
\draw[arr] (quant.east) -- ++(0.7,0) |- (rep.west);
\draw[arr] (quant.east) -- ++(0.7,0) |- (digits.west);
\draw[arr] (rep.east) -- (reports.west);
\draw[arr] (digits.east) -- (reports.west);

\node[label, above=0.36cm of operator] {efficient graph-aligned analysis};
\node[label, below=0.36cm of arithmetic] {finer PyTorch-visible arithmetic propagation};
\node[label, below=0.36cm of reports] {used over training, validation,\\and inference iterations};

\end{tikzpicture}%
}
\caption{Overview of the \texttt{noisefloat} framework. A deterministic scalar or tensor value is represented by a tuple of stochastic samples, and selected deep-learning operators are lifted sample-wise. The full arithmetic mode propagates stochastic tensors through supported PyTorch-visible primitives. A ``trusted GEMM-like operator'' is a backend matrix multiplication, linear map, or convolution whose internal multiply-add operations are not instrumented. The method instead uses the input-perturbation safeguard defined in Eq.~\eqref{eq:gemm_data_perturbation}. Reports record operator-level significant digits and optional instability-source diagnostics.}
\label{fig:nfloat-method-overview}
\end{figure}

% The operator-boundary mode quantizes after each instrumented operator output and is an MCA-style approximation rather than instruction-level CESTAC. 

% \subsection{Program Transformation Model}
Thus, the instrumentation boundary is explicit. Overloaded scalar values apply randomized directed rounding after each supported Python-visible arithmetic operation. Deep-learning wrappers, by contrast, may either round only the output of a selected neural-network operator or use the arithmetic-level path to propagate stochastic values through supported PyTorch-visible primitives. Operations fused inside a vendor kernel remain outside instruction-level instrumentation. Therefore, the framework should be interpreted as CESTAC-inspired rather than as a complete instruction-level implementation of CESTAC: it follows CESTAC-style synchronous stochastic propagation for supported exposed operations, while GEMM-like backend kernels are handled as trusted operators through the operator-level data-perturbation approximation described below.

% This formulation is useful for software engineering because it separates three concerns. First, $f_v$ is the mathematical or backend operation already present in the user program. Second, $Q$ specifies the stochastic precision model used for validation. Third, the reporting map attached to $v$ records numerical reliability without changing the graph-level semantics of the original model except for stochastic quantization effects.

\subsection{Stochastic Representation and Operator Lifting}

The Control and Stochastic Estimation of Rounding Errors (CESTAC) method~\citep{vignes1974error,VIGNES1993233} estimates the effect of rounding errors by executing a computation over several synchronous stochastic realizations. Let $x$ denote the deterministic value represented by a floating-point scalar, vector, matrix, or tensor. The framework associates with $x$ a stochastic representation
\begin{equation}
X = \left(x^{(1)},x^{(2)},\ldots,x^{(s)}\right),
\label{eq:stochastic_representation}
\end{equation}
where $s$ is the number of stochastic samples. Following the standard CADNA setting, the default value is $s=3$, although the implementation exposes this as a configuration parameter.

Let $z$ be an exact intermediate result, and let $z^{-}$ and $z^{+}$ be the adjacent representable values of the simulated format that bracket it. We write $Q_r(z)$ for software quantization with randomized directed rounding: each stochastic realization evaluates an arithmetic operation under a randomly selected upward or downward directed rounding mode. This CESTAC rule is distinct from conventional stochastic rounding, which selects $z^{-}$ and $z^{+}$ with probabilities proportional to their relative distances from $z$. The implementation emulates the directed-rounding outcomes with backend tensor operations rather than repeatedly changing the hardware rounding mode.

For an $m$-ary arithmetic operation or numerical operator $f$, the lifted stochastic operator is defined sample-wise as
\begin{equation}
\widehat{f}(X_1,\ldots,X_m)
=
\left(
Q_r\!\left(f\left(x_1^{(i)},\ldots,x_m^{(i)}\right)\right)
\right)_{i=1}^{s}.
\label{eq:lifted_operator}
\end{equation}
The synchronous evaluation first computes the $i$th representative of every operand and the corresponding output and then proceeds to representative $i+1$. This preserves sample correspondence across the graph and lets output dispersion reflect the rounding sensitivity of the executed computation. For scalar arithmetic, the lifting applies to supported Python-visible operations such as addition, subtraction, multiplication, division, power, and selected mathematical functions, with rounding after each exposed arithmetic operation. For tensor computations, two interpretations are distinguished. Arithmetic-level instrumentation applies the rule after each supported PyTorch-visible primitive. Operator-boundary instrumentation evaluates an entire backend operator independently for each stochastic sample and quantizes only its output. It is therefore an operator-level approximation and does not support instruction-level CESTAC claims or counts for arithmetic hidden inside the backend kernel.

For a stochastic output $Y=(y^{(1)},\ldots,y^{(s)})$, the representative value is the sample mean
\begin{equation}
\overline{Y}=\frac{1}{s}\sum_{i=1}^{s}y^{(i)}.
\label{eq:representative_value}
\end{equation}
The sample spread estimates the numerical uncertainty. Let
\begin{equation}
\sigma_Y=
\sqrt{\frac{1}{s-1}\sum_{i=1}^{s}\left(y^{(i)}-\overline{Y}\right)^2}
\end{equation}
be the sample standard deviation. The estimated number of significant decimal digits is given by
\begin{equation}\label{eq:significant_digit}
C_{\overline{Y}}
=
\log_{10}
\left(
\frac{\sqrt{s}\,|\overline{Y}|}{\tau_{\beta}\sigma_Y}
\right),
\end{equation}
where $\tau_{\beta}$ is the Student $t$ critical value with $s-1$ degrees of freedom at the configured confidence level. For tensor outputs, this equation is applied elementwise, yielding a digit field
\begin{equation}
D_Y = C_{\overline{Y}} \in \mathbb{R}^{\operatorname{shape}(Y)}.
\end{equation}

% Our software records the aggregate summaries $\operatorname{mean}(D_Y)$, $\operatorname{median}(D_Y)$, $\operatorname{min}(D_Y)$, and $\operatorname{max}(D_Y)$ for operator ranking and iteration-level plots. In practice, we rely more on  $\operatorname{mean}(D_Y)$ and $\operatorname{min}(D_Y)$. Therefore, in the experiment section, our analysis focuses on the two metrics.

\begin{figure}[htp]
\centering
\begin{pythoncode}[label={lst:scalar-diagnostics}]{Scalar stochastic arithmetic and diagnostics}
from noisefloat import NFloat, configure
from noisefloat import clear_diagnostics, get_diagnostics_summary

configure(
    backend="numpy",
    exp_bits=11,
    sig_bits=52,
    n_samples=3,
    random_state=42,
    diagnostics_level="summary",
)
clear_diagnostics()

x = NFloat(10864.0)
y = NFloat(18817.0)
p = 9.0 * x**4 - y**4 + 2.0 * y**2

print(p.value)                  # representative sample mean
print(p.digits)                 # estimated significant digits
print(get_diagnostics_summary())
\end{pythoncode}
\caption{Minimal scalar \texttt{NFloat} example for stochastic arithmetic and diagnostic source counting.}\label{fig:scalar-diagnostics}
\end{figure}

The implementation follows robust conventions for exceptional cases: identical samples are treated as fully stable up to the implementation cap on reported digits. A zero mean with nonzero spread is treated as numerical noise. Samples containing NaN or infinity are reported with zero significant digits. These conventions support automated diagnostics and avoid interpreting nonfinite or purely noisy quantities as reliable numerical results.

The scalar interface in \figurename~\ref{fig:scalar-diagnostics} is the minimal executable form of this transformation. The user configures the simulated format, replaces ordinary floating-point inputs by \texttt{NFloat} values, and obtains both a representative value and a CESTAC-style digit estimate. When diagnostics are enabled, the same execution also updates aggregate counters for numerical instabilities.

For arithmetic-level instrumentation, the method also records source-attribution counters patterned after CADNA. Each counter is incremented when an exposed operation satisfies the corresponding instability test: loss of accuracy due to cancellation, unstable branching, unstable division, unstable multiplication, unstable power, intrinsic-function instability, or a mathematical anomaly such as NaN or infinity. These counters attribute events only at operations intercepted by the software and do not describe arithmetic hidden inside backend kernels. NumPy and PyTorch examples and their outputs are shown in Figures~\ref{fig:hilbert-determinant-demo-np} and~\ref{fig:hilbert-determinant-demo-th}. % Accordingly, we observe that \texttt{noisefloat} with PyTorch backend achieves lower runtime that NumPy.

\begin{figure*}[htp]
\centering
\begin{minipage}[t]{0.485\textwidth}
\vspace{0pt}
\begin{pythoncode}{Deterministic rounding}
def hilbert_determinant() -> np.float32:
    n = 11
    a = [
        [np.float32(1.0) / np.float32(i + j + 1.0)
         for j in range(n)]
        for i in range(n)
    ]
    det = np.float32(1.0)

    for i in range(n - 1):
        det = det * a[i][i]
        aux = np.float32(1.0) / a[i][i]
        for j in range(i + 1, n):
            a[i][j] = a[i][j] * aux
        for j in range(i + 1, n):
            aux = a[j][i]
            for k in range(i + 1, n):
                a[j][k] = (
                    a[j][k]
                    - aux * a[i][k]
                )

    det = det * a[n - 1][n - 1]
    return det
\end{pythoncode}
\end{minipage}
\hfill
\begin{minipage}[t]{0.485\textwidth}
\vspace{0pt}
\begin{pythoncode}{NFloat stochastic equivalents}
def hilbert_determinant_nfloat():
    n = 11
    a = [
        [NFloat(1.0 / (i + j + 1.0))
         for j in range(n)]
        for i in range(n)
    ]
    det = NFloat(1.0)

    for i in range(n - 1):
        det = det * a[i][i]
        aux = 1.0 / a[i][i]
        for j in range(i + 1, n):
            a[i][j] = a[i][j] * aux
        for j in range(i + 1, n):
            aux = a[j][i]
            for k in range(i + 1, n):
                a[j][k] = (
                    a[j][k]
                    - aux * a[i][k]
                )

    det = det * a[n - 1][n - 1]
    return det
\end{pythoncode}
\end{minipage}

\begin{minipage}[t]{0.60\linewidth}
\vspace{0.8em}
\begin{codeoutput}{Single-precision stochastic output}
Pivot number  1 = 9.999999602636E-01  digits=6.466
Pivot number  2 = 8.333335071802E-02  digits=5.930
Pivot number  3 = 5.555580214908E-03  digits=4.649
Pivot number  4 = 3.571556784057E-04  digits=3.696
Pivot number  5 = 2.266229845797E-05  digits=2.241
Pivot number  6 = 1.380530610125E-06  digits=0.554
Pivot number  7 = @.0
Pivot number  8 = @.0
Pivot number  9 = @.0
Pivot number 10 = @.0
Pivot number 11 = @.0
Determinant = @.0
\end{codeoutput}
\end{minipage}
\hfill
\begin{minipage}[t]{0.36\linewidth}
\vspace{0.8em}
\begin{codeoutput}{Timing comparison}
Example: Hilbert determinant
Precision: single
Repeats: 3

Deterministic mean: 1.856873e-04 s
NFloat mean : 2.087731e-01 s
Overhead : 1124.326x

Detected sources:
4 unstable divisions
44 unstable multiplications
\end{codeoutput}
\end{minipage}
\caption{Hilbert determinant diagnostic under stochastic single precision. The two code fragments show the same Gaussian-elimination recurrence on the Hilbert matrix: the left fragment uses NumPy single-precision scalar arithmetic with deterministic round-to-nearest, whereas the right fragment replaces scalar values by \texttt{NFloat} stochastic variables. Global configuration and output-formatting helpers are omitted from the code fragment. The excerpted output is taken from the verification script run with three FP32 stochastic samples. The stochastic execution exposes the progressive loss of significant digits in the pivots and reports the determinant as numerically unreliable. The timing block reports mean runtime over three repeats and compares the single-precision deterministic baseline with the instrumented \texttt{NFloat} execution.}
\label{fig:hilbert-determinant-demo-np}
\end{figure*}

\begin{figure*}[htp]
\centering
\begin{minipage}[t]{0.485\textwidth}
\vspace{0pt}
\begin{pythoncode}{Deterministic rounding}
def tdet(value: float | int) -> torch.Tensor:
    return torch.tensor(value, 
        dtype=torch.float32)

def hilbert_determinant_th() -> torch.float32:
    n = 11
    a = [[tdet(1.0 / (i + j + 1.0)) for j in range(n)] for i in range(n)]
    det = tdet(1.0)
    for i in range(n - 1):
        det = det * a[i][i]
        aux = 1.0 / a[i][i]
        for j in range(i + 1, n):
            a[i][j] = a[i][j] * aux
        for j in range(i + 1, n):
            aux = a[j][i]
            for k in range(i + 1, n):
                a[j][k] = a[j][k] - aux * a[i][k]
               
    det = det * a[n - 1][n - 1]
    return det
\end{pythoncode}
\end{minipage}
\hfill
\begin{minipage}[t]{0.485\textwidth}
\vspace{0pt}
\begin{pythoncode}{NFloat stochastic variables}
def nf(value: float | int) -> NFloat:
    return NFloat(torch.tensor(value, dtype=torch.float32))
    
def hilbert_determinant_th_nfloat() -> NFloat:
    n = 11
    a = [[nf(1.0 / (i + j + 1.0)) for j in range(n)] for i in range(n)]
    det = nf(1.0)
    for i in range(n - 1):
        det = det * a[i][i]
        aux = 1.0 / a[i][i]
        for j in range(i + 1, n):
            a[i][j] = a[i][j] * aux
        for j in range(i + 1, n):
            aux = a[j][i]
            for k in range(i + 1, n):
                a[j][k] = a[j][k] - aux * a[i][k]
               
    det = det * a[n - 1][n - 1]
    return det
\end{pythoncode}
\end{minipage}

\begin{minipage}[t]{0.60\linewidth}
\vspace{0.8em}
\begin{codeoutput}{Single-precision stochastic output}
Pivot number  1 = 1.000000039736E+00  digits=6.466
Pivot number  2 = 8.333325137695E-02  digits=5.611
Pivot number  3 = 5.555567486833E-03  digits=4.708
Pivot number  4 = 3.572066101090E-04  digits=3.656
Pivot number  5 = 2.269889409945E-05  digits=2.008
Pivot number  6 = 1.311535053598E-06  digits=0.732
Pivot number  7 = @.0
Pivot number  8 = @.0
Pivot number  9 = @.0
Pivot number 10 = @.0
Pivot number 11 = @.0
Determinant = @.0
\end{codeoutput}
\end{minipage}
\hfill
\begin{minipage}[t]{0.36\linewidth}
\vspace{0.8em}
\begin{codeoutput}{Timing comparison}
Example: Hilbert determinant
Precision: single
Repeats: 3

Deterministic mean: 7.840321e-03 s
NFloat mean : 1.000234e+00 s
Overhead : 127.576x

Detected sources:
4 unstable divisions
44 unstable multiplications
\end{codeoutput}
\end{minipage}
\caption{Hilbert determinant diagnostic under stochastic single precision. The two code fragments show the same Gaussian-elimination recurrence on the Hilbert matrix: the left fragment uses PyTorch single-precision scalar arithmetic with deterministic round-to-nearest, whereas the right fragment replaces scalar values by \texttt{NFloat} stochastic variables backended by PyTorch. Global configuration and output-formatting helpers are omitted from the code fragment. The excerpted output is taken from the verification script run with three FP32 stochastic samples. The stochastic execution exposes the progressive loss of significant digits in the pivots and reports the determinant as numerically unreliable. The timing block reports mean runtime over three repeats and compares the single-precision deterministic baseline with the instrumented \texttt{NFloat} execution.}
\label{fig:hilbert-determinant-demo-th}
\end{figure*}

\subsection{Backend-Native Stochastic Quantization}

The quantizer in \texttt{noisefloat} is parameterized by the backend, exponent bits, significand bits, number of stochastic samples, random seed, confidence level, diagnostics configuration, and reporting detail. The public configuration object exposes these parameters through \texttt{configure} and \texttt{get\_config}, with relevant fields including \texttt{backend}, \texttt{exp\_bits}, \texttt{sig\_bits}, \texttt{n\_samples}, \texttt{random\_state}, \texttt{confidence}, \texttt{trace}, \texttt{digits\_threshold}, \texttt{zero\_digits\_threshold}, \texttt{diagnostics\_level}, \texttt{kernel\_report\_detail}, and thresholds used for cancellation and data perturbation.

Several of these parameters are also runtime-overhead controls. Setting \texttt{diagnostics\_level="off"} disables CADNA-style source attribution counters, whereas \texttt{"summary"} retains only aggregate source-of-instability counts, and \texttt{"detailed"}, or equivalently \texttt{trace=True}, stores per-operation diagnostic events and is therefore more expensive. Similarly, \texttt{kernel\_report\_detail="summary"} stores only scalar report fields and tensor shapes, whereas \texttt{"full"} stores sample tensors, representative values, and digit arrays inside each report. The verification examples use \texttt{diagnostics\_level="summary"} because their purpose is to compare source-of-instability summaries with CADNA C. By contrast, deep-learning training experiments primarily analyze operator significant digits over iterations, so they are intended to run with \texttt{trace=False}, \texttt{diagnostics\_level="off"} unless source attribution is explicitly needed, and \texttt{kernel\_report\_detail="summary"}. This reduces logging and memory traffic, although the dominant computational overhead remains the stochastic sample count $s$, the selected instrumentation granularity, and any arithmetic-level propagation inside neural-network operators.

Unlike CADNA implementations that change the floating-point unit rounding mode, \texttt{noisefloat} emulates quantization elementwise with backend array operations. For each array element, the software determines the two representable endpoints of the configured format, selects the endpoint associated with the sampled directed-rounding mode, and reconstructs an array of the original backend type. The available abstractions are \texttt{NumpyBackend}, \texttt{TorchBackend}, \texttt{JaxBackend}, and \texttt{TensorFlowBackend}, accessed through \texttt{BackendOps} and \texttt{get\_backend}. NumPy arrays, PyTorch tensors, JAX arrays, and TensorFlow tensors therefore remain in their respective ecosystems during quantization. The PyTorch path uses device-specific random generators, allowing randomized rounding on CPU or CUDA tensors without changing the global hardware rounding mode. Public statistics are exposed in NumPy-compatible form where appropriate, while neural-network tensor representatives remain backend tensors.

\begin{figure}[t]
\centering
\begin{pythoncode}[label={lst:configuration}]{Configuration interface}
from noisefloat import configure, get_config

configure(
    backend="torch",          # "numpy" | "torch" | "jax" | "tensorflow"
    exp_bits=8,               # exponent bits of the simulated format
    sig_bits=23,              # significand bits of the simulated format
    n_samples=3,              # stochastic samples per logical value
    random_state=42,          # reproducible stochastic rounding
    confidence=0.95,          # Student-t confidence level
    diagnostics_level="off", # enable "summary"/"detailed" for source attribution
    kernel_report_detail="summary",
)

cfg = get_config(); print(cfg.n_samples)
\end{pythoncode}
\caption{Configuration example for selecting backend-native stochastic quantization parameters.}
\end{figure}

Mathematical functions such as \texttt{sqrt}, \texttt{exp}, \texttt{log}, \texttt{where}, \texttt{dot}, \texttt{matmul}, \texttt{sum}, \texttt{mean}, and \texttt{norm} dispatch through the backend abstraction. User-defined numerical operators can be lifted through the same sample-wise principle by applying a function independently to each stochastic sample and quantizing the resulting sample batch.

\subsection{Automatic Differentiation and Straight-Through Estimation}

Gradient-based training requires stochastic quantization to coexist with automatic differentiation. The differentiable variant, \texttt{NFloatSTE}, uses the same stochastic quantization in the forward pass,
\begin{equation}
Y = Q_r(X),
\end{equation}
but approximates the derivative of quantization by the identity map in the backward pass,
\begin{equation}
\frac{\partial Q_r(x)}{\partial x}\approx I.
\end{equation}
Equivalently, for a scalar loss $\mathcal{L}$, the backward pass uses
\begin{equation}
\frac{\partial \mathcal{L}}{\partial x}
\approx
\frac{\partial \mathcal{L}}{\partial Q_r(x)}.
\end{equation}
The identity straight-through estimator is a biased surrogate rather than the derivative of quantization, but it is widely used to train networks containing discrete or quantized operations because it preserves a usable gradient path \cite{bengio2013estimating}. Here it is used only as a compatibility mechanism for an instrumented forward pass, but it does not certify the correctness of the resulting gradient. The implementation provides backend-specific paths where the framework supports them: PyTorch uses \texttt{torch.autograd.Function}, JAX uses custom vector-Jacobian products, and TensorFlow uses \texttt{tf.custom\_gradient}. NumPy supports stochastic arithmetic and diagnostics but has no native automatic-differentiation path.

The reported significant digits describe stochastic forward values at selected operators, whereas gradients are propagated through the host autodiff system's surrogate. In the experiments reported here, parameter optimization and task-loss computation remain on the original deterministic path, while a synchronized shadow model performs the stochastic diagnostic pass, as specified in Eq.~\eqref{eq:shadow-evaluation}. Thus, the loss algorithm that determines the optimizer update is unchanged.

\subsection{Neural-Network Operator Instrumentation}

The deep-learning interface treats layers and functions as graph-level operators. Instrumentation is enabled by replacing a layer with its \texttt{NFloat}-prefixed wrapper or decorating a function with \texttt{nfloat\_operator}. These wrappers are available for linear layers, convolutions, RNN/GRU/LSTM modules, activations, normalizations, pooling, softmax and log-softmax, embeddings, losses, multi-head attention, transformer layers, matrix multiplication, reductions, and user-defined operators. A wrapper does not imply that every machine instruction inside the backend kernel is randomly rounded. Operator-boundary wrappers quantize the operator output, whereas arithmetic-level wrappers intercept only the supported PyTorch-visible primitives described below. We use \emph{operator} for a mathematical or graph-level computation and \emph{backend kernel} for its low-level implementation. \texttt{KernelReport} and its \texttt{kernel\_name} field retain ``kernel'' only for compatibility with the software API.

For an instrumented operator node $v$ and execution phase $p\in\{\text{training},\text{validation},\text{inference}\}$, the transformed execution computes
\begin{equation}
X_v=\widehat{f}_v(X_{u_1},\ldots,X_{u_m};\Theta_v),
\qquad
\overline{x}_v=\operatorname{mean}_s(X_v),
\qquad
D_v=C_{\overline{x}_v}.
\end{equation}
A structured report can be viewed as a map
\begin{equation}
\mathcal{R}(v,p,t)=
\left(
\operatorname{name}(v),p,t,
\operatorname{avg}(D_v),
\operatorname{min}(D_v),
\operatorname{max}(D_v),
\operatorname{stable}(D_v),
\operatorname{meta}(v)
\right),
\end{equation}
where $t$ denotes the current iteration and $\operatorname{meta}(v)$ records backend and shape metadata. Given a configured digit threshold $d_{\min}$, $\operatorname{stable}(D_v)$ is true precisely when all reported digit values are finite, no exceptional value is present, and the selected aggregate digit statistic is at least $d_{\min}$. This report abstraction connects the stochastic computation to the engineering task of localizing unstable operators.

The standard PyTorch interface is built around \texttt{NFloatTensor}, \texttt{NFloatModule}, and layer wrappers such as \texttt{NFloatLinear}, \texttt{NFloatConv2d}, \texttt{NFloatLayerNorm}, \texttt{NFloatSoftmax}, \texttt{NFloatRNN}, \texttt{NFloatGRU}, \texttt{NFloatLSTM}, and \texttt{NFloatMultiheadAttention}. Each instrumented operator receives stochastic tensor inputs, evaluates each sample through the corresponding backend operator, stacks the outputs, applies stochastic quantization at the operator boundary, and returns an \texttt{NFloatTensor}. The representative tensor is the sample mean, exposed as \texttt{value}, while elementwise and aggregate significant digits are available through \texttt{digits}, \texttt{avg\_digits}, and \texttt{min\_digits}.

The arithmetic-level path uses \texttt{ArithmeticNFloatModule} and wrappers such as \texttt{ArithmeticNFloatLinear}, \texttt{ArithmeticNFloatLayerNorm}, \texttt{ArithmeticNFloatSoftmax}, \texttt{ArithmeticNFloatGRU}, \texttt{ArithmeticNFloatLSTM}, and \texttt{ArithmeticNFloatMultiheadAttention}. In this mode, \texttt{NFloatTensor} values propagate through supported Python-visible PyTorch primitives, so rounding is applied after exposed arithmetic operations, reductions, normalizations, recurrent operations, attention operations, and activations. The guarantee is limited to those intercepted primitives, whereas arithmetic fused within a backend kernel is not overloaded. GEMM-like backend kernels are treated as trusted operators and use the data-perturbation safeguard below when their stochastic operands are insufficiently separated. Thus, operator-boundary analysis provides graph-aligned localization, whereas arithmetic-level analysis provides finer but still Python-visible instrumentation.

More precisely, let $G(A,B)=AB$ denote a GEMM-like operator such as a matrix multiplication, linear layer, or convolution lowered to matrix multiplication. In the arithmetic-level mode, stochastic operands $A=(A^{(1)},\ldots,A^{(s)})$ and $B=(B^{(1)},\ldots,B^{(s)})$ are first tested for sample separation. For each tensor entry $j$, define
\begin{equation}
\rho_j(A)=
\frac{\operatorname{std}_{i}\!
\left(A^{(i)}_j\right)}
{\max\!\left(|\overline{A}_j|,
\frac{1}{s}\sum_{i=1}^{s}|A^{(i)}_j|,\epsilon\right)},
\label{eq:gemm_sample_separation}
\end{equation}
with an analogous definition for $\rho_j(B)$. Here $\epsilon$ is a positive safeguard that prevents division by zero when both the sample mean and mean magnitude vanish. Given the configured threshold $\eta$---currently $\eta=2^{-p}$ by default for a simulated significand precision $p$---entries satisfying $\rho_j<\eta$ are considered insufficiently separated. For those entries, \texttt{noisefloat} applies a controlled perturbation
\begin{equation}
\widetilde{A}^{(i)}_j =
Q_r\!\left(A^{(i)}_j + \xi^{(i)}_j\,\eta\,
\max\!\left(|\overline{A}_j|,
\frac{1}{s}\sum_{k=1}^{s}|A^{(k)}_j|,\epsilon\right)\right),
\qquad
\xi^{(i)}_j\sim\mathcal{U}[-1,1],
\label{eq:gemm_data_perturbation}
\end{equation}
and leaves already separated entries unchanged. The GEMM-like operator is then evaluated sample-wise as
\begin{equation}
\widehat{G}(A,B)=
\left(Q_r\!\left(G(\widetilde{A}^{(i)},\widetilde{B}^{(i)})\right)\right)_{i=1}^{s}.
\label{eq:gemm_lifted_perturbed}
\end{equation}
This formulation preserves CESTAC propagation through the surrounding graph while treating the GEMM computation itself as a trusted backend operator under a data-perturbation safeguard. Deterministic parameters, such as ordinary learned weights, are broadcast across samples without perturbation, but users who require stochastic weights can pass them explicitly as \texttt{NFloatTensor} values.

\paragraph{Approximation induced by GEMM data perturbation. The GEMM safeguard follows data-perturbed numerical validation for dense linear algebra: a highly optimized matrix multiplication is treated as a trusted backward-stable operator, and variability is supplied through perturbed inputs rather than by replacing each internal multiply-add \citep{jezequel2020trustworthy, jezequel2024innerproducts}. This is not instruction-level CESTAC inside GEMM. The validity and perturbation amplitude required for probabilistic inner-product validation are analyzed by Jézéquel and Mary \cite{jezequel2024innerproducts}, and we rely on that analysis rather than restating its results as new lemmas. Consequently, GEMM reports must be interpreted as operator-level, data-perturbation estimates whose reliability depends on the hypotheses of the cited analysis.} Following \cite[Chp. 4]{higham2002accuracy}, we have Lemma~\ref{lem:gemm-perturbation-forward}:
\begin{lemma}
\label{lem:gemm-perturbation-forward}
Let $G(A,B)=AB$ and let $\widetilde{A}=A+\Delta A$ and $\widetilde{B}=B+\Delta B$. For any submultiplicative norm,
\begin{equation}
\left\|\widetilde{A}\widetilde{B}-AB\right\|
\le
\|A\|\,\|\Delta B\|
+
\|\Delta A\|\,\|B\|
+
\|\Delta A\|\,\|\Delta B\|.
\label{eq:gemm_perturbation_forward_bound}
\end{equation}
If $\|\Delta A\|\le \eta \chi_A\|A\|$ and $\|\Delta B\|\le \eta \chi_B\|B\|$, then, whenever $AB\ne0$,
\begin{equation}
\frac{\left\|\widetilde{A}\widetilde{B}-AB\right\|}{\|AB\|}
\le
\mu(A,B)\,
\eta\left(\chi_A+\chi_B+\eta\chi_A\chi_B\right),
\qquad
\mu(A,B)=\frac{\|A\|\,\|B\|}{\|AB\|}.
\label{eq:gemm_relative_perturbation_bound}
\end{equation}
\end{lemma}

\begin{proof}
Expanding the perturbed product gives
\[
\widetilde{A}\widetilde{B}-AB
=A\Delta B+\Delta A B+\Delta A\Delta B.
\]
The triangle inequality and submultiplicativity give Eq.~\eqref{eq:gemm_perturbation_forward_bound}. Substituting the relative perturbation bounds for $\Delta A$ and $\Delta B$ and dividing by $\|AB\|$ gives Eq.~\eqref{eq:gemm_relative_perturbation_bound}.
\end{proof}

Lemma~\ref{lem:gemm-perturbation-forward} gives a local interpretation of \texttt{data\_perturbation\_threshold}. With the default $\eta=2^{-p}$ for simulated significand precision $p$, the additional input perturbation is at the configured significand scale. Its output effect is small when the matrix product is well conditioned in the sense that $\mu(A,B)$ is moderate, and it can be amplified when the product itself is ill conditioned or cancellation dominated. This is the desired behavior for diagnostics: the perturbation should remain negligible for stable GEMM-like operators and become visible when the product is sensitive to small input changes.

The following bound is therefore used only to interpret the sensitivity of the CESTAC significant-digit estimator to the additive perturbation induced by the surrogate. It is not a replacement for the probabilistic validation theory of the underlying GEMM-like operator. The effect on the significant-digit estimator is given in Theorem~\ref{thm:gemm_digit_bias}

\begin{theorem}
\label{thm:gemm_digit_bias}
Consider one scalar coordinate of an ideal stochastic GEMM output $Z=(z^{(1)},\ldots,z^{(s)})$ and the corresponding data-perturbed surrogate $\widetilde{Z}=Z+E$, where $E=(e^{(1)},\ldots,e^{(s)})$. Let $\overline{z}$ and $\sigma_Z$ denote the sample mean and sample standard deviation of $Z$, and let $\overline{e}$ and $\sigma_E$ be the corresponding quantities for $E$. If $|\overline{e}|<|\overline{z}|$ and $\sigma_E<\sigma_Z$, then the absolute change in the CESTAC digit estimate satisfies
\begin{equation}
\left|
C_{\overline{\widetilde{Z}}}
-
C_{\overline{Z}}
\right|
\le
\log_{10}
\frac{|\overline{z}|+|\overline{e}|}
     {|\overline{z}|-|\overline{e}|}
+
\log_{10}
\frac{\sigma_Z+\sigma_E}
     {\sigma_Z-\sigma_E}.
\label{eq:digit_bias_bound}
\end{equation}
\end{theorem}

\begin{proof}
The estimator in Eq.~\eqref{eq:significant_digit} differs between $Z$ and $\widetilde{Z}$ only through the absolute sample mean and the sample standard deviation, because the factors $\sqrt{s}$ and $\tau_\beta$ cancel. The mean satisfies $|\overline{z}+\overline{e}|\le|\overline{z}|+|\overline{e}|$ and $|\overline{z}+\overline{e}|\ge|\overline{z}|-|\overline{e}|$. The sample standard deviation is the Euclidean norm of the centered sample vector divided by $\sqrt{s-1}$, hence it is Lipschitz with respect to sample perturbations: $|\sigma_{\widetilde{Z}}-\sigma_Z|\le\sigma_E$. Therefore $\sigma_Z-\sigma_E\le\sigma_{\widetilde{Z}}\le\sigma_Z+\sigma_E$. Combining these upper and lower multiplicative bounds inside the base-10 logarithm yields Eq.~\eqref{eq:digit_bias_bound}.
\end{proof}

Theorem~\ref{thm:gemm_digit_bias} clarifies when the GEMM surrogate has a small effect on the digit estimate that would be obtained from a finer stochastic treatment of the same scalar coordinate. If the perturbation contribution $E$ is small relative to both the representative signal $|\overline{z}|$ and the stochastic spread $\sigma_Z$, the digit estimate changes by only the two logarithmic terms in Eq.~\eqref{eq:digit_bias_bound}. If $\sigma_Z$ is nearly zero, the bound correctly becomes uninformative: the original samples have collapsed and the software intentionally injects significand-scale data perturbations to avoid reporting a spuriously deterministic GEMM output.

The preceding theorem is local to one GEMM-like output coordinate. In a neural-network graph, however, several individually backward-stable operators may be composed with nonlinear activations, normalizations, softmax, attention, or recurrent updates. The composition need not be backward stable as a whole: small perturbations introduced at one trusted GEMM-like operator can be amplified by later operators when the downstream map is ill conditioned or close to a nonsmooth regime. The following result makes this approximation gap explicit for the software strategy used in \texttt{noisefloat}.

We further investigate composition-level perturbation gap, and have Theorem~\ref{thm:composition_perturbation_gap}. 
\begin{theorem}[]
\label{thm:composition_perturbation_gap}
Consider one scalar output coordinate of an ideal arithmetic-level stochastic execution
\[
Y=(y^{(1)},\ldots,y^{(s)})
\]
and the corresponding execution
\[
\widetilde{Y}=(\widetilde{y}^{(1)},\ldots,\widetilde{y}^{(s)})
\]
in which a set $\mathcal{P}$ of GEMM-like operators is evaluated through the data-perturbed surrogate in Eq.~\eqref{eq:gemm_lifted_perturbed}. For sample $i$, let $\Delta_p^{(i)}$ be the perturbation at the output of the trusted operator $p\in\mathcal{P}$ induced by the input perturbations in Eq.~\eqref{eq:gemm_data_perturbation}. Suppose that, along the segment joining the ideal and surrogate trajectories, the downstream subgraph from the output of $p$ to the selected scalar coordinate is locally $L_p$-Lipschitz. Define
\begin{equation}
b_i =
\sum_{p\in\mathcal{P}} L_p\,\|\Delta_p^{(i)}\|,
\qquad
\alpha=\frac{1}{s}\sum_{i=1}^{s}b_i,
\qquad
\gamma=
\left(\frac{1}{s-1}\sum_{i=1}^{s}b_i^2\right)^{1/2}.
\label{eq:composition_gap_terms}
\end{equation}
Then the representative values satisfy
\begin{equation}
\left|
\overline{\widetilde{Y}}-\overline{Y}
\right|
\le \alpha.
\label{eq:representative_gap_bound}
\end{equation}
Moreover, if $\alpha<|\overline{Y}|$ and $\gamma<\sigma_Y$, where $\sigma_Y$ is the sample standard deviation of $Y$, then the significant-digit estimates satisfy
\begin{equation}
\left|
C_{\overline{\widetilde{Y}}}
-
C_{\overline{Y}}
\right|
\le
\log_{10}
\frac{|\overline{Y}|+\alpha}
     {|\overline{Y}|-\alpha}
+
\log_{10}
\frac{\sigma_Y+\gamma}
     {\sigma_Y-\gamma}.
\label{eq:composition_digit_gap_bound}
\end{equation}
\end{theorem}

\begin{proof}
Fix a stochastic sample $i$. Replace the trusted GEMM-like operators by their data-perturbed counterparts one at a time, in topological order. When the operator $p$ is replaced, all previous replacements are held fixed and the remaining downstream computation is the same in the two trajectories. By the assumed local Lipschitz property of that downstream map, the change in the selected scalar output coordinate caused by this replacement is at most $L_p\|\Delta_p^{(i)}\|$. Summing these changes by the triangle inequality gives
\[
|\widetilde{y}^{(i)}-y^{(i)}|
\le
\sum_{p\in\mathcal{P}} L_p\|\Delta_p^{(i)}\|
=b_i.
\]
Let $e^{(i)}=\widetilde{y}^{(i)}-y^{(i)}$. The representative gap is the mean of these sample perturbations, hence
\[
|\overline{\widetilde{Y}}-\overline{Y}|
=|\overline{e}|
\le \frac{1}{s}\sum_{i=1}^{s}|e^{(i)}|
\le \alpha.
\]
The sample standard deviation of $E=(e^{(1)},\ldots,e^{(s)})$ satisfies
\[
\sigma_E^2
=
\frac{1}{s-1}\sum_{i=1}^{s}(e^{(i)}-\overline{e})^2
\le
\frac{1}{s-1}\sum_{i=1}^{s}(e^{(i)})^2
\le
\frac{1}{s-1}\sum_{i=1}^{s}b_i^2
=\gamma^2,
\]
because the sample mean minimizes the sum of squared deviations. Applying Theorem~\ref{thm:gemm_digit_bias} with $|\overline{e}|\le\alpha$ and $\sigma_E\le\gamma$ yields Eq.~\eqref{eq:composition_digit_gap_bound}.
\end{proof}

Theorem~\ref{thm:composition_perturbation_gap} separates two sources of error in the approximation. The factors $\|\Delta_p^{(i)}\|$ are controlled by the configured perturbation amplitude and by the local magnitude of the GEMM inputs, whereas the factors $L_p$ describe the sensitivity of the downstream composition. Thus, even when each trusted GEMM-like operator is backward stable, the overall graph can still produce a large approximation gap if a later normalization, softmax, attention, or recurrent update amplifies the perturbation. Conversely, when the downstream maps are well conditioned and the perturbations remain small relative to both the representative signal and the stochastic spread, the data-perturbed surrogate has a provably small effect on both the representative value and the significant-digit estimate. This is the regime in which the software reports should be interpreted as efficient operator-level approximations of a finer stochastic propagation. Outside this regime, large or unstable digit reports are themselves a diagnostic signal that the chosen composition is numerically sensitive.

\begin{figure}[htp]
\centering
\begin{pythoncode}[label={lst:layers}]{Neural-network operator wrappers}
import torch
from noisefloat import configure
from noisefloat.nn import NFloatLinear, NFloatReLU, NFloatTensor
from noisefloat.nn import get_kernel_reports

configure(backend="torch", n_samples=3, random_state=42)

model = torch.nn.Sequential(
    NFloatLinear(10, 32),
    NFloatReLU(),
    NFloatLinear(32, 2),
)

x = NFloatTensor(torch.randn(8, 10, dtype=torch.float64))
out = model(x)

print(out.value)       # representative mean tensor
print(out.digits)      # element-wise significant digits
print(get_kernel_reports()[-1])
\end{pythoncode}
\caption{PyTorch neural-network wrapper example showing representative tensors, digit fields, and operator reports.}
\end{figure}

The data-perturbation path is therefore used only when stochastic GEMM inputs have insufficient sample separation. Its reports diagnose sensitivity of the trusted operator to the prescribed input perturbations, but they must not be interpreted as measurements of every rounding event inside the GEMM implementation.

\subsection{Structured Diagnostics and Iteration-Level Tracking}

Numerical-stability information is exposed as structured software reports. Each instrumented neural-network operator may emit a \texttt{KernelReport} containing the recorded operator name, execution phase, average/minimum/maximum significant digits, number of tensor elements, stability flag, backend metadata, and optional detailed tensors depending on \texttt{kernel\_report\_detail}. The reporting API includes \texttt{record\_kernel}, \texttt{get\_kernel\_reports}, \texttt{clear\_kernel\_reports}, \texttt{print\_kernel\_report}, and \texttt{summary}. These reports make the numerical unit of analysis explicit: they allow a developer to identify which neural-network operators lose significant digits and whether the loss occurs in training, validation, or inference.

Iteration-level capture is provided by \texttt{NFloatIterationTracker}. The tracker runs within the \texttt{nfloat\_analysis} context, which controls automatic conversion from floating tensors to stochastic tensors, optional STE rounding, report-name prefixes, execution mode, and user metadata. For an epoch $e$ and iteration $t$, the tracker records a multiset of reports
\begin{equation}
\mathcal{H}_{e,t}=\{\mathcal{R}(v,p,t): v\in\mathcal{V}_{\mathrm{inst}}\},
\end{equation}
where $\mathcal{V}_{\mathrm{inst}}$ is the set of instrumented operators reached during that execution. The exported trace is therefore a time series over operators,
\begin{equation}
\left(\mathcal{H}_{0,0},\mathcal{H}_{0,1},\ldots,\mathcal{H}_{E,T_E}\right),
\end{equation}
which can be joined with task-level metrics such as loss, accuracy, token-level reliability, BLEU, or reconstruction error. This mechanism is designed for numerical debugging of long-running learning pipelines, where a single final accuracy or loss value is insufficient to localize instability.

\begin{figure}
\centering
\begin{pythoncode}[label={lst:iteration-tracker}]{Iteration-level operator tracing}
import torch
from noisefloat.nn import NFloatIterationTracker

tracker = NFloatIterationTracker(task_name="classifier")
criterion = torch.nn.CrossEntropyLoss()

for epoch in range(num_epochs):
    for step, (inputs, targets) in enumerate(loader):
        optimizer.zero_grad()
        logits = model(inputs)              # ordinary training path
        loss = criterion(logits, targets)
        loss.backward()
        optimizer.step()

        with tracker.iteration(
            epoch=epoch,
            iteration=step,
            split="train",
            global_iteration=epoch * len(loader) + step,
            use_ste=True,
            metadata={"backend": "torch"},
        ):
            nfloat_logits = nfloat_model(inputs)
            _ = nfloat_logits.value        # representative prediction

tracker.export("nfloat_trace")
\end{pythoncode}
\caption{Training-loop tracing example for collecting operator reports at each epoch and iteration.}
\end{figure}
Listing~\ref{lst:iteration-tracker} shows the corresponding training-loop pattern. The original optimizer and deterministic loss remain in the user program, while the stochastic analysis context captures the reports emitted by wrapped operators at each iteration. This idiom is deliberately small: it demonstrates how numerical-stability instrumentation is added without restructuring the whole training pipeline.

\subsection{Design Philosophy and Application Value}

The method complements existing stochastic-arithmetic systems rather than replacing them. Its main distinction is software integration: CESTAC-style stochastic arithmetic is available directly inside Python numerical and tensor programs, through backend-native abstractions rather than external hardware rounding-mode control. The same conceptual transformation applies to NumPy, PyTorch, JAX, and TensorFlow. Differentiable backends can use the STE path to preserve compatibility with automatic differentiation. In addition, operator-level reports align the diagnostic output with the structure of modern deep-learning software.

The transformed program estimates the numerical reliability of the selected Python-visible operations and instrumented neural-network operators. It does not claim to reproduce every instruction-level rounding event inside vendor libraries, CUDA kernels, or fused backend implementations. Consequently, the diagnostic granularity is controlled by the chosen instrumentation boundary. This boundary is deliberate: it provides a practical level of localization for deep-learning developers while retaining a mathematically defined stochastic semantics for each selected operator.

This design is compositional. Predefined wrappers cover many common neural-network operators, and custom functions can be lifted with \texttt{NFloatOperator} or \texttt{nfloat\_operator}. The reports can be aggregated over iterations to study how numerical reliability evolves during training or inference. The framework is therefore useful for detecting catastrophic cancellation, rounding-error amplification, loss of significant digits, and operator-level sensitivity under reduced or mixed precision. In practice, these diagnostics can guide decisions about where to use higher precision, where to apply stable reformulations, and where to replace or reconfigure numerically fragile neural-network operators.

\section{Experimental Simulations}\label{sec:experiments}
\subsection{Experimental Setup}

Our simulations were performed on a single computing node equipped with one NVIDIA A100 GPU (80~GB memory, configured as a \texttt{7g.80gb} MIG instance), eight CPU cores, and 64~GB of RAM. The deep-learning operator experiments use PyTorch 2.12.0 with CUDA 13.0. Unless stated otherwise, deterministic modules and image-valued inputs use FP32. For numerical analysis, a separate shadow model copies the current deterministic parameters and evaluates three synchronous stochastic representatives in FP32 precision at confidence level $0.95$. Random seeds control data ordering, model initialization, and randomized quantization. CUDA is selected when available, except where a separate CPU device is explicitly used to bound the memory footprint of stochastic analysis. In addition, the mathematical formulation in this section is self-contained.

In the numerical experiments, we simulate the \texttt{noisefloat} in five sessions. The first one examines examples of classical numerical computations and deep learning operators in a single pass. The second session conducts pathology study on the unstable learning operators in training and compare it with the stable one. The second and third sessions conduct pathology study to examine the practical task of training with controlled operator with injected numerical instabilities. The last two sessions separately examine the relationship between accuracy and independent deep learning operators on activations and normalization, respectively.

% Across all learning experiments, we report deterministic task-level performance together with stochastic numerical diagnostics. The deterministic metrics serve as a sanity check that the analyzed model is solving the intended learning problem, whereas the CESTAC-derived significant digits characterize the numerical reliability of selected operators along that trajectory. For classification tasks, the deterministic metric is mini-batch accuracy. For diffusion, it is reconstruction or noise-prediction MSE. For translation, it is validation loss, token-level accuracy, and sequence-level BLEU or exact match. These task metrics are not substitutes for the significant-digit estimator: they quantify semantic performance, while the stochastic reports quantify local numerical sensitivity.

The deep-learning experiments separate parameter optimization from numerical instrumentation. Let $F_{\theta_t}$ denote the deterministic network at iteration $t$, $\mathcal{L}$ its task loss, and $\mathcal{U}$ the optimizer update. The training path computes
\begin{equation}
\theta_{t+1}
=
\mathcal{U}\!\left(\theta_t,
\nabla_{\theta}\mathcal{L}(F_{\theta_t}(x_t),y_t)\right).
\label{eq:deterministic-training-update}
\end{equation}
After this update, \texttt{NFloatShadowModel} copies both parameters and persistent buffers from the deterministic model into an arithmetic-level stochastic model $\widehat{F}_{\widetilde{\theta}_{t+1}}$ and evaluates a diagnostic forward pass,
\begin{equation}
\widetilde{\theta}_{t+1}\leftarrow\theta_{t+1},
\qquad
Y_t=\widehat{F}_{\widetilde{\theta}_{t+1}}(X_t),
\qquad
X_t=(x_t^{(1)},x_t^{(2)},x_t^{(3)}).
\label{eq:shadow-evaluation}
\end{equation}

At every analysis point, the shadow state is therefore an exact copy of the latest deterministic state, and stochastic variability cannot alter the optimization trajectory. In the classification runners, deterministic batch accuracy comes from the pre-update training forward pass in Eq.~\eqref{eq:deterministic-training-update}, whereas representative accuracy and operator reports come from the post-update shadow pass in Eq.~\eqref{eq:shadow-evaluation}. Because the two accuracies use different parameter states, they are reported only as task monitors and are not interpreted as a rounding-error estimate. All operator digit estimates correspond to the synchronized post-update state. The arithmetic-level entry points set \texttt{noisefloat\_NFLOAT\_MODE=arithmetic}. Supported PyTorch-visible arithmetic, activations, normalizations, recurrent operations, and attention operations therefore follow the semantics in Section~\ref{sec:method}, while GEMM-like operations use Eq.~\eqref{eq:gemm_lifted_perturbed}. Detailed event tracing is disabled in these experiments.

\subsection{Controlled Arithmetic and Operator Instability}
\label{sec:exp-controlled-instability}

This experiment is  to verify the numerical stability of the numerical computations and common deep learning layers using our software. Its purpose is to test a falsifiable implication of the method: when two computations implement related tasks but one contains a known numerical hazard, the hazardous construction should exhibit lower CESTAC significant digits or a non-finite stochastic output. We consider two common mechanisms. For a subtraction $a-b$, the cancellation ratio
\begin{equation}
\rho_{\ominus}(a,b)
=
\frac{|a-b|}{|a|+|b|}
\label{eq:cancellation-ratio}
\end{equation}
approaches zero when large, nearly equal operands produce a small residual. More generally, for a differentiable operator $f$, the local relative condition measure
\begin{equation}
\kappa_f(x)
=
\frac{\|x\|\,\|J_f(x)\|}{\|f(x)\|}
\label{eq:operator-condition-measure}
\end{equation}
becomes large when a small perturbation of the input can induce a large relative perturbation of the output. The controlled cases below deliberately vary $\rho_{\ominus}$, $\kappa_f$, or the representability of an intermediate result while keeping the semantic task as close as possible.

\paragraph{Classical arithmetic controls.} The first component uses the NumPy backend and scalar operator overloading through \texttt{NFloat}, focusing on classical examples from numerical analysis. It evaluates stochastic variables in simulated FP32 and FP64 formats with $s=3$, confidence $0.95$, and seed 2026. Four cases compare a cancellation-prone expression $f^{\mathrm{u}}$ with an algebraically or intrinsically stabilized expression $f^{\mathrm{s}}$:
\begin{align}
f_{\exp}^{\mathrm{u}}(x)
&= \exp(x)-1,
&
f_{\exp}^{\mathrm{s}}(x)
&= \operatorname{expm1}(x),
& x&=10^{-11},
\label{eq:control-expm1}
\\
f_{\sqrt{\phantom{x}}}^{\mathrm{u}}(x)
&= \sqrt{x+1}-\sqrt{x},
&
f_{\sqrt{\phantom{x}}}^{\mathrm{s}}(x)
&= \frac{1}{\sqrt{x+1}+\sqrt{x}},
& x&=10^{14},
\label{eq:control-sqrt}
\\
f_{\cos}^{\mathrm{u}}(x)
&= 1-\cos(x),
&
f_{\cos}^{\mathrm{s}}(x)
&= 2\sin^2(x/2),
& x&=10^{-6},
\label{eq:control-cos}
\\
f_{\mathrm{quad}}^{\mathrm{u}}(b)
&= \frac{-b-\sqrt{b^2-4}}{2},
&
f_{\mathrm{quad}}^{\mathrm{s}}(b)
&= \frac{2}{-b+\sqrt{b^2-4}},
& b&=-10^6.
\label{eq:control-quadratic}
\end{align}
In Eqs.~\eqref{eq:control-expm1}--\eqref{eq:control-quadratic}, the unstable form subtracts quantities whose leading digits agree, whereas the stable form avoids that subtraction or delegates it to a compensated intrinsic. The quadratic pair~\eqref{eq:control-quadratic} computes the small root of $x^2+bx+1=0$, and the stable form follows from the Vi\`ete relation $x_1x_2=1$ \citep{Goldberg1991FloatingPoint}.

The suite then increases structural conditioning. The orthogonal control evaluates a component of
\begin{equation}
Q(\theta)
=
\begin{bmatrix}
\cos\theta & -\sin\theta\\
\sin\theta & \cos\theta
\end{bmatrix},
\qquad
\theta=\frac{\pi}{4},
\qquad
\kappa_2(Q)=1.
\label{eq:control-orthogonal}
\end{equation}
For the Hilbert case, $H\in\mathbb{R}^{12\times12}$ is defined by
\begin{equation}
H_{ij}=\frac{1}{i+j+1},
\qquad
b=H\mathbf{1},
\qquad
H\widehat{x}=b,
\label{eq:control-hilbert}
\end{equation}
with zero-based $i,j$ and Gaussian elimination without pivoting. The monitored component has exact value $x_6=1$. For the clustered Vandermonde case,
\begin{equation}
V_{ij}=t_i^j,
\qquad
t_i=0.99+\frac{0.02i}{7},
\qquad
b=Vc,
\qquad
c=(1,\ldots,8)^\top,
\label{eq:control-vandermonde}
\end{equation}
and the monitored component is $c_4=5$. The tightly clustered nodes make the columns of $V$ nearly dependent, so elimination can amplify input and rounding perturbations. Finally, the Rump polynomial
\begin{equation}
P(x,y)=9x^4-y^4+2y^2
\label{eq:control-rump}
\end{equation}
is evaluated at $(10864,18817)$ and $(1/3,2/3)$. The first input combines very large terms into the exact value one and is therefore cancellation dominated, whereas the second has the benign exact value $65/81$. These controls cover stable reformulation, cancellation, ill-conditioned linear algebra, and polynomial error amplification. For each case, the experiment records the representative value, stochastic spread, significant digits, deviation from the case-specific reference expression, and summary counts of diagnosed instability sources. The verification runner uses a six-digit stability criterion.

\begin{table*}[t]
  \centering
  \caption{Classical numerical-stability benchmark results under stochastic arithmetic. Each row reports the estimated number of significant digits for the corresponding formulation.}
  \label{tab:classic-stability-suite}
  \scriptsize

  \noindent\begin{minipage}[t]{0.49\textwidth}
    \centering
    \textbf{(a) Single precision \((e,p)=(8,23)\)}
    \label{tab:classic-stability-single}\par\medskip

    \resizebox{\linewidth}{!}{%
    \begin{tabular}{llcr}
      \toprule
      Case & Variant & Type & Sig. digits \\
      \midrule
      expm1
        & \(\exp(x)-1\)
        & U & 0.000 \\
      expm1
        & \(\operatorname{expm1}(x)\)
        & S & 6.905 \\
      sqrt diff.
        & \(\sqrt{x+1}-\sqrt{x}\)
        & U & 0.000 \\
      sqrt diff.
        & rationalized
        & S & 6.352 \\
      \(1-\cos(x)\)
        & direct
        & U & 0.000 \\
      \(1-\cos(x)\)
        & \(2\sin^2(x/2)\)
        & S & 5.895 \\
      quadratic root
        & direct formula
        & U & 0.000 \\
      quadratic root
        & Vieta formula
        & S & 6.231 \\
      orthogonal map
        & rotation
        & S & 6.194 \\
      Hilbert solve
        & Gaussian, no pivot
        & U & 0.000 \\
      Vandermonde
        & clustered nodes, no pivot
        & U & 0.000 \\
      Rump \(P_1\)
        & \(P_1\)
        & U & 0.000 \\
      Rump \(P_2\)
        & \(P_2\)
        & S & 6.973 \\
      \bottomrule
    \end{tabular}}
  \end{minipage}%
  \hfill
  \begin{minipage}[t]{0.49\textwidth}
    \centering
    \textbf{(b) Double precision \((e,p)=(11,52)\)}
    \label{tab:classic-stability-double}\par\medskip

    \resizebox{\linewidth}{!}{%
    \begin{tabular}{llcr}
      \toprule
      Case & Variant & Type & Sig. digits \\
      \midrule
      expm1
        & \(\exp(x)-1\)
        & U & 4.196 \\
      expm1
        & \(\operatorname{expm1}(x)\)
        & S & 15.000 \\
      sqrt diff.
        & \(\sqrt{x+1}-\sqrt{x}\)
        & U & 0.781 \\
      sqrt diff.
        & rationalized
        & S & 14.973 \\
      \(1-\cos(x)\)
        & direct
        & U & 2.773 \\
      \(1-\cos(x)\)
        & \(2\sin^2(x/2)\)
        & S & 14.638 \\
      quadratic root
        & direct formula
        & U & 3.656 \\
      quadratic root
        & Vieta formula
        & S & 14.978 \\
      orthogonal map
        & rotation
        & S & 14.787 \\
      Hilbert solve
        & Gaussian, no pivot
        & U & 0.142 \\
      Vandermonde
        & clustered nodes, no pivot
        & U & 0.000 \\
      Rump \(P_1\)
        & \(P_1\)
        & U & 0.000 \\
      Rump \(P_2\)
        & \(P_2\)
        & S & 15.000 \\
      \bottomrule
    \end{tabular}}
  \end{minipage}

  \vspace{0.25em}

  \footnotesize
  U denotes a cancellation-prone or ill-conditioned formulation, whereas S denotes its numerically stable or well-conditioned counterpart.
\end{table*}

Table~\ref{tab:classic-stability-suite} shows that the examined numerical validity described by the significant digits behaves as intended. In the FP32 validation, every cancellation-prone or ill-conditioned formulation is reported with zero significant digits, whereas the stabilized counterparts retain approximately six to seven significant digits, approximately the same number of significant digits as FP32. In the FP64 validation, the stabilized formulations reach the estimator ceiling of roughly 15 digits, while the unstable formulations remain far below 6 digits. The Rump polynomial is especially diagnostic{--}the benign input $P(1/3,2/3)=65/81$ is stable, while the large-integer input loses all reliable digits. These results demonstrate the validity of our implementation while serving as a calibration check that the scalar stochastic-arithmetic path distinguishes algebraically equivalent but numerically different formulations.

\paragraph{Controlled neural-network operators.} The second component is examined with arithmetic-level stochastic-validation wrappers and constructs six matched stable/unstable pairs. Let
\begin{equation}
Y_{g,v,r}
=
\widehat{f}_{g,v}\!\left(X_{g,v,r}\right),
\qquad
v\in\{\mathrm{s},\mathrm{u}\},
\label{eq:controlled-operator-lift}
\end{equation}
denote the stochastic output for operator group $g$, variant $v$, and trial $r$. The pairs are defined as follows.

\emph{Linear transformation.} The stable map uses an orthogonal matrix obtained by QR factorization,
\begin{equation}
y^{\mathrm{s}}=Qx,
\qquad
Q^\top Q=I,
\qquad
\kappa_2(Q)=1.
\label{eq:controlled-linear-stable}
\end{equation}
The unstable map uses $w=(1,-1)$ and inputs $x=(M+\delta,M)^\top$, where $M=10^8$ and $\delta\in\{1,2,4\}$:
\begin{equation}
y^{\mathrm{u}}
=
w^\top x
=
(M+\delta)-M
=
\delta,
\qquad
\rho_{\ominus}(M+\delta,M)
\approx
\frac{\delta}{2M}.
\label{eq:controlled-linear-unstable}
\end{equation}
The semantic operation remains a linear layer, but the second construction exposes a small signal through cancellation of two large activations.

\emph{One-dimensional convolution.} For a moderate-amplitude sinusoidal signal with small additive noise and a local window $x_{j:j+2}$, the stable filter is the convex smoothing stencil
\begin{equation}
h^{\mathrm{s}}=\left(\frac14,\frac12,\frac14\right),
\qquad
y_j^{\mathrm{s}}=\frac14x_j+\frac12x_{j+1}+\frac14x_{j+2}.
\label{eq:controlled-conv-stable}
\end{equation}
The unstable filter is a first difference,
\begin{equation}
h^{\mathrm{u}}=(1,-1),
\qquad
x_j=M+r_j,
\qquad
y_j^{\mathrm{u}}=r_j-r_{j+1},
\label{eq:controlled-conv-unstable}
\end{equation}
with $M=10^8$. The common offset is irrelevant in exact arithmetic but dominates the represented operands, so the convolution becomes a sequence of cancellation-prone subtractions.

\emph{Layer normalization.} For feature dimension $d=32$, both variants compute
\begin{equation}
\mu(x)=\frac{1}{d}\sum_{j=1}^{d}x_j,
\qquad
\sigma^2(x)=\frac{1}{d}\sum_{j=1}^{d}(x_j-\mu)^2,
\qquad
\operatorname{LN}_{\varepsilon}(x)_j
=
\frac{x_j-\mu}{\sqrt{\sigma^2+\varepsilon}}.
\label{eq:controlled-layernorm}
\end{equation}
The stable input has $x_j\sim\mathcal{N}(0,1)$ and $\varepsilon=10^{-5}$. The unstable input has
\begin{equation}
x_j=1+2^{-22}\xi_j,
\qquad
\xi_j\sim\mathcal{N}(0,1),
\qquad
\varepsilon=10^{-12}.
\label{eq:controlled-layernorm-nearconstant}
\end{equation}
Here, both $x_j-\mu$ and $\sigma^2$ are formed from nearly equal quantities, and perturbations are amplified approximately by $(\sigma^2+\varepsilon)^{-1/2}$.

\emph{Softmax.} For logits $z\in\mathbb{R}^d$, the stable implementation uses translation invariance,
\begin{equation}
p_j^{\mathrm{s}}
=
\frac{\exp(z_j-m)}{\sum_k\exp(z_k-m)},
\qquad
m=\max_k z_k,
\label{eq:controlled-softmax-stable}
\end{equation}
and the stable implementation of Log-Softmax uses translation invariance:
\begin{equation}
\log p_j^{\mathrm{s}}
=
(z_j-m)
-
\log\left(\sum_k \exp(z_k-m)\right),
\label{eq:controlled-log-softmax-stable}
\end{equation}
whereas the naive implementation computes%\qquad m=\max_k z_k.
\begin{equation}
p_j^{\mathrm{u}}
=
\frac{\exp(z_j)}{\sum_k\exp(z_k)}.
\label{eq:controlled-softmax-unstable}
\end{equation}
The logits lie near $10^3$. Equation~\eqref{eq:controlled-softmax-stable} constrains the largest exponential to one, while Eq.~\eqref{eq:controlled-softmax-unstable} deliberately creates non-finite intermediate exponentials. This pair tests whether the instrumentation distinguishes algebraically equivalent formulations with different intermediate ranges.

\emph{Scaled dot-product attention.} With one query $q$, two keys $k_1,k_2$, opposing value vectors $v_1,v_2$, and $d=4$, the operator is
\begin{equation}
s_j=\frac{q^\top k_j}{\sqrt{d}},
\qquad
(p_1,p_2)=\operatorname{softmax}(s_1,s_2),
\qquad
o=p_1v_1+p_2v_2.
\label{eq:controlled-attention}
\end{equation}
The stable input gives $(s_1,s_2)=(2,-2)$ and margin $m_s=4$. In the near-tie construction,
\begin{equation}
q=(1,1,1,1),
\quad
k_1=q,
\quad
k_2=(1-2^{-28},1,1,1),
\quad
m_u=s_1-s_2=2^{-29}.
\label{eq:controlled-attention-neartie}
\end{equation}
The small score margin makes the probability allocation sensitive to rounding. Because $v_1$ and $v_2$ have opposing signs, a perturbation of the attention weights is not masked by similar values at the output.

\emph{Reduction.} The stable reduction sums 256 positive values $a_j\in[1,2]$:
\begin{equation}
S^{\mathrm{s}}=\sum_{j=1}^{256}a_j.
\label{eq:controlled-reduction-stable}
\end{equation}
The unstable reduction sums 256 alternating pairs,
\begin{equation}
S^{\mathrm{u}}
=
\sum_{j=1}^{256}\left(M+(-M+1)\right)
=
256,
\qquad
M=10^8.
\label{eq:controlled-reduction-unstable}
\end{equation}
Although the exact result is moderate, every pair recovers one by subtracting values of magnitude $10^8$, so summation order and rounding determine how much of that residual is retained.

% \paragraph{Detection objective.}
We formulate detection objective as follows. Let $D_{g,v,r}^{\min}$ denote the minimum significant digits reported for the output of variant $v$. The matched-pair hypothesis is
\begin{equation}
\Delta_{g,r}
=
D_{g,\mathrm{s},r}^{\min}
-
D_{g,\mathrm{u},r}^{\min}
>0,
\label{eq:controlled-pair-hypothesis}
\end{equation}
with a non-finite unstable output interpreted as zero reliable digits. Independently of the paired gap, each case is classified by
\begin{equation}
\widehat{u}_{g,v,r}
=
\mathbf{1}\!\left[
D_{g,v,r}^{\min}<\gamma_p
\;\lor\;
\operatorname{nonfinite}(Y_{g,v,r})
\right],
\qquad
\gamma_{23}=3,\quad \gamma_{52}=10,
\label{eq:operator-instability-classifier}
\end{equation}
where $\widehat{u}_{g,v,r}=1$ denotes a detected unstable case. The default execution uses simulated FP32 and three trials with distinct random seeds. Alternatively, FP64 can be selected using the same runner. Detection accuracy, precision, and recall compare Eq.~\eqref{eq:operator-instability-classifier} with the construction labels. Equation~\eqref{eq:controlled-pair-hypothesis} tests relative separation within an operator family, whereas Eq.~\eqref{eq:operator-instability-classifier} tests absolute detection under a precision-dependent digit threshold. The source label attached by this synthetic runner is auxiliary: it is inferred from the controlled construction and observed non-finiteness, while the significant-digit estimate remains the primary experimental endpoint.

\begin{table*}[t]
\centering
\caption{Arithmetic-level stochastic-validation results for deep-learning operators. Each row reports the average and minimum number of significant digits estimated for the corresponding operator.}
\label{tab:operator-level-dl-cestac}
\scriptsize

\begin{minipage}[t]{0.48\textwidth}
\centering
\textbf{Single precision-like format ((e,p)=(8,23))}
\vspace{0.3em}

\resizebox{\linewidth}{!}{%
\begin{tabular}{llcrr}
  \toprule
  Category & Operator & Type & Avg. & Min. \\
  \midrule
  softmax
    & naive\_softmax\_overflow
    & U & 0.000 & 0.000 \\
  softmax
    & shifted\_softmax\_large\_logits
    & S & 3.527 & 3.487 \\
  softmax
    & translation\_invariant\_log\_softmax
    & S & 3.666 & 3.620 \\
  normalization
    & layernorm\_regular\_variance
    & S & 4.124 & 3.672 \\
  normalization
    & layernorm\_near\_constant
    & U & 0.000 & 0.000 \\
  attention
    & attention\_separated\_logits
    & S & 6.466 & 6.257 \\
  attention
    & attention\_near\_tie\_logits
    & U & 0.000 & 0.000 \\
  convolution
    & smooth\_conv1d
    & S & 6.363 & 6.225 \\
  convolution
    & difference\_conv1d\_large\_signal
    & U & 0.000 & 0.000 \\
  linear
    & orthogonal\_linear\_map
    & S & 4.723 & 4.347 \\
  linear
    & cancelling\_linear\_map
    & U & 0.000 & 0.000 \\
  reduction
    & positive\_reduction\_sum
    & S & 6.405 & 5.989 \\
  reduction
    & alternating\_cancellation\_sum
    & U & 0.662 & 0.424 \\
  \bottomrule
\end{tabular}}

\end{minipage}
\hfill
\begin{minipage}[t]{0.49\textwidth}
\centering
\textbf{Double precision-like format ((e,p)=(11,52))}
\vspace{0.3em}

\resizebox{\linewidth}{!}{%
\begin{tabular}{llcrr}
  \toprule
  Category & Operator & Type & Avg. & Min. \\
  \midrule
  softmax
    & naive\_softmax\_overflow
    & U & 0.000 & 0.000 \\
  softmax
    & shifted\_softmax\_large\_logits
    & S & 12.257 & 12.217 \\
  softmax
    & translation\_invariant\_log\_softmax
    & S & 12.394 & 12.347 \\
  normalization
    & layernorm\_regular\_variance
    & S & 12.846 & 12.264 \\
  normalization
    & layernorm\_near\_constant
    & U & 5.387 & 5.317 \\
  attention
    & attention\_separated\_logits
    & S & 14.925 & 14.840 \\
  attention
    & attention\_near\_tie\_logits
    & U & 5.783 & 5.624 \\
  convolution
    & smooth\_conv1d
    & S & 14.941 & 14.842 \\
  convolution
    & difference\_conv1d\_large\_signal
    & U & 4.733 & 3.880 \\
  linear
    & orthogonal\_linear\_map
    & S & 13.258 & 12.881 \\
  linear
    & cancelling\_linear\_map
    & U & 7.272 & 7.113 \\
  reduction
    & positive\_reduction\_sum
    & S & 14.721 & 14.424 \\
  reduction
    & alternating\_cancellation\_sum
    & U & 8.457 & 8.310 \\
  \bottomrule
\end{tabular}}

\end{minipage}

\vspace{0.25em}

\footnotesize
S denotes an operator expected to be numerically stable, while U denotes an operator expected to exhibit numerical instability. The benchmark achieved perfect classification, with accuracy, precision, and recall all equal to (1.000).
\end{table*}

Table~\ref{tab:operator-level-dl-cestac} shows the results of the numerical stability test for the neural operator under arithmetic-level stochastic instrumentation. Stability is measured by the lowest number of significant digits found in the output tensor. In the FP32 setting, the unstable formulations keep very few reliable digits on average (between 0 and 1.015), while the stable versions keep more than three. In the FP64 setting the qualitative ordering stays the same, but the unstable versions gain a few more digits because the format has a larger significand. The two softmax formulations highlight that it is the way the computation is written, not the meaning of the operation, that decides stability. The shifted version avoids overflow by keeping the exponentials in a safe range and is therefore stable. The naive version overflows and is flagged as unstable immediately. The attention and normalization examples show that instability does not always require overflow. It can also happen when there are nearly identical values or when the variance is close to zero.

% The Table~\ref{tab:operator-level-dl-cestac} presents the numerical stability (via reporting the minimum significant digits in the output tensor) check for the aforementioned neural operator under CESTAC-style arithmetic-level instrumentation. In the FP32 setting, the unstable constructions have a mean minimum significant digits between 0 and 1.015, while the stable constructions remain above the three-digit detection threshold. In the FP64 setting, the same qualitative ordering is preserved, but the unstable cases recover some additional digits because the simulated format has a wider significand. The softmax pair illustrates the role of formulation rather than task semantics: the shifted implementation keeps the exponentials in range, whereas the naive implementation overflows and is immediately classified as unstable. The controlled attention and normalization cases show that instability can also arise without overflow, through near ties and near-zero variance.

\begin{figure}[ht]
    \centering
    \subfloat[Significant digits estimated for the three softmax implementations.]{
        \includegraphics[width=0.48\linewidth]{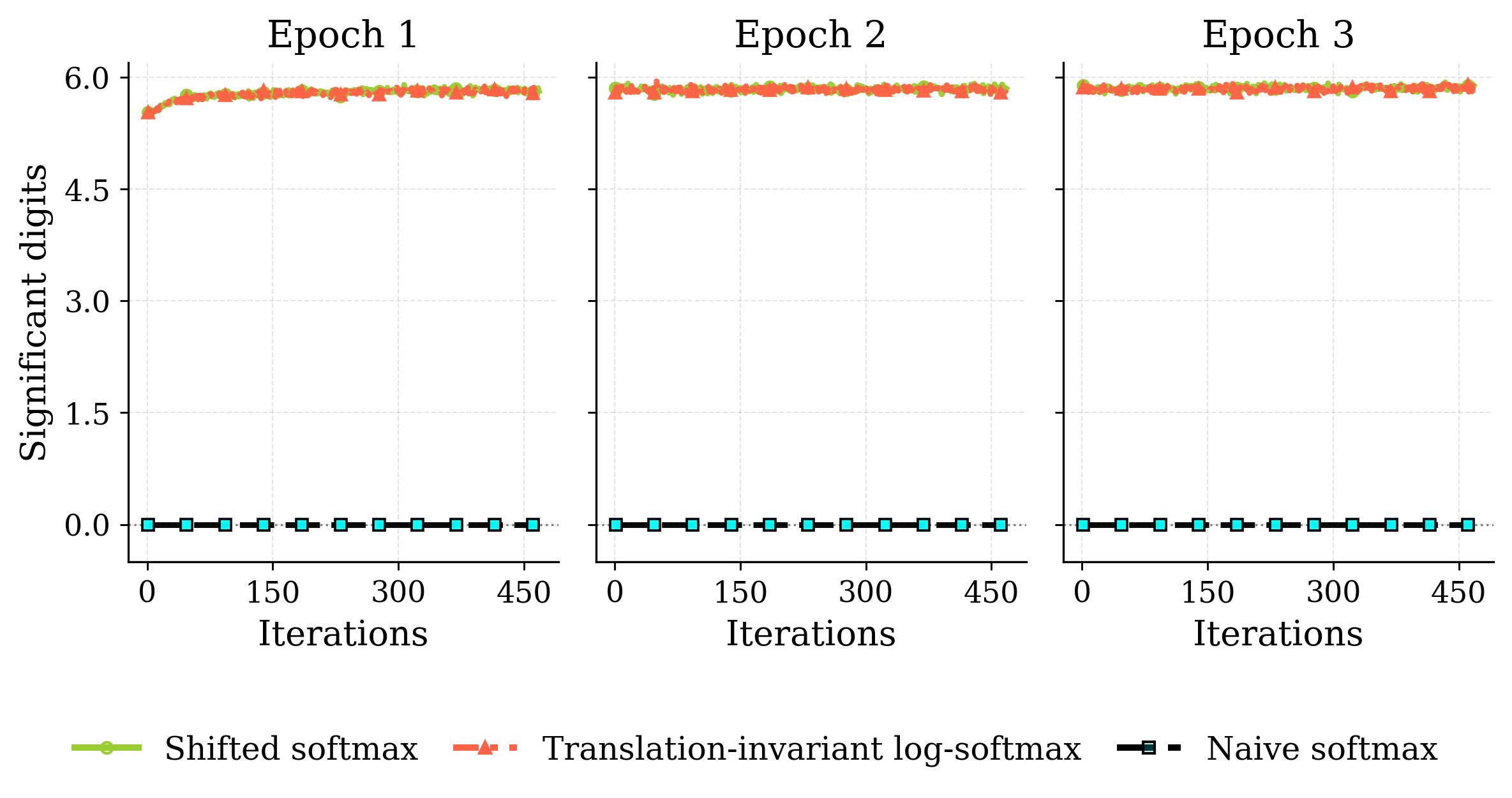}}
    \subfloat[Representative negative log-likelihood loss during training.]{
        \includegraphics[width=0.48\linewidth]{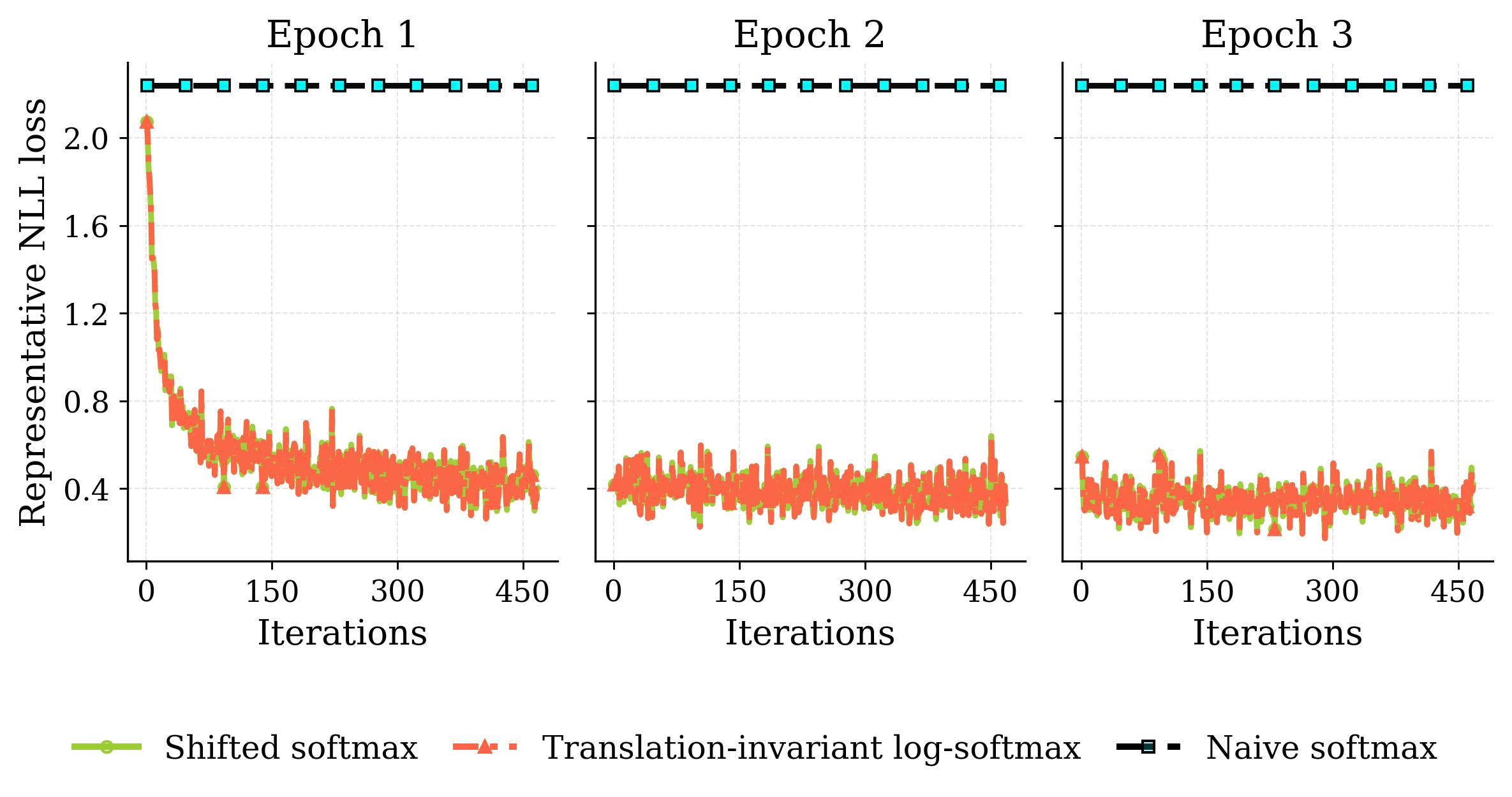}}
    \caption{Fashion-MNIST softmax stability experiment under arithmetic-level CESTAC instrumentation. The shifted softmax and the translation-invariant log-softmax remain finite and numerically usable throughout training, while the naive softmax becomes non-finite under the large-logit stress setting and its loss is clipped for visualization. The result illustrates that CESTAC significant digits diagnose local numerical reliability of the softmax operator, whereas the training loss measures the end-to-end optimization objective. }
    \label{fig:unstable_stable_softmax}
\end{figure}

\subsection{Controlled Operator-Level Numerical Pathology Study}
\label{sec:exp-controlled-neural-pathology}

The preceding controlled neural-operator benchmark in Table~\ref{tab:operator-level-dl-cestac} evaluates isolated operator families and shows that \xyy{the CESTAC-style arithmetic-level reports} separate stable formulations from deliberately unstable ones. The present experiment turns that isolated benchmark into an end-to-end diagnostic test. We select three representative hazards from the same families{--}cancelling linear transformations, near-constant normalization, and near-tie attention{--}and insert one of them into an otherwise ordinary classifier. The question is therefore more operational than in Table~\ref{tab:operator-level-dl-cestac}: when a known numerical pathology is embedded in a trainable network, can the stochastic reports localize the injected operator and indicate whether its loss of significant digits reaches the final decision? To answer this question, we simulate the classification tasks with Fashion-MNIST \citep{xiao2017fashionmnist} and CIFAR-10 \citep{krizhevsky2009learning}, respectively. Both runs use the same training budget and stochastic-analysis configuration: Adam with learning rate $10^{-3}$, batch size $64$, hidden dimension $128$, a training subset with 2,048 examples,  and at most 32 optimizer steps per epoch. In this and the following experiments, our objective is not to achieve state-of-the-art task performance, but to evaluate whether the proposed stochastic numerical diagnostics can localize and quantify numerical instability. Therefore, a short three-epoch training schedule is sufficient for the validation purpose.
 
The architectures differ only in the image front end needed by the dataset. For Fashion-MNIST, the deterministic classifier is the shallow MLP
\begin{equation}
z
=
W_2\,
\Pi\!\left(
\operatorname{ReLU}(W_1\operatorname{vec}(x)+b_1)
\right)+b_2,
\label{eq:pathology-mlp}
\end{equation}
whereas for CIFAR-10 the deterministic classifier uses a fixed compact convolutional front end before the same controlled hidden operator,
\begin{align}
u_1 &= \operatorname{ReLU}(\operatorname{Conv}_{3\times 3,s=2,p=1}^{3\rightarrow 32}(x)),\nonumber, \\
u_2 &= \operatorname{ReLU}(\operatorname{Conv}_{3\times 3,s=2,p=1}^{32\rightarrow 64}(u_1)),\nonumber\\
z
&=
W_2\,
\Pi\!\left(
\operatorname{ReLU}(W_1\operatorname{vec}(u_2)+b_1)
\right)+b_2 ,
\label{eq:pathology-cnn}
\end{align}
where $\Pi$ is the controlled operator under test and the CIFAR-10 images are normalized with the channel-wise mean and standard deviation used by the runner. The deterministic model is trained in FP32. After each optimizer step, a \xyy{synchronized CESTAC-style arithmetic-level shadow model} replays the same batch with three stochastic samples in FP32 precision. The shadow pass is used only for numerical diagnosis and does not modify the optimizer state or the deterministic training trajectory.

The experiment instantiates six versions of $\Pi$, organized as stable and unstable counterparts of the same operator classes used in the controlled operator table. Three are stable controls:
\begin{align}
\Pi_{\mathrm{id}}(h)&=h,
\label{eq:pathology-identity}\\
\Pi_{\mathrm{ln}}^{\mathrm{s}}(h)
&=
\operatorname{LayerNorm}_{10^{-5}}(h),
\label{eq:pathology-regular-ln}\\
\Pi_{\mathrm{att}}^{\mathrm{s}}(h)
&=
\operatorname{sigmoid}(8)h+\left(1-\operatorname{sigmoid}(8)\right)(-h),
\label{eq:pathology-separated-att}
\end{align}
where the attention-like gate in Eq.~\eqref{eq:pathology-separated-att} has a large logit separation. In the normalization formulas below, $\mu(h)$ denotes the feature-wise mean of the hidden vector. Three variants inject controlled numerical hazards:
\begin{align}
&\Pi_{\mathrm{can}}^{\mathrm{u}}(h)
=(h+10^{8})-10^{8},
\label{eq:pathology-cancellation}\\
&\Pi_{\mathrm{ln}}^{\mathrm{u}}(h)
=
\operatorname{LayerNorm}_{10^{-12}}
\!\left((\mu(h)+\alpha_{\mathcal{D}}(h-\mu(h))+\delta_{\mathcal{D}})-\delta_{\mathcal{D}}\right),
\label{eq:pathology-nearconstant-ln}\\
&\Pi_{\mathrm{att}}^{\mathrm{u}}(h)
=
\operatorname{sigmoid}(-2^{-22})(h+10^{8})
+
\left(1-\operatorname{sigmoid}(-2^{-22})\right)(-h-10^{8}+1).
\label{eq:pathology-neartie-att}
\end{align}
Equation~\eqref{eq:pathology-cancellation} is the network-level analogue of the cancelling linear map in Table~\ref{tab:operator-level-dl-cestac}: it removes a large offset and should lose the small residual in FP32-like arithmetic. Equation~\eqref{eq:pathology-nearconstant-ln} corresponds to the near-constant LayerNorm case: it drives the variance toward zero, so centering and variance normalization amplify perturbations. The dataset-specific stress parameters are $(\alpha_{\mathcal{D}},\delta_{\mathcal{D}})=(2^{-22},0)$ for Fashion-MNIST and $(2^{-30},10^8)$ for CIFAR-10, with the latter adding the same large-offset cancellation used elsewhere in this controlled study before normalization. Equation \eqref{eq:pathology-neartie-att} mirrors the near-tie attention case: nearly tied logits are combined with opposing large values, so small perturbations in the gate can change the cancellation pattern at the output. For each dataset, these constructions preserve the surrounding classifier architecture, so differences in the reports can be attributed to the inserted operator rather than to a different model class.

For each batch $t$, the pathology-operator output $Y_{\Pi,t}$ emits average and minimum significant-digit summaries, denoted $D_{\Pi,t}^{\mathrm{avg}}$ and $D_{\Pi,t}^{\min}$. The minimum is the stricter indicator because it captures the least reliable tail of the output tensor. To measure whether stochastic operator uncertainty affects decisions, let $\widetilde{z}_b^{(i)}$ be the shadow-model class-logit vector for example $b$ and stochastic sample $i\in\{1,2,3\}$, and define
\begin{equation}
\widehat{y}_b^{(i)}
=
\arg\max_c \widetilde{z}_{b,c}^{(i)}.
\end{equation}
The sample top-1 agreement rate is
\begin{equation}
A_t
=
\frac{1}{B}\sum_{b=1}^{B}
\mathbf{1}\!\left[
\widehat{y}_b^{(1)}
=\widehat{y}_b^{(2)}
=\widehat{y}_b^{(3)}
\right],
\label{eq:sample-top1-agreement}
\end{equation}
and the disagreement rate is $1-A_t$. We also record the fifth percentile of the representative top-1/top-2 logit margin and the fifth percentile of the decision signal-to-noise ratio, but the primary output-level diagnostic in this subsection is Eq.~\eqref{eq:sample-top1-agreement}: it directly asks whether the three stochastic realizations induce the same class decision.

\begin{figure}
\centering
\subfloat[Epoch-wise pathology-operator significant digits.]{
  \includegraphics[width=0.48\linewidth]{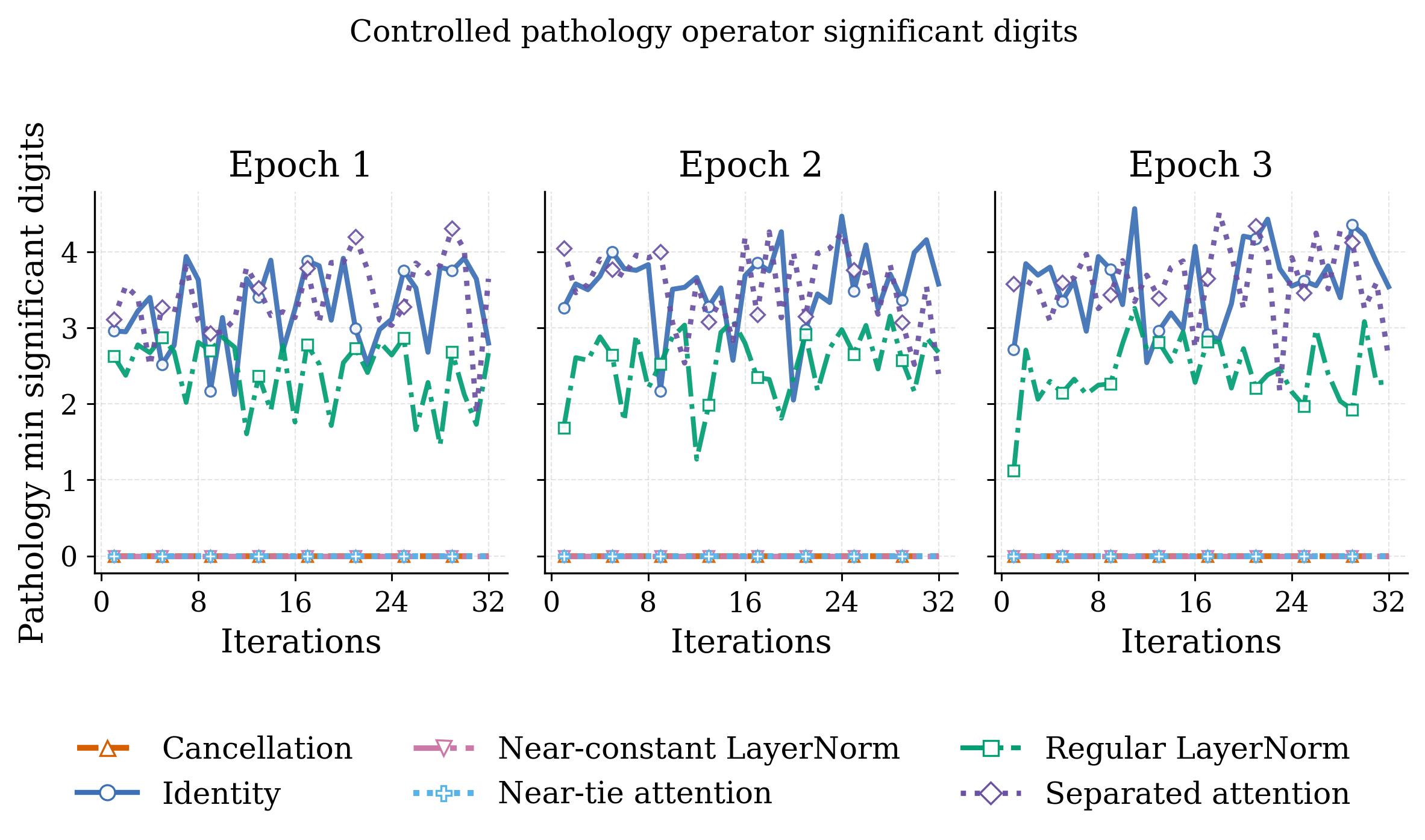}}
\subfloat[Mean digit ranking by inserted operator.]{
  \includegraphics[width=0.48\linewidth]{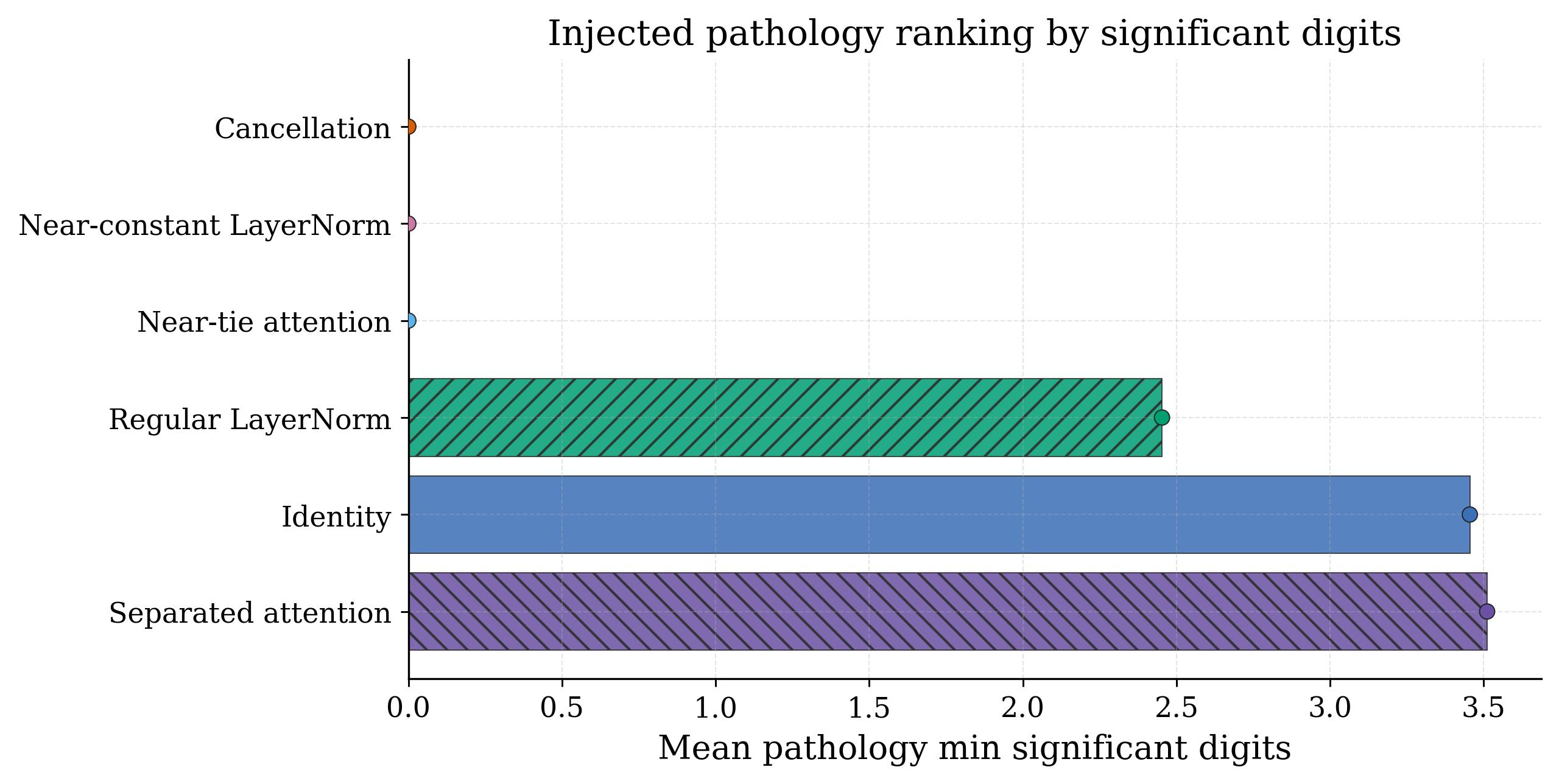}}\\
\subfloat[Task accuracy and stochastic decision agreement.]{
  \includegraphics[width=0.51\linewidth]{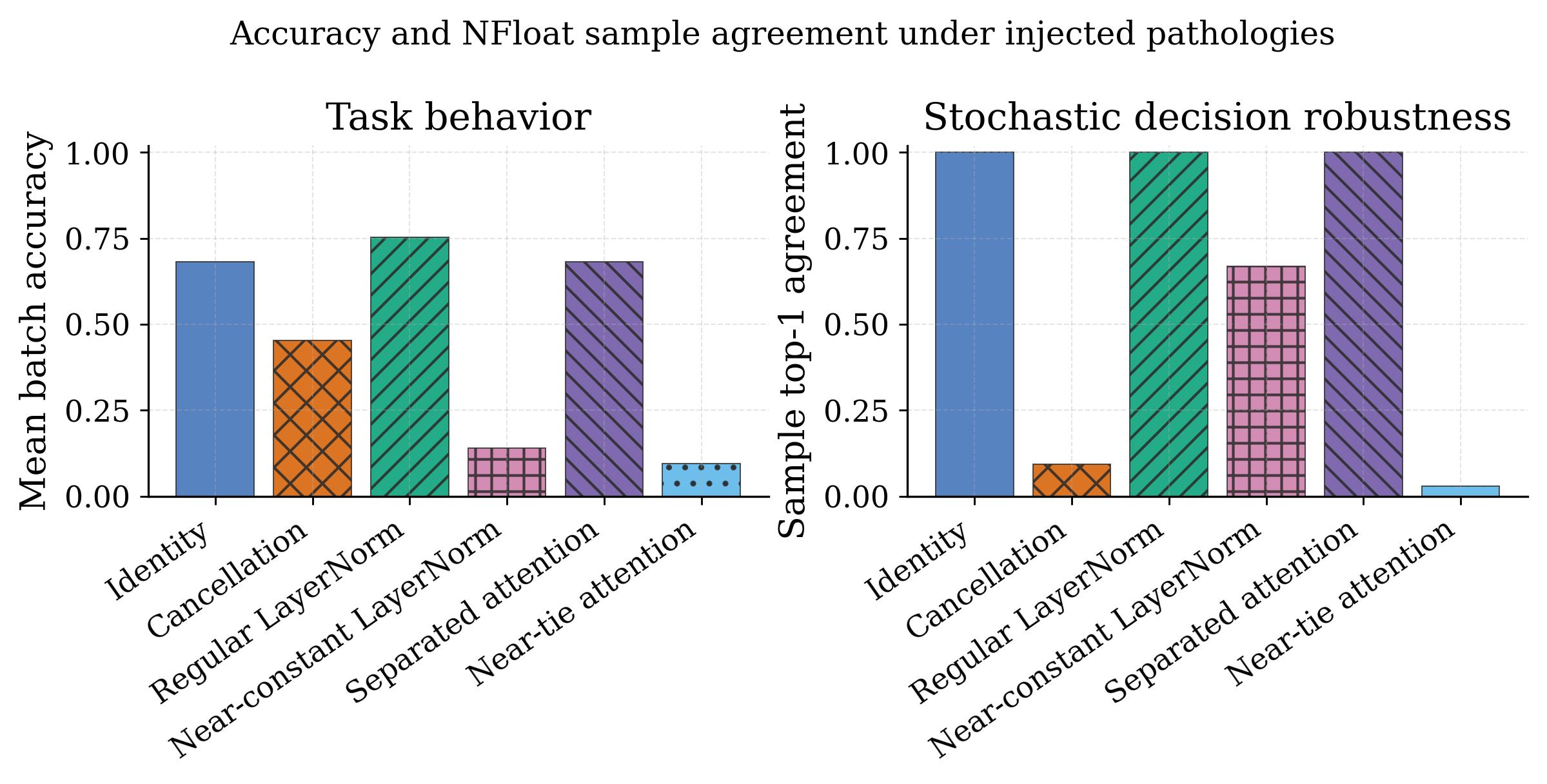}}
\subfloat[Operator digits versus decision agreement.]{
  \includegraphics[width=0.42\linewidth]{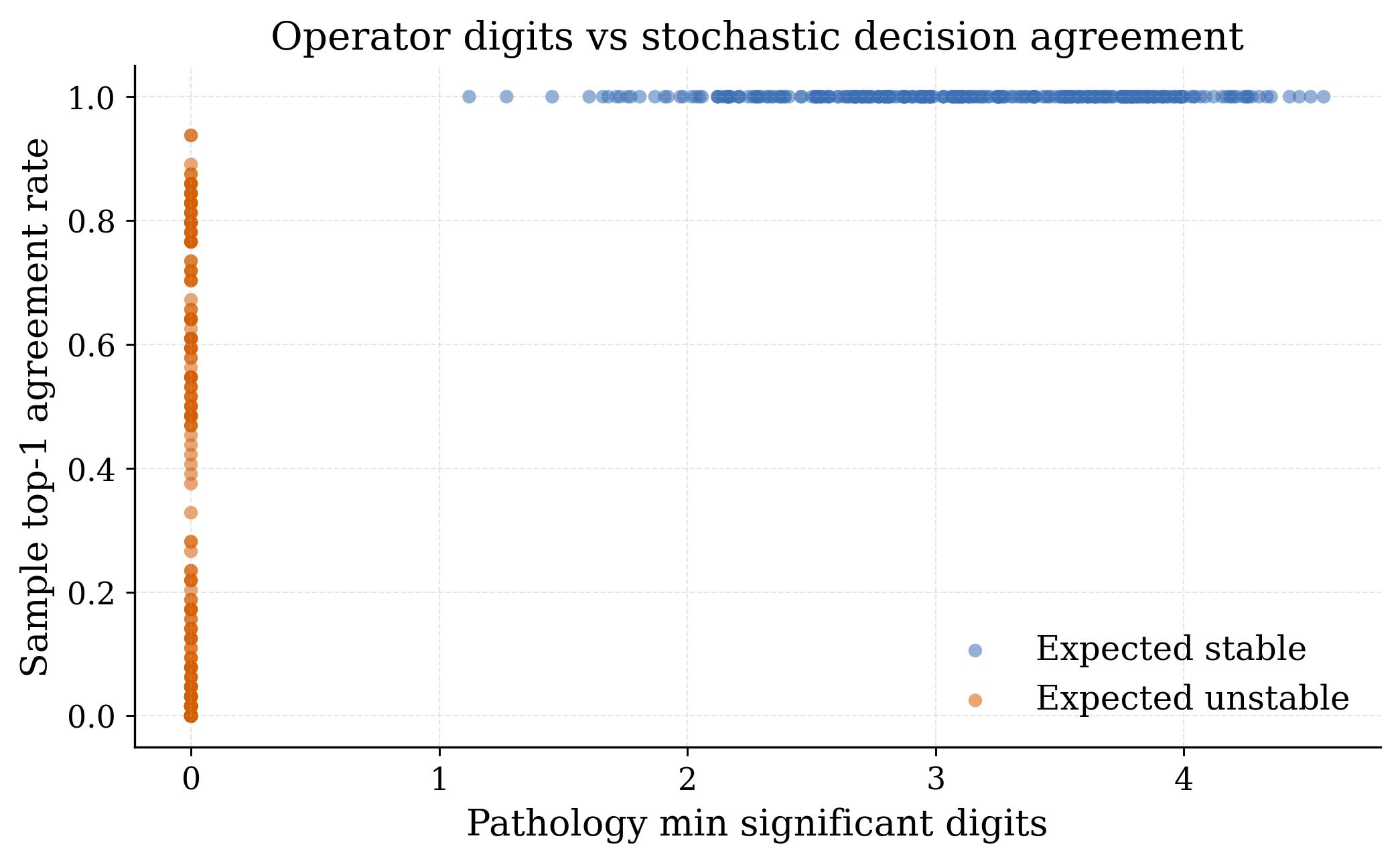}}
\caption{Controlled neural-pathology study on Fashion-MNIST. Each model differs only in the inserted operator $\Pi$ in Eq.~\eqref{eq:pathology-mlp}. Stable controls are shown alongside cancellation, near-constant normalization, and near-tie attention pathologies. The figure compares operator-level significant digits with output-level robustness measured by the agreement of the three stochastic class predictions.}
\label{fig:pathology_study_fashionmnist}
\end{figure}

\begin{figure}
\centering
\subfloat[Epoch-wise pathology-operator significant digits.]{
  \includegraphics[width=0.48\linewidth]{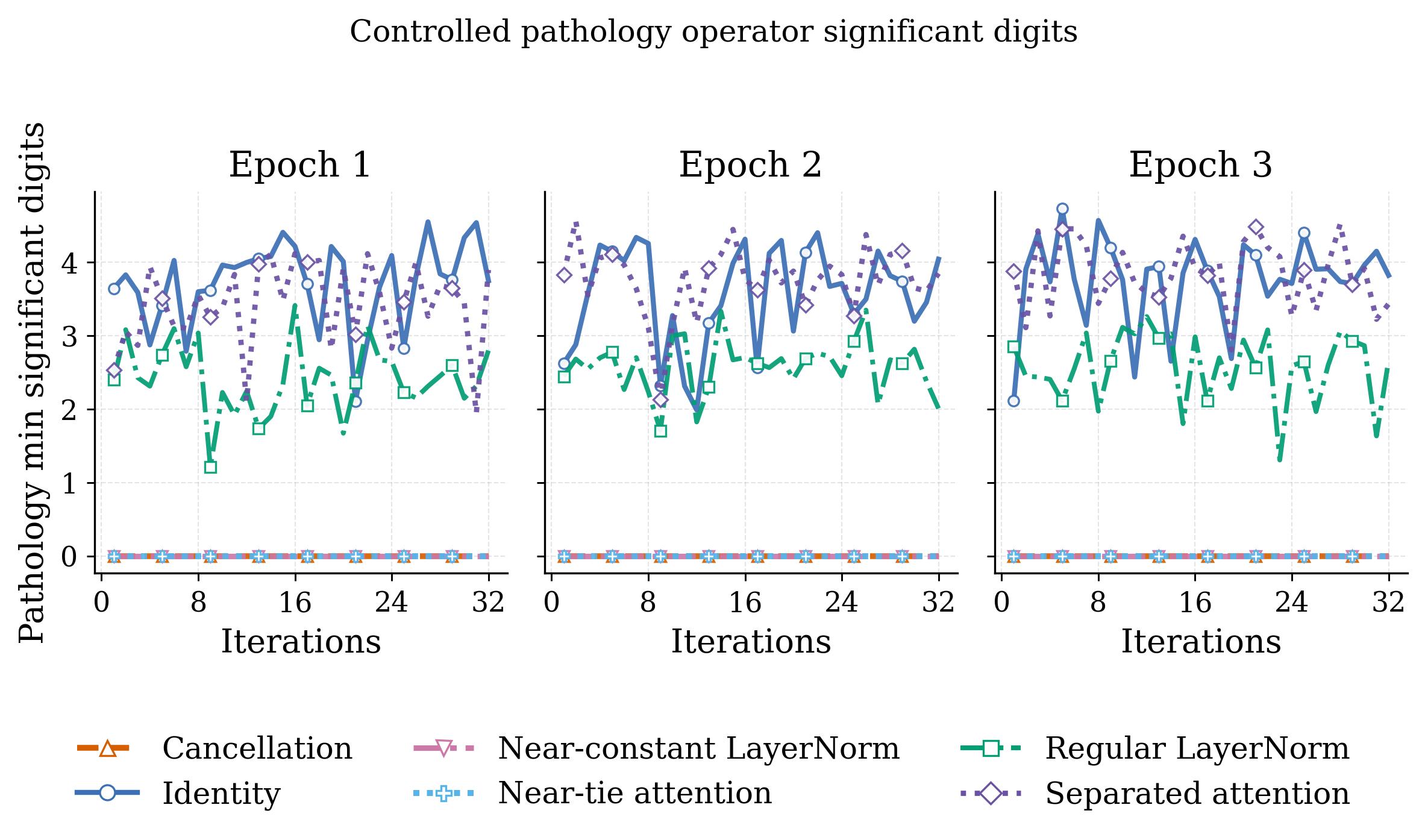}}
\subfloat[Mean digit ranking by inserted operator.]{
  \includegraphics[width=0.48\linewidth]{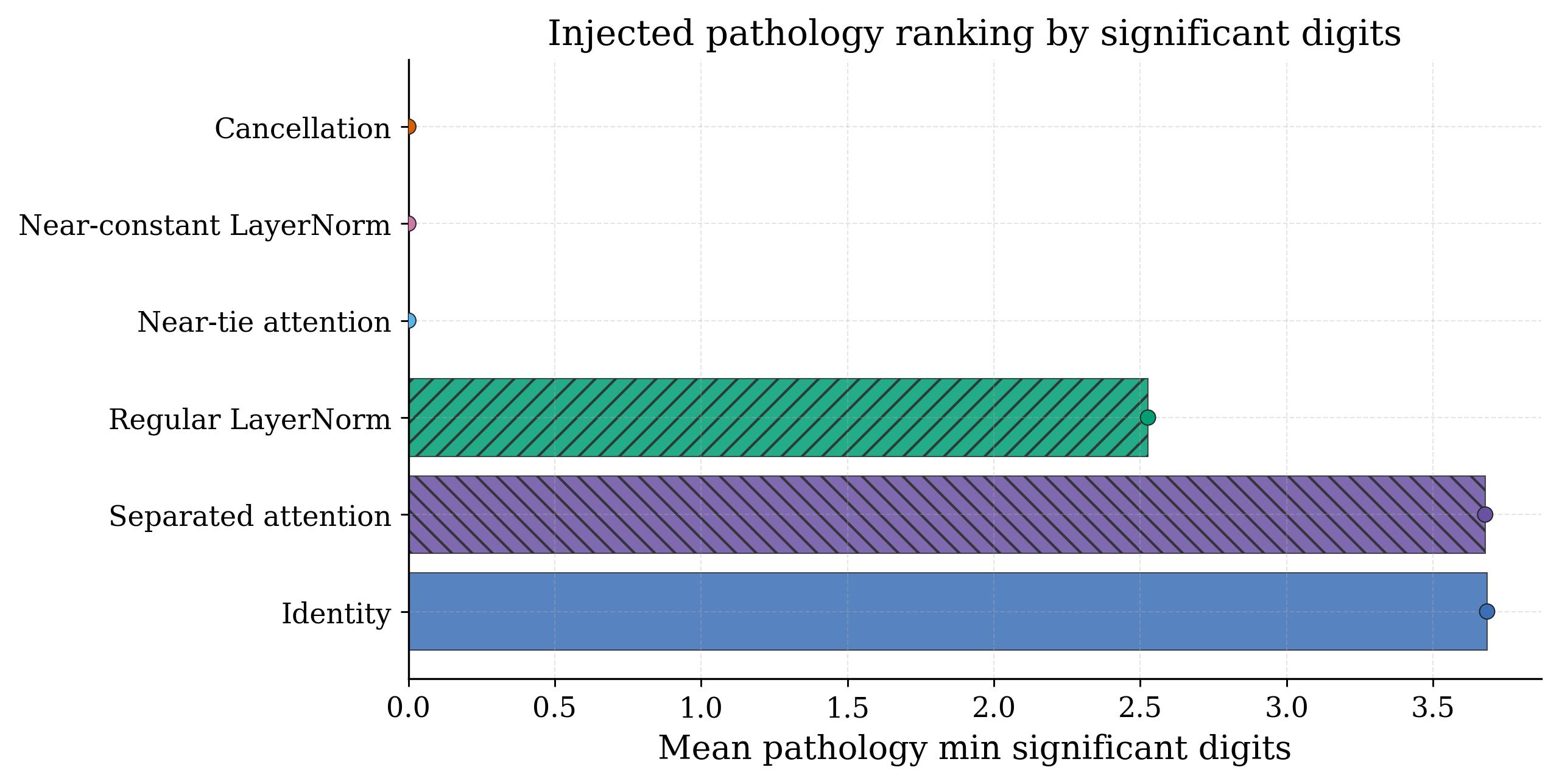}}\\
\subfloat[Task accuracy and stochastic decision agreement.]{
  \includegraphics[width=0.51\linewidth]{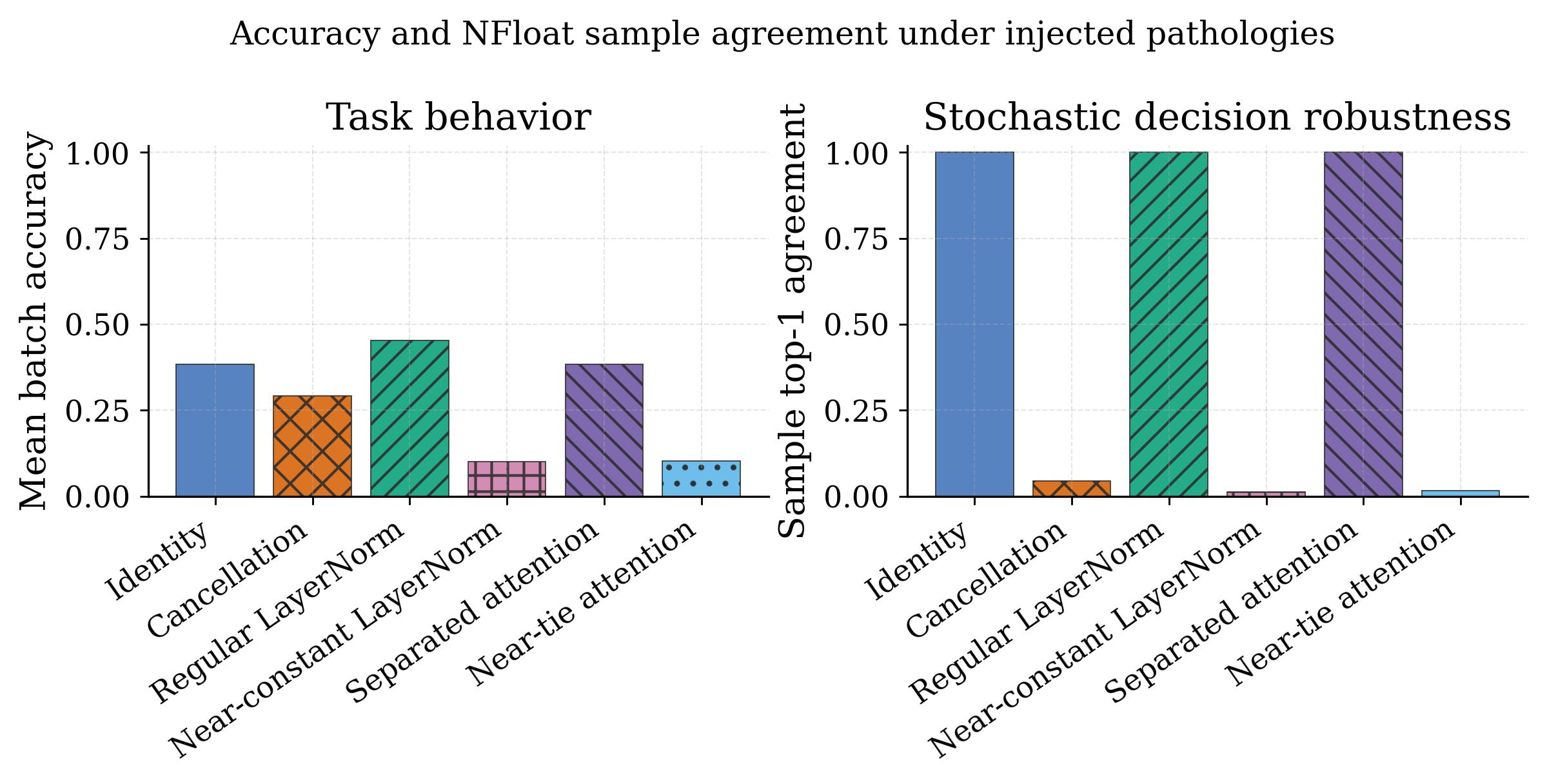}}
\subfloat[Operator digits versus decision agreement.]{
  \includegraphics[width=0.42\linewidth]{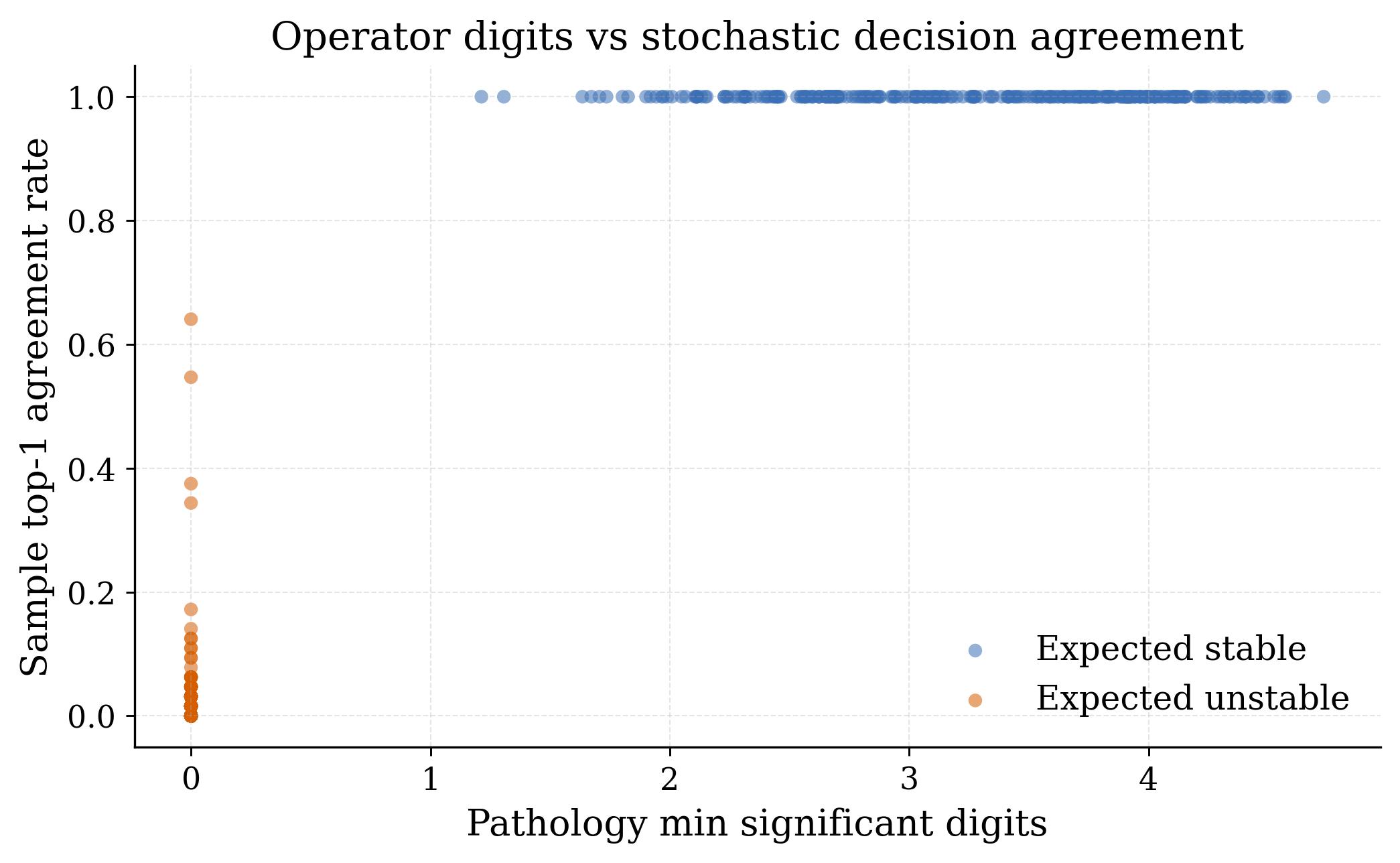}}
\caption{Controlled neural-pathology study on CIFAR-10. Each model differs only in the inserted operator $\Pi$ after the fixed convolutional front end in Eq.~\eqref{eq:pathology-cnn}. Stable controls are shown alongside cancellation, near-constant normalization, and near-tie attention pathologies. The figure compares operator-level significant digits with output-level robustness measured by the agreement of the three stochastic class predictions.}
\label{fig:pathology_study_cifar}
\end{figure}

Figures~\ref{fig:pathology_study_fashionmnist} and~\ref{fig:pathology_study_cifar} show that the injected numerical hazards are localized by the operator reports. Across the MLP Fashion-MNIST run and the CNN CIFAR-10 run, the identity and separated-attention controls retain the highest pathology-operator digit levels and high stochastic top-1 agreement. Regular LayerNorm typically reports lower digits than these two controls because normalization involves centering and rescaling, but it remains a stable control in the sense that the stochastic class predictions are still largely consistent. This distinction is important: moderate digit loss inside an operator is not automatically a decision-level failure when the downstream logit margins remain sufficiently separated.

The three unstable constructions produce a different pattern. The cancellation and near-tie attention operators are designed to remove significant residual information through large-offset cancellation, while the near-constant LayerNorm operator collapses the hidden variance before normalization. In both image tasks, these operators are ranked among the lowest-digit reports, and their batches show weaker sample-level decision robustness than the stable controls. The near-constant LayerNorm case can be partially masked at the output: the operator itself may be numerically unreliable, but only batches with sufficiently small downstream margins change class across stochastic samples. The near-tie attention construction is the most direct decision-level stress test because the nearly tied gate and the cancelling values in Eq.~\eqref{eq:pathology-neartie-att} allow local arithmetic uncertainty to propagate into the classifier decision.

Panel~(d) provides the most direct link between the method's local estimator and the task-level effect. Points in the upper-right region correspond to operators whose minimum significant digits are sufficient for the three stochastic samples to select the same class. Points in the lower-left region correspond to batches where the injected operator loses nearly all reliable digits and the stochastic samples disagree. Points with low digits but high agreement, mainly from the near-constant LayerNorm case, indicate that local instability can be absorbed by later layers when the representative prediction has enough margin. This behavior is consistent with the stability analysis in Section~\ref{sec:method}: CESTAC significant digits measure local numerical reliability, whereas the final decision also depends on conditioning and margin in the subsequent computation. The experiment therefore supports the intended use of the framework as a localization tool: low significant digits identify where numerical reliability is lost, and the joined decision metrics indicate whether that loss is merely local or reaches the model output.

\subsection{Controlled Sequence-Operator Pathology Study}
\label{sec:exp-controlled-sequence-pathology}

The above pathology experiment evaluates injected numerical hazards in a vector-valued hidden representation. To test the same diagnostic principle on a different data modality and a different group of polluted operators, we also construct a controlled sequence-classification experiment on AG News \citep{zhang2015character}. The task is four-class topic classification from tokenized news headlines and short descriptions. The model is intentionally compact: each token $w_{\ell}$ is mapped to an embedding $e_{\ell}$, projected to a hidden representation, transformed by a controlled sequence operator $\Omega$, mean pooled over the sequence, and classified:
\begin{align}
e_{\ell} &= E[w_{\ell}], \qquad
h_{\ell} = \rho(W_e e_{\ell}+b_e), \quad \ell=1,\ldots,L,
\label{eq:agnews-token-hidden}\\
R &= \Omega(H), \qquad
r = \frac{1}{L}\sum_{\ell=1}^{L} R_{\ell}, \qquad
z = W_o r+b_o,
\label{eq:agnews-sequence-classifier}
\end{align}
where $H=(h_1,\ldots,h_L)$ and $\rho$ is ReLU. As in the preceding pathology experiment, the deterministic model is optimized in FP32 and a synchronized shadow model replays each mini-batch using \xyy{the CESTAC-style arithmetic-level mode} in FP32 precision. The shadow pass records operator significant digits and stochastic decision metrics, but it does not update the deterministic optimizer state.

The controlled operator $\Omega$ is chosen from six sequence-level formulations. Three act as stable controls:
\begin{align}
\Omega_{\mathrm{id}}(H) &= H,
\label{eq:agnews-identity}\\
\Omega_{\mathrm{gate}}^{\mathrm{s}}(H)
&= H\odot \sigma(H),
\label{eq:agnews-smooth-gate}\\
\Omega_{\mathrm{att}}^{\mathrm{s}}(H)_{\ell}
&= L\,a_{\ell} h_{\ell}, \qquad
a_{\ell}
=
\frac{\exp(s_{\ell}-m)}
{\sum_{j=1}^{L}\exp(s_j-m)},
\label{eq:agnews-stable-attention}
\end{align}
where $s_{\ell}=\operatorname{mean}(h_{\ell})$ and $m=\max_j s_j$. Equation~\eqref{eq:agnews-stable-attention} uses the translation-invariant softmax form, so the exponentials remain in a bounded range. The three polluted operators inject different numerical hazards:
\begin{align}
\Omega_{\mathrm{can}}^{\mathrm{u}}(H)
&=(H+10^8)-10^8,
\label{eq:agnews-token-cancellation}\\
\Omega_{\mathrm{exp}}^{\mathrm{u}}(H)_{\ell}
&=
L\,
\frac{\exp(160s_{\ell}+96)}
{\sum_{j=1}^{L}\exp(160s_j+96)}
h_{\ell},
\label{eq:agnews-naive-exp-attention}\\
\Omega_{\mathrm{rec}}^{\mathrm{u}}(H)
&=
\epsilon\,
\frac{H}{(\bar{H}+\epsilon)-\bar{H}},
\qquad
\epsilon=2^{-24}.
\label{eq:agnews-reciprocal-gate}
\end{align}
The cancellation operator in Eq.~\eqref{eq:agnews-token-cancellation} tests whether the framework detects loss of the small token-level residual after a large offset is introduced and removed. The naive attention operator in Eq.~\eqref{eq:agnews-naive-exp-attention} deliberately removes the translation-invariance safeguard and scales the logits so that overflow or nearly singular normalization can occur. The reciprocal gate in Eq.~\eqref{eq:agnews-reciprocal-gate} is algebraically close to the identity in exact arithmetic, but it evaluates a division by an almost-zero quantity before rescaling, thereby probing unstable division inside a sequence operator. These polluted operators are distinct from the preceding experiment: they target token-wise cancellation, exp-normalization over sequence positions, and reciprocal gating rather than hidden-vector LayerNorm or near-tie attention on image features.

For each mini-batch $t$, the instrumented operator emits $D_{\Omega,t}^{\mathrm{avg}}$ and $D_{\Omega,t}^{\min}$, the average and minimum significant digits over the stochastic output tensor. The same decision-level metric $A_t$ in Eq.~\eqref{eq:sample-top1-agreement} is used to measure whether the three stochastic realizations predict the same AG News class. Thus the experiment tests a falsifiable localization property: a polluted operator should appear among the lowest-digit reports, while stable controls should retain higher digit estimates and stronger sample-level agreement.

\begin{figure}
\centering
\subfloat[Epoch-wise sequence-operator significant digits.]{
  \IfFileExists{agnews_pathology_digits_epoch_subplots.jpg}
  {\includegraphics[width=0.48\linewidth]{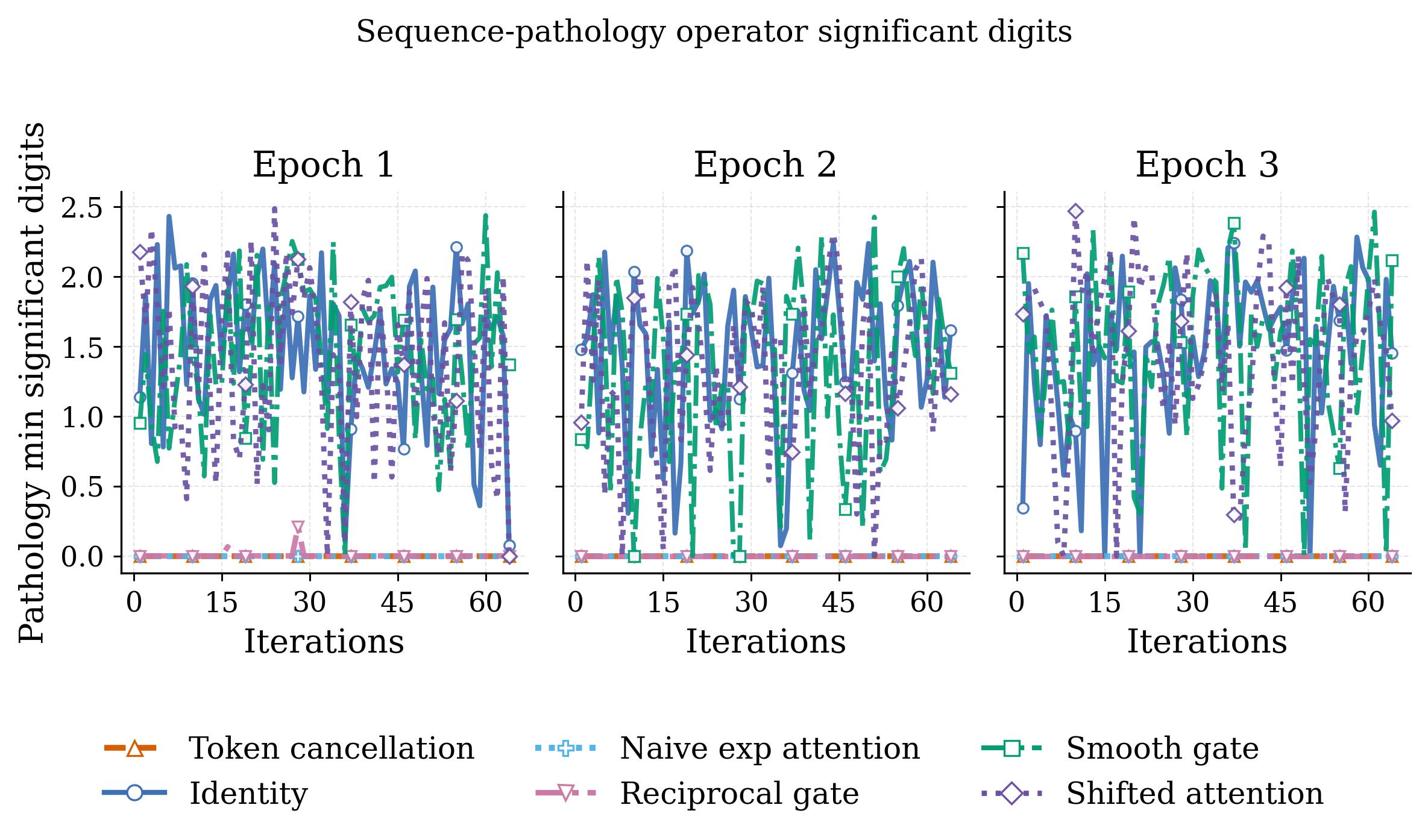}}
  {\fbox{\parbox[c][0.22\textheight][c]{0.43\linewidth}{\centering AG News digit subplot placeholder}}}}
\subfloat[Mean digit ranking by sequence operator.]{
  \IfFileExists{agnews_pathology_digit_ranking.jpg}
  {\includegraphics[width=0.48\linewidth]{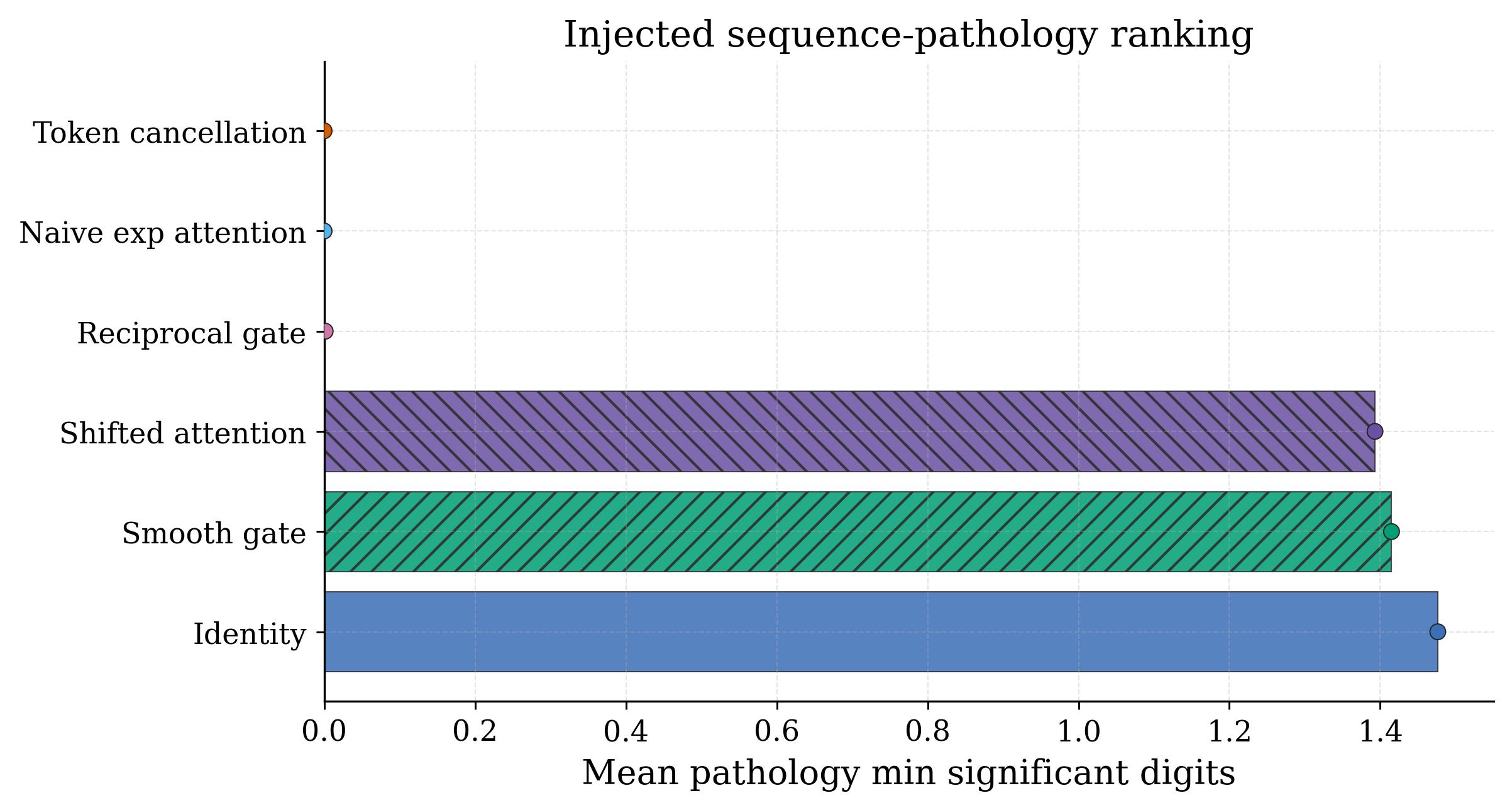}}
  {\fbox{\parbox[c][0.22\textheight][c]{0.43\linewidth}{\centering AG News digit ranking placeholder}}}}\\
\subfloat[Task accuracy and stochastic decision agreement.]{
  \IfFileExists{agnews_pathology_decision_robustness.jpg}
  {\includegraphics[width=0.51\linewidth]{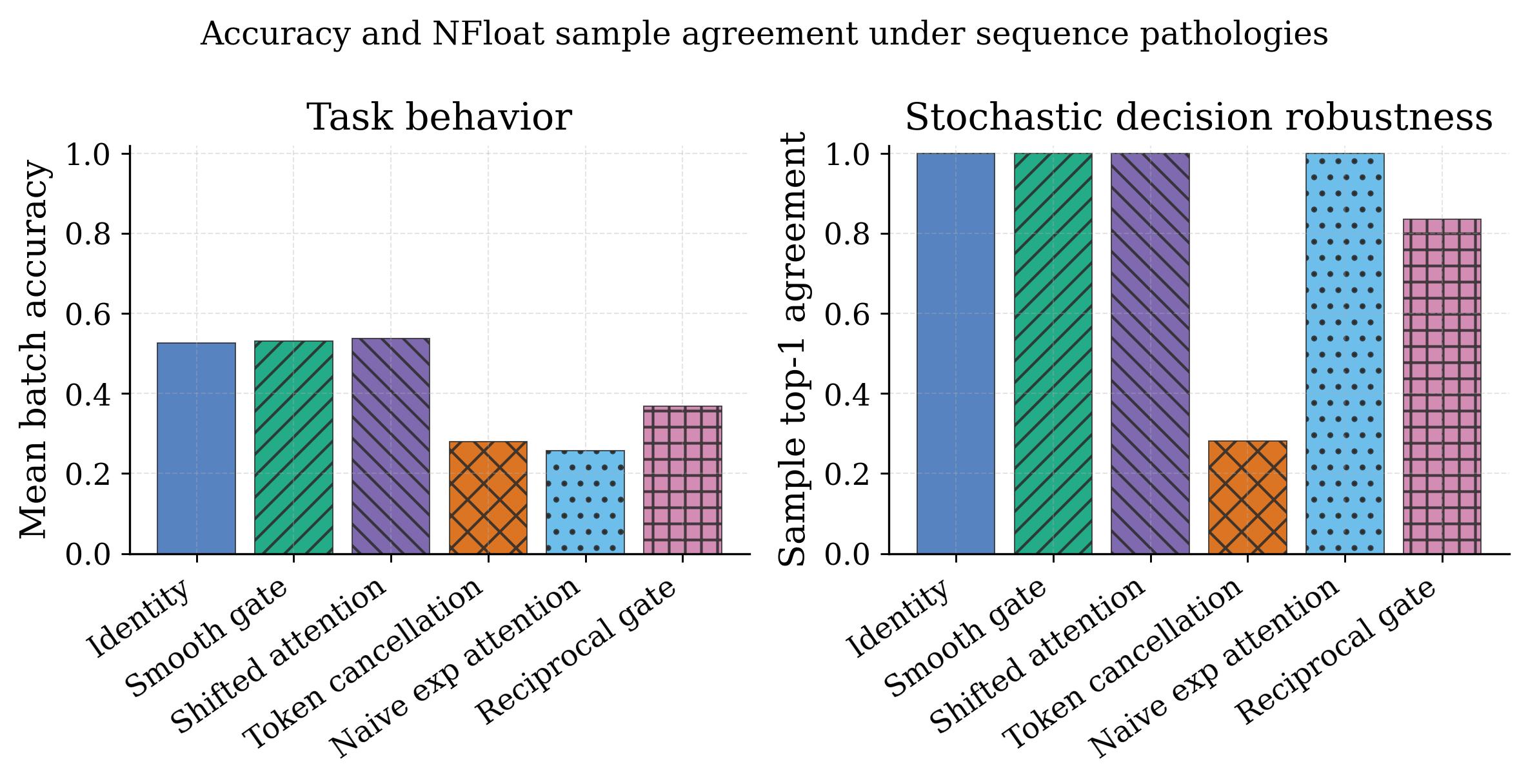}}
  {\fbox{\parbox[c][0.22\textheight][c]{0.46\linewidth}{\centering AG News robustness placeholder}}}}
\subfloat[Operator digits versus decision agreement.]{
  \IfFileExists{agnews_pathology_digits_vs_decision_agreement.jpg}
  {\includegraphics[width=0.42\linewidth]{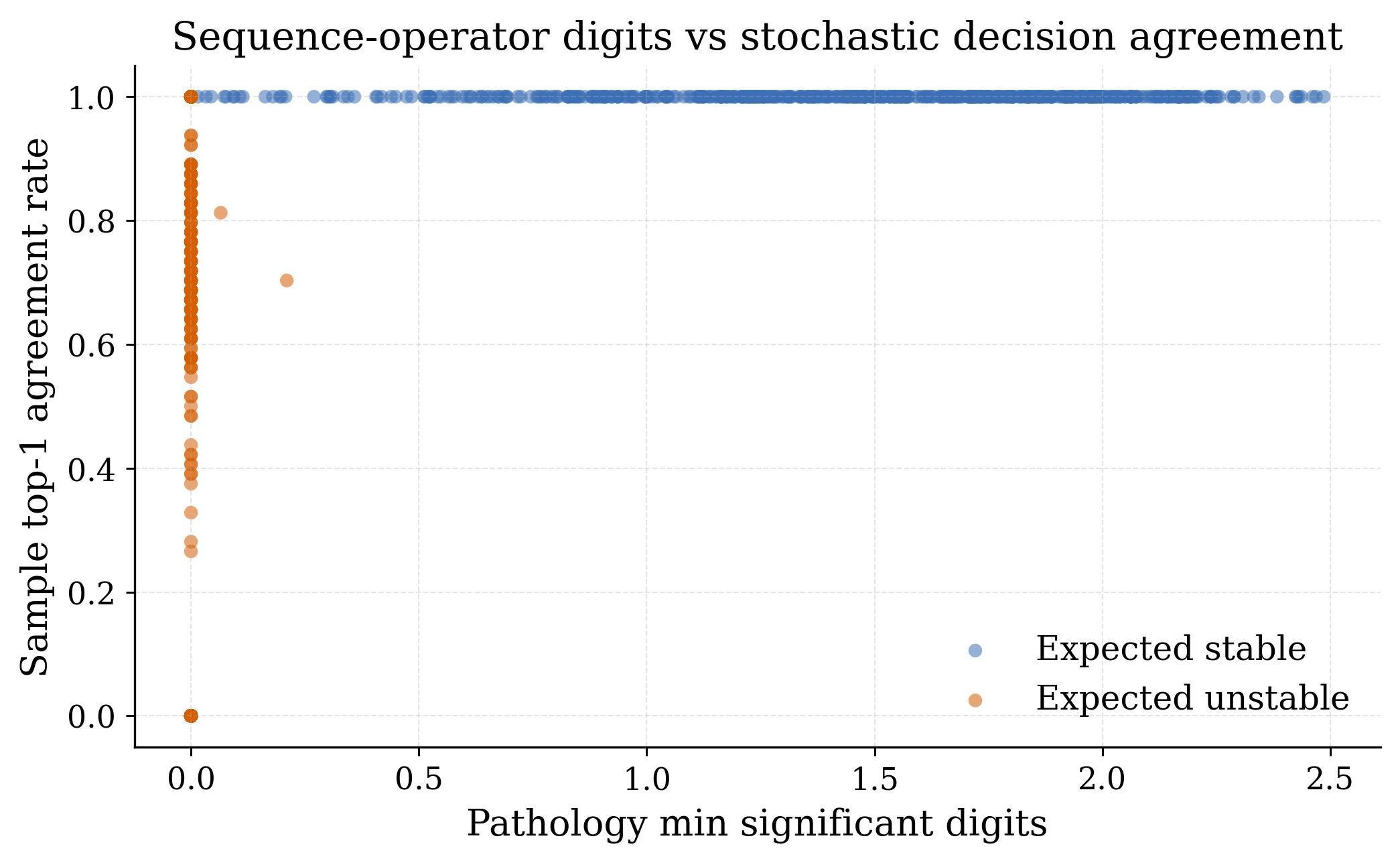}}
  {\fbox{\parbox[c][0.22\textheight][c]{0.38\linewidth}{\centering AG News digit-agreement placeholder}}}}
\caption{Controlled sequence-operator pathology study on AG News. The model architecture in Eq.~\eqref{eq:agnews-sequence-classifier} is fixed while the sequence operator $\Omega$ is replaced by stable controls or polluted operators. The panels mirror the Fashion-MNIST pathology study, but the numerical hazards are injected into token-level sequence operators: large offset cancellation, unshifted exponential attention, and reciprocal gating through a nearly zero statistic. }
\label{fig:agnews_pathology_study}
\end{figure}

This experiment serves a different role from ordinary task benchmarking. If two sequence classifiers have comparable deterministic accuracy, that alone does not show whether their internal operators are numerically reliable. The expected outcome is instead structural: the shifted attention and smooth gate behave like stable controls, while the naive exponential attention and reciprocal gate concentrate low significant digits at the injected operator. When low $D_{\Omega,t}^{\min}$ is also accompanied by a decrease in $A_t$, the local loss of numerical reliability reaches the final class decision. Conversely, batches with low digits but high agreement indicate that the instability is locally detectable but masked by downstream margins. This interpretation is consistent with the preceding pathology study and with the operator-level benchmark in Table~\ref{tab:operator-level-dl-cestac}: CESTAC digits identify where rounding uncertainty is amplified, whereas decision-level metrics indicate whether that uncertainty changes the task output.

\subsection{Activation Functions on Fashion-MNIST}
\label{sec:exp-activation}
Activation functions play a critical role in facilitating information flow and introducing nonlinearity into neural networks. The simulation on activation functions isolates the numerical effect of the hidden activation in a shallow Fashion-MNIST classifier~\citep{xiao2017fashionmnist}. For every activation $\phi$, the model is
\begin{equation}
z = W_2\,\phi\!\left(W_1\operatorname{vec}(x)+b_1\right)+b_2,
\qquad
W_1\in\mathbb{R}^{256\times784},
\quad
W_2\in\mathbb{R}^{10\times256}.
\label{eq:activation-mlp}
\end{equation}

We instantiate nine otherwise identical models with ReLU, Leaky ReLU, PReLU, ELU, GELU, SiLU, Swish, $\tanh$, and sigmoid. Reinitializing with the same seed and reconstructing the same shuffled loader for each activation implements a controlled-variable comparison{—}only $\phi$ changes between runs.

Each model is trained on the Fashion-MNIST training split for three epochs using Adam, learning rate $10^{-3}$, batch size 128, hidden width 256, cross-entropy loss, and fixed random state. Following every deterministic optimizer step, the \xyy{synchronized CESTAC-style arithmetic-level shadow model} evaluates the same batch. We measure deterministic and representative batch accuracy and record the significant-digit report emitted by the hidden activation. The purpose is to determine whether activation choice changes the numerical reliability of otherwise matched hidden representations, and whether such changes evolve systematically over training iterations.

\begin{figure}[ht]
\centering
\subfloat[Hidden-activation significant digits.]{\includegraphics[width=0.48\linewidth]{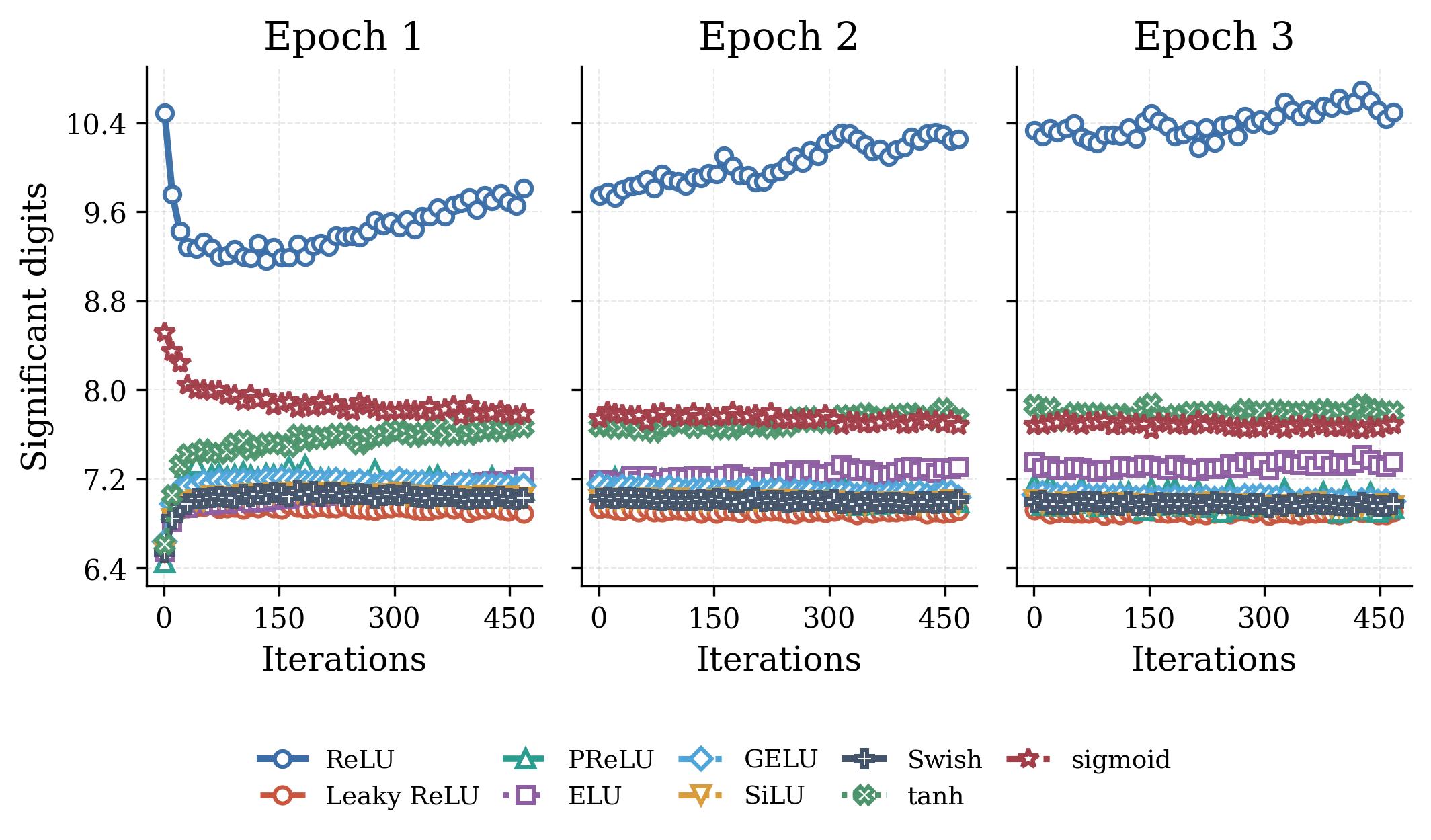}}
\subfloat[Hidden-activation minimum significant digits.]{\includegraphics[width=0.48\linewidth]{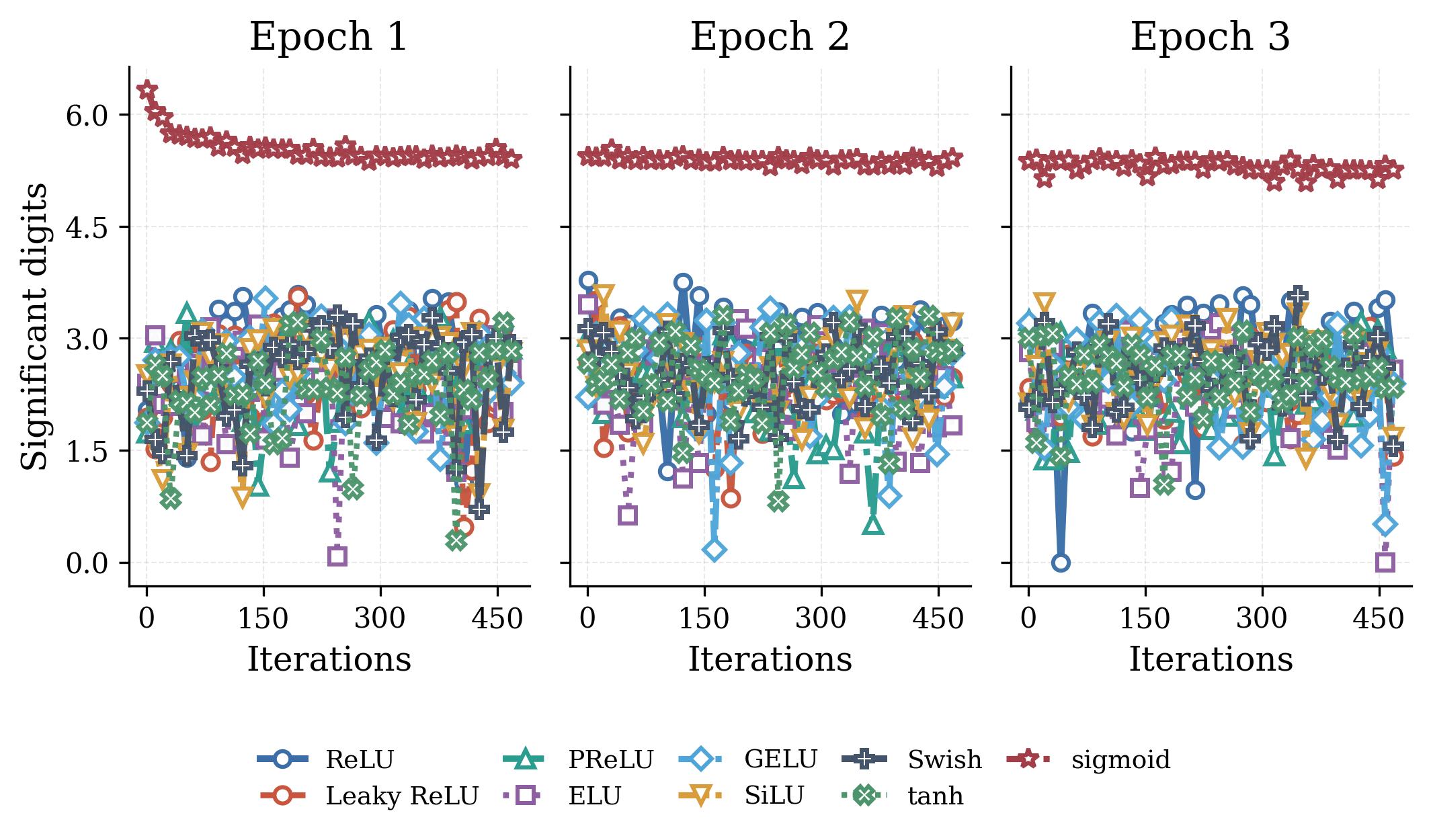}}

\subfloat[Classification accuracy.]{\includegraphics[width=0.48\linewidth]{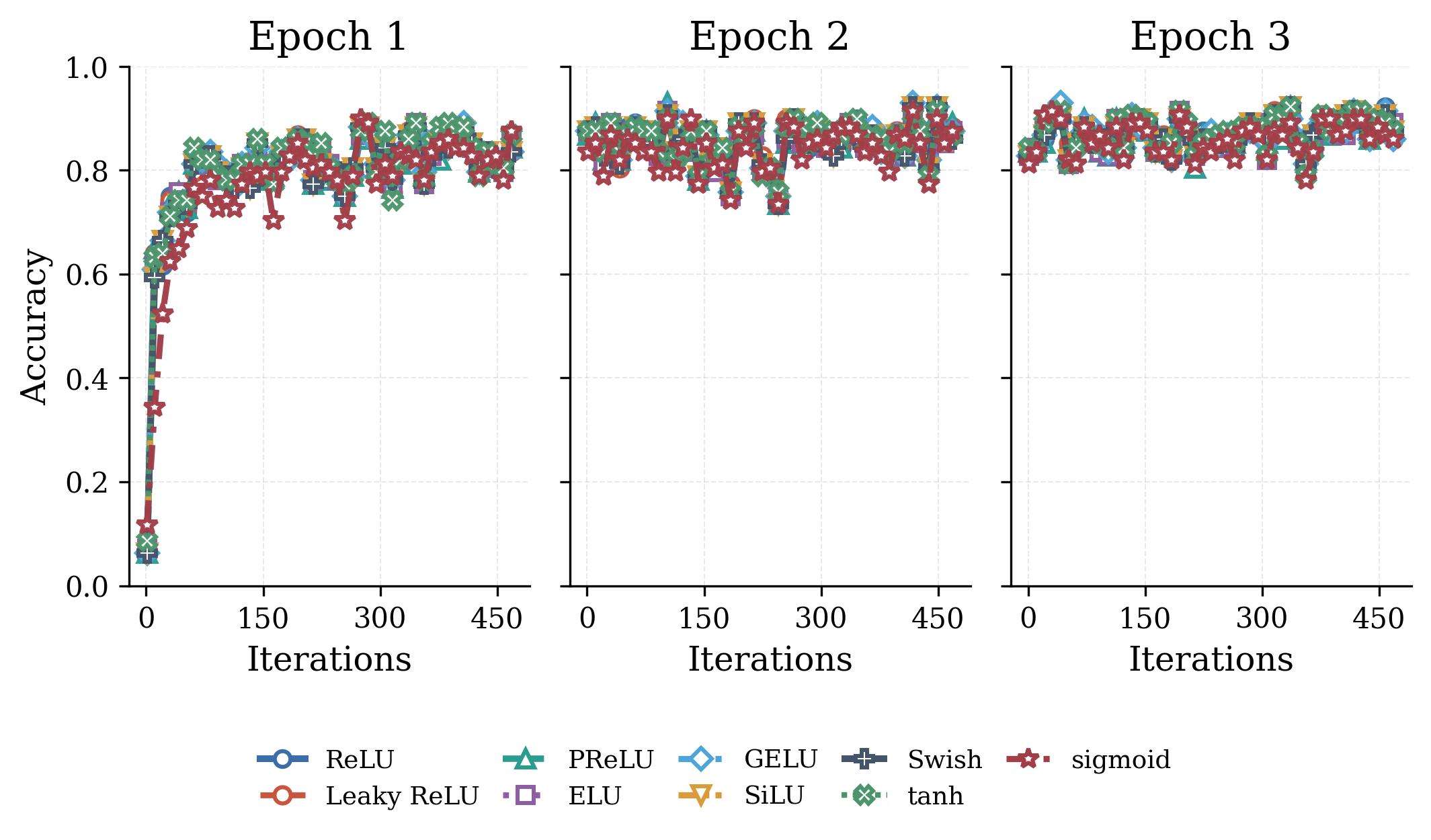}}
\caption{Hidden-activation reliability and batch accuracy on Fashion-MNIST. The first two panels report the average and minimum significant digits of the instrumented activation output, separated by epoch, while the final panel reports the corresponding deterministic batch accuracy. For visual clarity, each curve shows a 10\% evenly spaced rendering subset selected by the plotting script, while the full iteration range is preserved on the horizontal axis.}
\label{fig:Fashion-MNIST_activation}
\end{figure}

Figure~\ref{fig:Fashion-MNIST_activation} shows that task accuracy alone does not explain the numerical reliability, and it does not exhibit a significant relationship with the numerical stability of each activation function. All neural networks with distinct activation functions achieve comparable classification accuracy  after training with three epochs. ReLU yields the highest-activation-level digit estimates, averaging about 10 significant digits across the run and increasing from roughly 9.2 to 10.5 digits across the three epochs. This phenomenon matches the theoretical analysis that ReLU is inherently stable with respect to perturbations: as indicated in \citep{reynacruz:hal-05367563}, the output error never exceeds the input error, and ReLU is a preferred robust choice for architectures that is sensitive to noise or quantization effect. Sigmoid and PReLU achieve comparable final significant digits between approximately 7.6 and 7.8 digits. The experiment separates two effects that are often conflated: several activations are equally useful for the classifier, yet they expose different stochastic sensitivity at the hidden representation.

\subsection{Pre- and Post-Normalization Reliability}
\label{sec:exp-normalization}

The normalization experiment uses the same Fashion-MNIST data path and optimization protocol but replaces the hidden block by
\begin{equation}
h=W_1\operatorname{vec}(x)+b_1,
\qquad
\widetilde{h}=\mathcal{N}(h),
\qquad
z=W_2\operatorname{ReLU}(\widetilde{h})+b_2,
\label{eq:normalization-mlp}
\end{equation}
where $\mathcal{N}$ is BatchNorm, LayerNorm, or RMSNorm. Similarly to the simulation on activation function. All three models have width 256 and are trained for three epochs with Adam, learning rate $10^{-3}$, batch size 128, cross-entropy loss, and fixed random state. BatchNorm uses its learned affine parameters and batch statistics. LayerNorm normalizes the hidden feature dimension, and RMSNorm applies
\begin{equation}
\operatorname{RMSNorm}(h)
=
g\odot
\frac{h}{\sqrt{256^{-1}\sum_{j=1}^{256}h_j^2+10^{-6}}}.
\end{equation}

For each batch, the experiment captures one report at the output of the first linear transformation and another at the output of $\mathcal{N}$. If $C(\cdot)$ denotes the CESTAC digit estimator applied elementwise, the paired quantities are therefore
\begin{equation}
D_{\mathrm{pre}}=\operatorname{agg} C(h),
\qquad
D_{\mathrm{post}}=\operatorname{agg} C(\mathcal{N}(h)),
\label{eq:prepost-normalization-digits}
\end{equation}
with \texttt{avg\_digits}, \texttt{min\_digits}, or \texttt{max\_digits} as the selected aggregation. This within-batch pairing is intended to isolate how each normalization transforms numerical reliability, rather than comparing normalization only through final classification accuracy.

\begin{figure}[ht]
    \centering
    \subfloat[Pre-normalization significant digits.]{\includegraphics[width=0.48\linewidth]{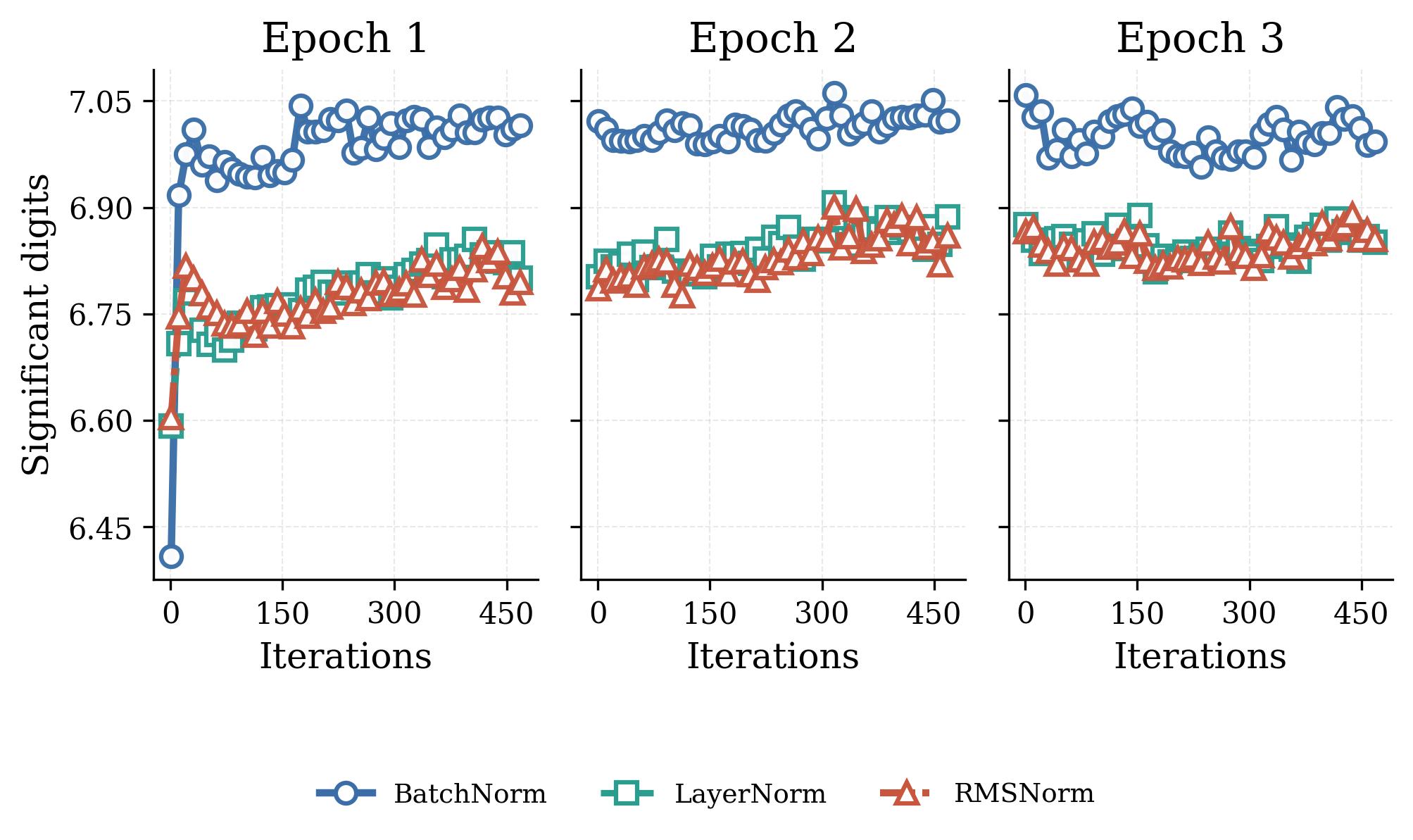}}
    \subfloat[Pre-normalization minimum significant digits.]{\includegraphics[width=0.48\linewidth]{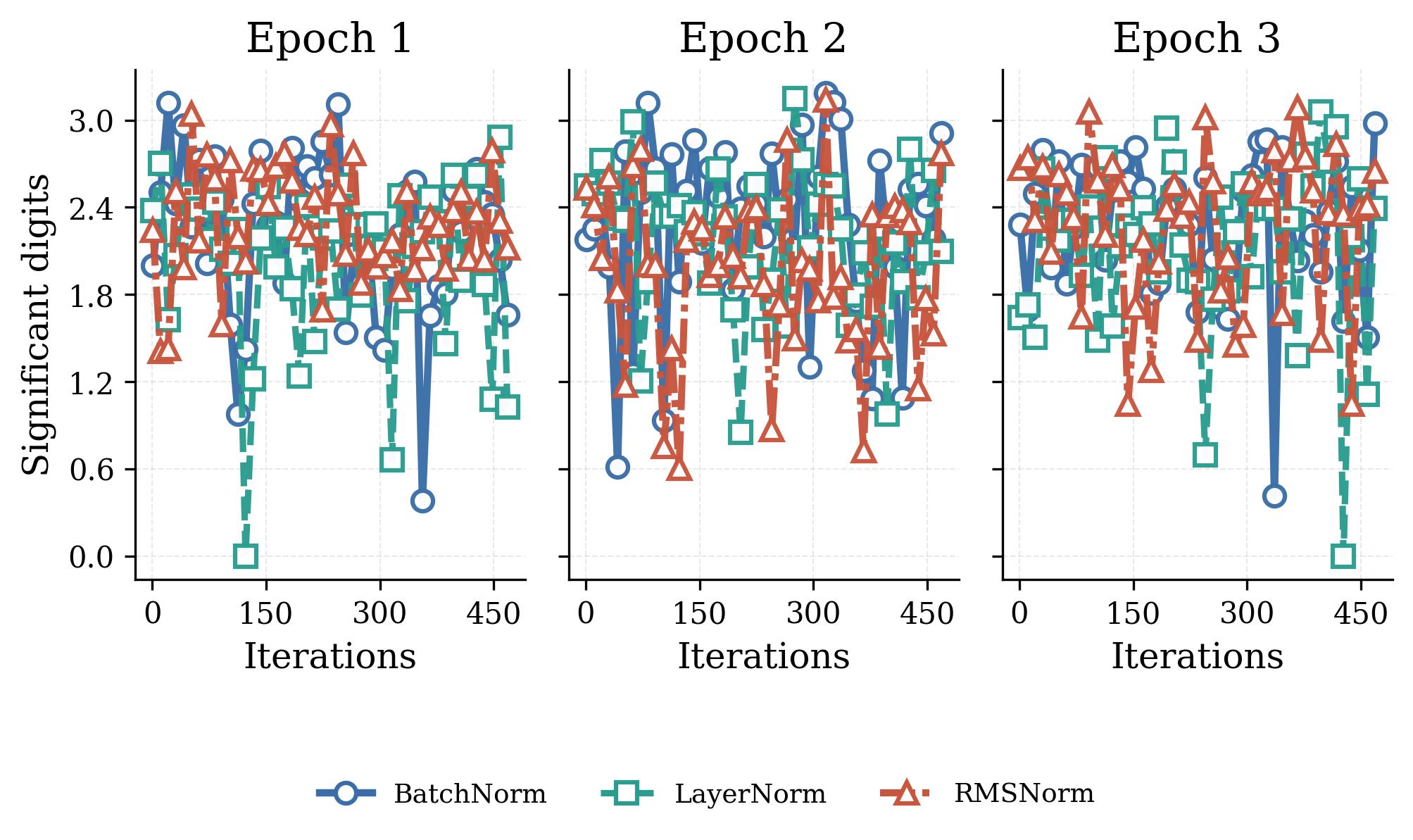}}\\
    \subfloat[Post-normalization significant digits.]{\includegraphics[width=0.48\linewidth]{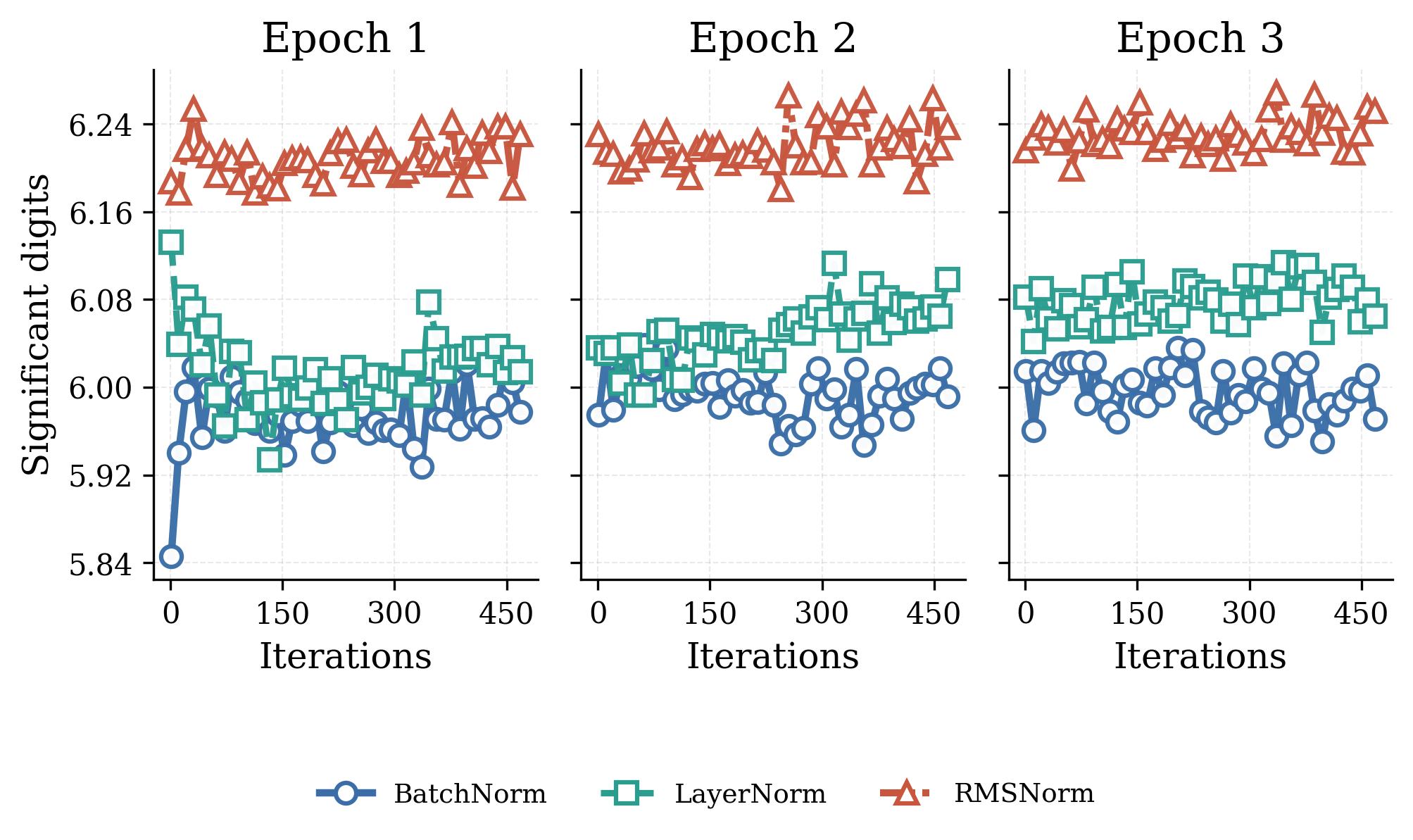}}
    \subfloat[Post-normalization minimum significant digits.]{\includegraphics[width=0.48\linewidth]{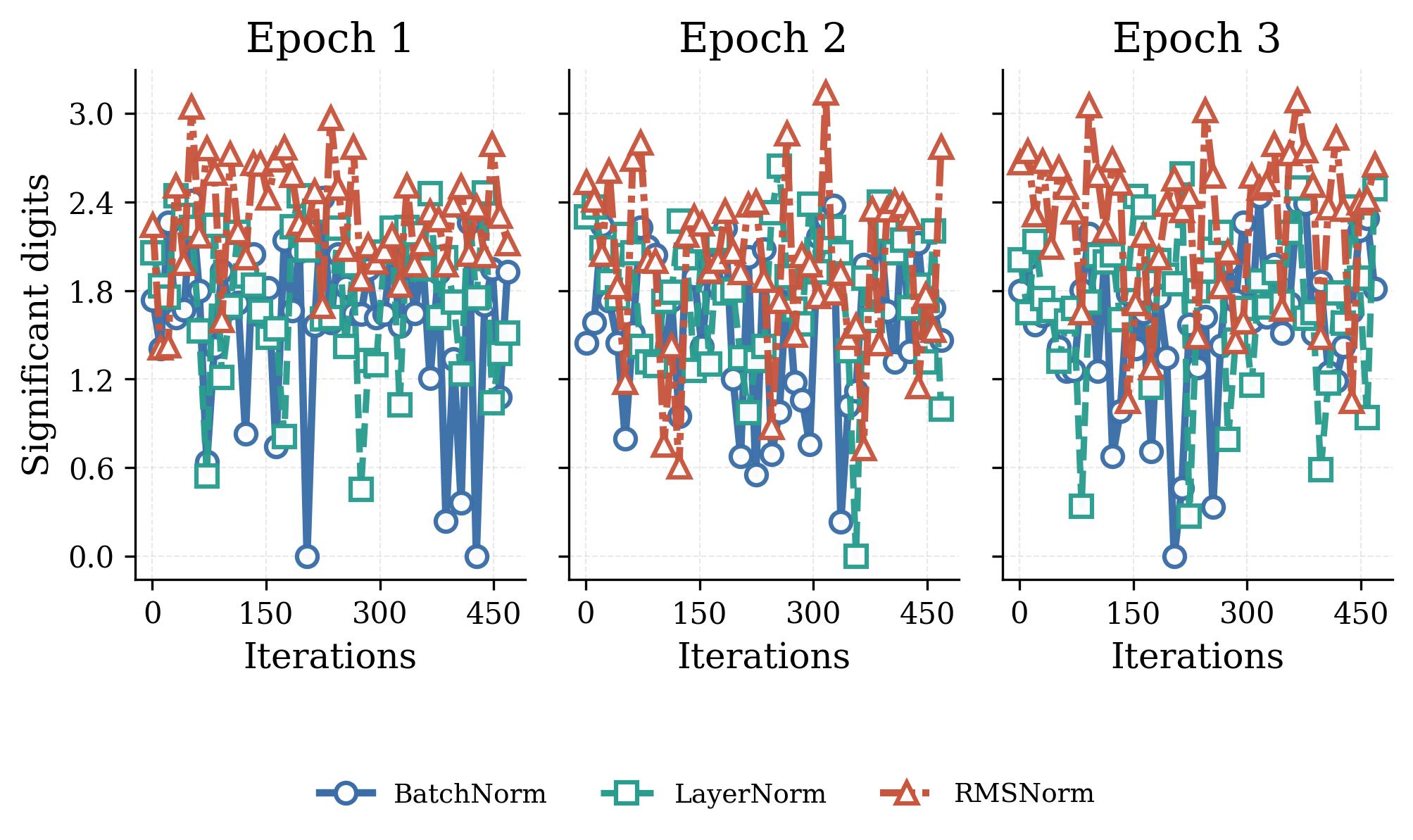}}\\
    \subfloat[Classification accuracy.]{\includegraphics[width=0.48\linewidth]{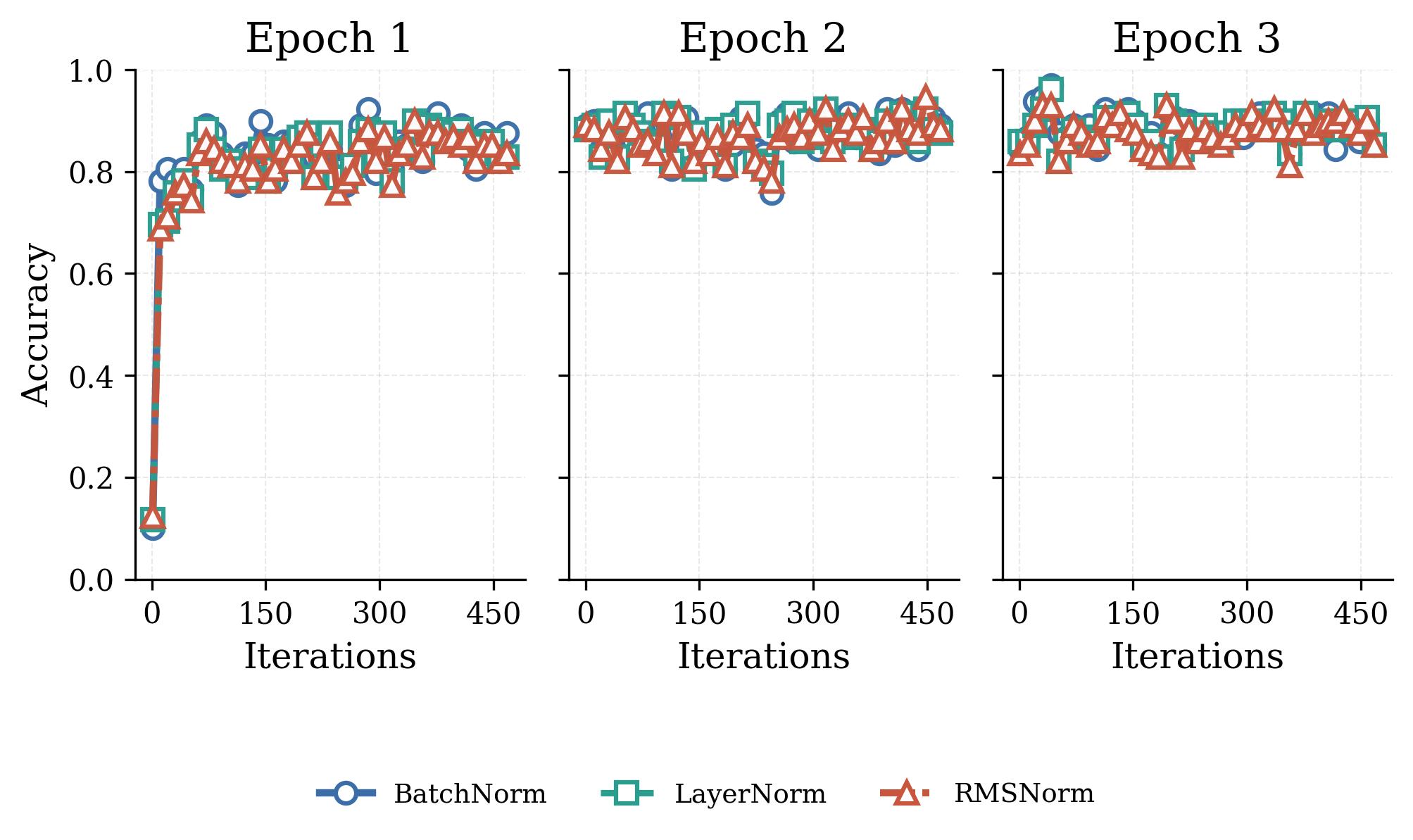}}
    \caption{Pre- and post-normalization reliability on Fashion-MNIST. The paired digit plots compare the hidden linear output before normalization with the normalized activation presented to the classifier. For visual clarity, each curve shows a 10\% evenly spaced rendering subset selected by the plotting script, while the full iteration range is preserved on the horizontal axis.}
    \label{fig:Fashion-MNIST}
\end{figure}

The paired normalization traces in Fig.~\ref{fig:Fashion-MNIST} indicate that normalization changes the numerical state of the hidden representation even when the final classification behavior remains similar. Averaged over the run, the pre-normalization hidden linear output's average digits lie around 6.45$\sim$7.05 significant digits, while the post-normalization output's average digits lie around 5.84$\sim$6.30 digits.  Thus, normalization does not merely rescale the representation from a learning perspective, but also introduces a distinct numerical transformation whose stochastic output exhibits lower estimated significant digits.

BatchNorm has the largest average reduction in this setting, from about 7.05 to 6.00 digits, whereas LayerNorm and RMSNorm show slightly smaller but still systematic reductions. This is consistent with the fact that BatchNorm estimates batch-dependent centering and variance statistics, so its output depends on reductions over a small mini-batch and on the subtraction of the batch mean. RMSNorm exhibits the highest number of significant digits across iterations and epochs, which is attributed to its elimination of explicit mean shift that might trigger a cancellation disaster. Since the accuracies of the three normalized classifiers are close, the observed digit loss is best interpreted as an operator-level reliability effect rather than as a failure of the classifier. 

From the minimum number of significant digits across iterations and epochs, three normalizations exhibit comparable low-tail behavior{---}similar trends and magnitudes. This further supports the claim that the observed loss of significant digits should therefore be interpreted as an operator-level numerical reliability effect rather than as direct evidence of degraded predictive performance.  To summarize, normalization changes where and how numerical uncertainty is injected into the hidden representation, but this uncertainty is not necessarily amplified enough to alter the final Fashion-MNIST accuracy in this shallow classifier.

% \subsection{Discussions}
% 应该是在没有极大伤害到operator的稳定性的情况下，保持输出accuracy的不变。
% The accuracy of the numerical examples in this section is described by the number of significant digits. Our first simulation demonstrates that the stable operators or numerical examples achieve roughly the significant digit of working precision (single or double) while the unstable ones achieve far below the threshold. In the last two simulations in deep learning tasks, the observed significant digits may be attributed to the stability of the operator's formulation and confounding factors. For diffusion-based generative tasks, the same diagnostics primarily reveal timestep-dependent numerical uncertainty and stochastic output dispersion, rather than serving as a direct proxy for reconstruction quality. For discriminative or token-decision tasks, in contrast, the operator’s significant digits are closely associated with decision-level robustness. 

\section{Limitations}

However, due to the intrinsic nature of the CESTAC principle, the computations are required to repeat, which triggers extra overhead of performance at runtime. Although the triplication of each arithmetic operation implies a minimum slowdown of approximately threefold, the practical overhead can be substantially higher because of stochastic-operation handling and numerical-instability detection. In particular, CADNA, the high-performance DSA library, has been reported to incur a 10$\times$$\sim$100$\times$ slowdown relative to standard floating-point arithmetic on both compute-bound and memory-bound benchmarks when DSA self-validation is enabled \cite{6614304, jezequel2020trustworthy}. Since the exact overhead depends on the application, the computer architecture, and the level of instability-detection \cite{Jez2008,eberhart2015high}, for large-scale datasets, we recommend applying our software to a representative calibration subset rather than to the entire data set.

% The computational overhead of CESTAC is not negligible. 

\section{Conclusion}

In this paper, we introduce a software framework called \texttt{noisefloat} for detection and localization of numerical instabilities in numerical simulations and deep learning computations. Our software not only enables numerical validation of a single-pass computation but also provides numerical stability checks during training and inference of neural networks, catering to most practical deep learning deployment applications.  Our experiments demonstrate that \texttt{noisefloat} is able to effectively detect numerically unstable operators. In the future, the study of focus will include reducing budget use while maintaining the effectiveness of current operators' functionalities.

\section{Acknowledgements}

The author would like to thank his advisor for the insightful comments and helpful suggestions.

\bibliographystyle{ACM-Reference-Format}
\bibliography{refs}

@article{pepe2026fuzzy,
  author  = {In{\'e}s Gonzalez Pepe and Hiba Akhaddar and Tristan Glatard and Yohan Chatelain},
  title   = {{Fuzzy PyTorch}: Rapid numerical variability evaluation for deep learning models},
  journal = {Transactions on Machine Learning Research},
  year    = {2026},
  url     = {https://openreview.net/forum?id=0ogq232VGP},
  issn    = {2835-8856},
}

@inproceedings{paszke2019pytorch,
  author    = {Paszke, Adam and Gross, Sam and Massa, Francisco and Lerer, Adam and Bradbury, James and Chanan, Gregory and Killeen, Trevor and Lin, Zeming and Gimelshein, Natalia and Antiga, Luca and Desmaison, Alban and Kopf, Andreas and Yang, Edward and DeVito, Zachary and Raison, Martin and Tejani, Alykhan and Chilamkurthy, Sasank and Steiner, Benoit and Fang, Lu and Bai, Junjie and Chintala, Soumith},
  title     = {{PyTorch}: An imperative style, high-performance deep learning library},
  booktitle = {Advances in Neural Information Processing Systems},
  volume    = {32},
  pages     = {8024--8035},
  publisher = {Curran Associates, Inc.},
  year      = {2019},
}

@inproceedings{abadi2016tensorflow,
  author    = {Abadi, Mart{\'i}n and Barham, Paul and Chen, Jianmin and Chen, Zhifeng and Davis, Andy and Dean, Jeffrey and Devin, Matthieu and Ghemawat, Sanjay and Irving, Geoffrey and Isard, Michael and Kudlur, Manjunath and Levenberg, Josh and Monga, Rajat and Moore, Sherry and Murray, Derek G. and Steiner, Benoit and Tucker, Paul and Vasudevan, Vijay and Warden, Pete and Wicke, Martin and Yu, Yuan and Zheng, Xiaoqiang},
  title     = {{TensorFlow}: A system for large-scale machine learning},
  booktitle = {12th {USENIX} Symposium on Operating Systems Design and Implementation ({OSDI} 16)},
  pages     = {265--283},
  publisher = {{USENIX} Association},
  address   = {Savannah, GA},
  year      = {2016},
}

@article{852391,
  author  = {Parker, D.S. and Pierce, B. and Eggert, P.R.},
  title   = {{Monte Carlo} arithmetic: How to gamble with floating point and win},
  journal = {Computing in Science \& Engineering},
  volume  = {2},
  number  = {4},
  pages   = {58--68},
  year    = {2000},
  doi     = {10.1109/5992.852391},
}

@misc{DBLP:journals/corr/DenisCP15,
  author     = {Christophe Denis and Pablo de Oliveira Castro and Eric Petit},
  title      = {{Verificarlo}: Checking floating point accuracy through {Monte Carlo} arithmetic},
  year       = {2015},
  url        = {https://arxiv.org/abs/1509.01347},
  eprint     = {1509.01347},
  eprinttype = {arXiv},
}

@article{10.1145/3773992,
  author  = {Wen, Zhongzhen and Liu, Hongyu and Zhu, Tingwei and Pan, Minxue and Wang, Shaohua and Lin, Yuanyi and Liu, Kairui and Zhang, Tian and Li, Xuandong},
  title   = {A study of floating-point precision tuning in deep learning operators implementations},
  journal = {{ACM} Transactions on Software Engineering and Methodology},
  year    = {2025},
  doi     = {10.1145/3773992},
}

@article{frechtling2015mcalib,
  author  = {Frechtling, Michael and Leong, Philip H. W.},
  title   = {{MCALIB}: Measuring sensitivity to rounding error with {Monte Carlo} programming},
  journal = {{ACM} Transactions on Programming Languages and Systems},
  volume  = {37},
  number  = {2},
  pages   = {5:1--5:25},
  year    = {2015},
  doi     = {10.1145/2665073},
}

@inproceedings{denis2016verificarlo,
  author    = {Denis, Christophe and de Oliveira Castro, Pablo and Petit, Eric},
  title     = {{Verificarlo}: Checking floating point accuracy through {Monte Carlo} arithmetic},
  booktitle = {{IEEE} 23rd Symposium on Computer Arithmetic ({ARITH})},
  pages     = {55--62},
  publisher = {{IEEE}},
  year      = {2016},
  doi       = {10.1109/ARITH.2016.31},
}

@article{GRAILLAT2019101017,
  author  = {Stef Graillat and Fabienne Jézéquel and Romain Picot and François Févotte and Bruno Lathuilière},
  title   = {Auto-tuning for floating-point precision with discrete stochastic arithmetic},
  journal = {Journal of Computational Science},
  volume  = {36},
  pages   = {101017},
  year    = {2019},
  doi     = {10.1016/j.jocs.2019.07.004},
}

@article{Vignes2004,
  author  = {Vignes, Jean},
  title   = {Discrete stochastic arithmetic for validating results of numerical software},
  journal = {Numerical Algorithms},
  volume  = {37},
  number  = {1},
  pages   = {377--390},
  year    = {2004},
  doi     = {10.1023/B:NUMA.0000049483.75679.ce},
}

@misc{jax2018github,
  author  = {James Bradbury and Roy Frostig and Peter Hawkins and Matthew James Johnson and Yash Katariya and Chris Leary and Dougal Maclaurin and George Necula and Adam Paszke and Jake Vander{P}las and Skye Wanderman-{M}ilne and Qiao Zhang},
  title   = {{JAX}: Composable transformations of {Python}+{NumPy} programs},
  year    = {2018},
  url     = {https://github.com/jax-ml/jax},
  note    = {Software},
  version = {0.3.13},
}

@article{Jez2008,
  author  = {J{\'e}z{\'e}quel, F. and Chesneaux, J.-M.},
  title   = {{CADNA}: A library for estimating round-off error propagation},
  journal = {Computer Physics Communications},
  volume  = {178},
  number  = {12},
  pages   = {933--955},
  year    = {2008},
  doi     = {10.1016/j.cpc.2008.02.013},
}

@article{scikit-learn,
  author  = {Pedregosa, F. and Varoquaux, G. and Gramfort, A. and Michel, V. and Thirion, B. and Grisel, O. and Blondel, M. and Prettenhofer, P. and Weiss, R. and Dubourg, V. and Vanderplas, J. and Passos, A. and Cournapeau, D. and Brucher, M. and Perrot, M. and Duchesnay, E.},
  title   = {{Scikit-learn}: Machine learning in {Python}},
  journal = {Journal of Machine Learning Research},
  volume  = {12},
  pages   = {2825--2830},
  year    = {2011},
}

@inproceedings{micikevicius2018mixed,
  author    = {Paulius Micikevicius and Sharan Narang and Jonah Alben and Gregory Diamos and Erich Elsen and David Garcia and Boris Ginsburg and Michael Houston and Oleksii Kuchaiev and Ganesh Venkatesh and Hao Wu},
  title     = {Mixed precision training},
  booktitle = {International Conference on Learning Representations},
  year      = {2018},
  url       = {https://openreview.net/forum?id=r1gs9JgRZ},
}

@article{10.1145/3381039,
  author  = {Cherubin, Stefano and Agosta, Giovanni},
  title   = {Tools for reduced precision computation: A survey},
  journal = {{ACM} Computing Surveys},
  volume  = {53},
  number  = {2},
  year    = {2020},
  month   = apr,
  doi     = {10.1145/3381039},
}

@misc{car2024kmeans,
  author        = {Erin Carson and Xinye Chen and Xiaobo Liu},
  title         = {Computing k-means in mixed precision},
  year          = {2026},
  url           = {https://arxiv.org/abs/2407.12208},
  eprint        = {2407.12208},
  archiveprefix = {arXiv},
  primaryclass  = {math.NA},
}

@article{eberhart2015high,
  author  = {Eberhart, Pac{\^o}me and Brajard, Julien and Fortin, Pierre and J{\'e}z{\'e}quel, Fabienne},
  title   = {High performance numerical validation using stochastic arithmetic},
  journal = {Reliable Computing},
  volume  = {21},
  pages   = {35--52},
  year    = {2015},
  url     = {https://interval.louisiana.edu/reliable-computing-journal/volume-21/reliable-computing-21-pp-035-052.pdf},
}

@inproceedings{6614304,
  author    = {Jézéquel, F. and Lamotte, J.-L. and Chubach, O.},
  title     = {Parallelization of discrete stochastic arithmetic on multicore architectures},
  booktitle = {2013 10th International Conference on Information Technology: New Generations},
  pages     = {160--166},
  year      = {2013},
  doi       = {10.1109/ITNG.2013.28},
}

@techreport{parker1997mca,
  author      = {Parker, Douglass Stott},
  title       = {{Monte Carlo} arithmetic: Exploiting randomness in floating-point arithmetic},
  number      = {CSD-970002},
  institution = {UCLA Computer Science Department},
  year        = {1997},
}

@book{higham2002accuracy,
  author    = {Higham, Nicholas J.},
  title     = {Accuracy and stability of numerical algorithms},
  edition   = {2nd},
  publisher = {{SIAM}},
  address   = {Philadelphia, PA},
  year      = {2002},
  doi       = {10.1137/1.9780898718027},
}

@inproceedings{jezequel2020trustworthy,
  author    = {J{\'e}z{\'e}quel, Fabienne and Graillat, Stef and Mukunoki, Daichi and Imamura, Toshiyuki and Iakymchuk, Roman},
  title     = {Can we avoid rounding-error estimation in {HPC} codes and still get trustworthy results?},
  booktitle = {Software Verification},
  series    = {Lecture Notes in Computer Science},
  volume    = {12549},
  pages     = {163--177},
  publisher = {Springer},
  year      = {2020},
  doi       = {10.1007/978-3-030-63618-0_10},
}

@article{doi:10.1137/17M1122918,
  author  = {Carson, Erin and Higham, Nicholas J.},
  title   = {A new analysis of iterative refinement and its application to accurate solution of ill-conditioned sparse linear systems},
  journal = {{SIAM} Journal on Scientific Computing},
  volume  = {39},
  number  = {6},
  pages   = {A2834--A2856},
  year    = {2017},
  doi     = {10.1137/17M1122918},
}

@inproceedings{Ben2022,
  author    = {Ben Khalifa, Dorra and Martel, Matthieu and Adj{\'e}, Assal{\'e}},
  title     = {{POP}: A tuning assistant for mixed-precision floating-point computations},
  booktitle = {Formal Techniques for Safety-Critical Systems},
  editor    = {Hasan, Osman and Mallet, Fr{\'e}d{\'e}ric},
  pages     = {77--94},
  publisher = {Springer},
  address   = {Cham},
  year      = {2020},
}

@misc{chen2026,
  author        = {Xinye Chen and Thibault Hilaire and Fabienne Jézéquel},
  title         = {Floating-point autotuning with customized precisions},
  year          = {2026},
  url           = {https://arxiv.org/abs/2606.08339},
  eprint        = {2606.08339},
  archiveprefix = {arXiv},
  primaryclass  = {cs.MS},
}

@inproceedings{10820739,
  author    = {Vanover, Jackson and Altuntas, Alper and Rubio-González, Cindy},
  title     = {Toward automated precision tuning of weather and climate models: A case study},
  booktitle = {{SC24-W}: Workshops of the International Conference for High Performance Computing, Networking, Storage and Analysis},
  pages     = {148--159},
  year      = {2024},
  doi       = {10.1109/SCW63240.2024.00026},
}

@misc{arar2025mixedprecisionaccumulationneural,
  author        = {El-Mehdi El Arar and Silviu-Ioan Filip and Theo Mary and Elisa Riccietti},
  title         = {Mixed precision accumulation for neural network inference guided by componentwise forward error analysis},
  year          = {2025},
  url           = {https://arxiv.org/abs/2503.15568},
  eprint        = {2503.15568},
  archiveprefix = {arXiv},
  primaryclass  = {cs.LG},
}

@article{VIGNES1993233,
  author  = {Vignes, Jean},
  title   = {A stochastic arithmetic for reliable scientific computation},
  journal = {Mathematics and Computers in Simulation},
  volume  = {35},
  number  = {3},
  pages   = {233--261},
  year    = {1993},
  doi     = {10.1016/0378-4754(93)90003-D},
}

@inproceedings{10387857,
  author    = {Ferro, Quentin and Graillat, Stef and Hilaire, Thibault and Jézéquel, Fabienne},
  title     = {Performance of precision auto-tuned neural networks},
  booktitle = {2023 {IEEE} 16th International Symposium on Embedded Multicore/Many-core Systems-on-Chip ({MCSoC})},
  pages     = {592--599},
  year      = {2023},
  doi       = {10.1109/MCSoC60832.2023.00092},
}

@misc{xiao2017fashionmnist,
  author        = {Xiao, Han and Rasul, Kashif and Vollgraf, Roland},
  title         = {{Fashion-MNIST}: A novel image dataset for benchmarking machine learning algorithms},
  year          = {2017},
  url           = {https://arxiv.org/abs/1708.07747},
  eprint        = {1708.07747},
  archiveprefix = {arXiv},
  primaryclass  = {cs.LG},
}

@unpublished{reynacruz:hal-05367563,
  author = {Reyna Cruz, Maria L. and Ruiz-Rohena, Kristalys and I. Guel, Yahriel and Salgado, Henry and Taldir, Lisa and Pena Ramos, Elian and Cervantes, Natalia and Rivero, Tzetzaith and Ceberio, Martine and Lauter, Christoph and Volkova, Anastasia},
  title  = {Contribution to error analysis of deep neural networks: Case of the activation functions},
  year   = {2025},
  url    = {https://inria.hal.science/hal-05367563},
  note   = {working paper or preprint},
}

@techreport{krizhevsky2009learning,
  author      = {Krizhevsky, Alex},
  title       = {Learning multiple layers of features from tiny images},
  institution = {University of Toronto},
  year        = {2009},
}

@inproceedings{vaswani2017attention,
  author    = {Vaswani, Ashish and Shazeer, Noam and Parmar, Niki and Uszkoreit, Jakob and Jones, Llion and Gomez, Aidan N. and Kaiser, Lukasz and Polosukhin, Illia},
  title     = {Attention is all you need},
  booktitle = {Advances in Neural Information Processing Systems},
  volume    = {30},
  year      = {2017},
}

@techreport{ieee_p3109_interim,
  author      = {{IEEE P3109 Working Group}},
  title       = {Interim report on binary floating-point formats for machine learning},
  institution = {{IEEE}},
  year        = {2024},
  note        = {Version 0.9.1},
}

@misc{ieee754-2019,
  author       = {{IEEE}},
  title        = {{IEEE} standard for floating-point arithmetic},
  year         = {2019},
  doi          = {10.1109/IEEESTD.2019.8766229},
  howpublished = {{IEEE} Std 754-2019},
}

@article{doi:10.1137/18M1226312,
  author  = {Higham, Nicholas J. and Mary, Theo},
  title   = {A new approach to probabilistic rounding error analysis},
  journal = {{SIAM} Journal on Scientific Computing},
  volume  = {41},
  number  = {5},
  pages   = {A2815--A2835},
  year    = {2019},
  doi     = {10.1137/18M1226312},
}

@inproceedings{krizhevsky2012imagenet,
  author    = {Alex Krizhevsky and Ilya Sutskever and Geoffrey E. Hinton},
  title     = {{ImageNet} classification with deep convolutional neural networks},
  booktitle = {Advances in Neural Information Processing Systems},
  volume    = {25},
  pages     = {1097--1105},
  year      = {2012},
}

@inproceedings{he2016deep,
  author    = {Kaiming He and Xiangyu Zhang and Shaoqing Ren and Jian Sun},
  title     = {Deep residual learning for image recognition},
  booktitle = {Proceedings of the {IEEE} Conference on Computer Vision and Pattern Recognition},
  pages     = {770--778},
  year      = {2016},
  doi       = {10.1109/CVPR.2016.90},
}

@inproceedings{devlin2019bert,
  author    = {Jacob Devlin and Ming-Wei Chang and Kenton Lee and Kristina Toutanova},
  title     = {{BERT}: Pre-training of deep bidirectional transformers for language understanding},
  booktitle = {Proceedings of the 2019 Conference of the North American Chapter of the Association for Computational Linguistics: Human Language Technologies},
  volume    = {1},
  pages     = {4171--4186},
  publisher = {Association for Computational Linguistics},
  address   = {Minneapolis, Minnesota},
  year      = {2019},
  doi       = {10.18653/v1/N19-1423},
}

@article{hinton2012deep,
  author  = {Geoffrey Hinton and Li Deng and Dong Yu and George E. Dahl and Abdel-rahman Mohamed and Navdeep Jaitly and Andrew Senior and Vincent Vanhoucke and Patrick Nguyen and Tara N. Sainath and Brian Kingsbury},
  title   = {Deep neural networks for acoustic modeling in speech recognition: The shared views of four research groups},
  journal = {{IEEE} Signal Processing Magazine},
  volume  = {29},
  number  = {6},
  pages   = {82--97},
  year    = {2012},
  doi     = {10.1109/MSP.2012.2205597},
}

@article{hannun2014deep,
  author        = {Awni Hannun and Carl Case and Jared Casper and Bryan Catanzaro and Greg Diamos and Erich Elsen and Ryan Prenger and Sanjeev Satheesh and Shubho Sengupta and Adam Coates and Andrew Y. Ng},
  title         = {{Deep Speech}: Scaling up end-to-end speech recognition},
  journal       = {arXiv preprint arXiv:1412.5567},
  year          = {2014},
  url           = {https://arxiv.org/abs/1412.5567},
  eprint        = {1412.5567},
  archiveprefix = {arXiv},
  primaryclass  = {cs.CL},
}

@article{lecun1998gradient,
  author  = {Yann LeCun and L{\'e}on Bottou and Yoshua Bengio and Patrick Haffner},
  title   = {Gradient-based learning applied to document recognition},
  journal = {Proceedings of the {IEEE}},
  volume  = {86},
  number  = {11},
  pages   = {2278--2324},
  year    = {1998},
  doi     = {10.1109/5.726791},
}

@article{lecun2015deep,
  author  = {Yann LeCun and Yoshua Bengio and Geoffrey Hinton},
  title   = {Deep learning},
  journal = {Nature},
  volume  = {521},
  number  = {7553},
  pages   = {436--444},
  year    = {2015},
  doi     = {10.1038/nature14539},
}

@inproceedings{qiu2026lowprecision,
  author        = {Qiu, Haiquan and Yao, Quanming},
  title         = {Why low-precision transformer training fails: An analysis on {Flash Attention}},
  booktitle     = {The Fourteenth International Conference on Learning Representations ({ICLR} 2026)},
  year          = {2026},
  url           = {https://openreview.net/forum?id=0jHyEKHDyx},
  eprint        = {2510.04212},
  archiveprefix = {arXiv},
  primaryclass  = {cs.LG},
}

@article{lee2024fp8back,
  author  = {Lee, Joonhyung and Bae, Jeongin and Kim, Byeongwook and Kwon, Se Jung and Lee, Dongsoo},
  title   = {To {FP8} and back again: Quantifying reduced precision effects on {LLM} training stability},
  journal = {arXiv preprint arXiv:2405.18710},
  year    = {2024},
  doi     = {10.48550/arXiv.2405.18710},
  url     = {https://arxiv.org/abs/2405.18710},
}

@inproceedings{wortsman2023stable,
  author    = {Wortsman, Mitchell and Dettmers, Tim and Zettlemoyer, Luke and Morcos, Ari S. and Farhadi, Ali and Schmidt, Ludwig},
  title     = {Stable and low-precision training for large-scale vision-language models},
  booktitle = {Advances in Neural Information Processing Systems},
  volume    = {36},
  pages     = {10271--10298},
  publisher = {Curran Associates, Inc.},
  year      = {2023},
  url       = {https://proceedings.neurips.cc/paper_files/paper/2023/hash/20bd42d82998bc61732c00452228e814-Abstract-Conference.html},
}

@inproceedings{liu2026grokking,
  author        = {Liu, Hanqing and Cao, Jianjun and Li, Yuanze and Zhou, Zijian},
  title         = {Grokking or glitching? How low-precision drives slingshot loss spikes},
  booktitle     = {{ICML} 2026 Workshop on High-Dimensional Learning Dynamics},
  year          = {2026},
  doi           = {10.48550/arXiv.2605.06152},
  url           = {https://arxiv.org/abs/2605.06152},
  eprint        = {2605.06152},
  archiveprefix = {arXiv},
  primaryclass  = {cs.LG},
  note          = {Spotlight},
}

@misc{budzinskiy20,
  author        = {Stanislav Budzinskiy and Wenyi Fang and Longbin Zeng and Philipp Petersen},
  title         = {Numerical stability analysis of large language models},
  year          = {2026},
  url           = {https://arxiv.org/abs/2503.10251},
  eprint        = {2503.10251},
  archiveprefix = {arXiv},
  primaryclass  = {math.NA},
}

@article{10.1093/imanum/draf130,
  author  = {Beuzeville, Théo and Buttari, Alfredo and Gratton, Serge and Mary, Theo},
  title   = {Deterministic and probabilistic rounding error analysis of neural networks in floating-point arithmetic},
  journal = {{IMA} Journal of Numerical Analysis},
  pages   = {draf130},
  year    = {2026},
  doi     = {10.1093/imanum/draf130},
}

@article{Goldberg1991FloatingPoint,
  author  = {Goldberg, David},
  title   = {What every computer scientist should know about floating-point arithmetic},
  journal = {{ACM} Computing Surveys},
  volume  = {23},
  number  = {1},
  pages   = {5--48},
  year    = {1991},
  doi     = {10.1145/103162.103163},
}

@inproceedings{vignes1974error,
  author    = {Vignes, Jean and La Porte, M.},
  title     = {Error analysis in computing},
  booktitle = {Information Processing: Proceedings of the 6th IFIP Congress},
  pages     = {610--614},
  publisher = {North-Holland},
  year      = {1974},
}

@inproceedings{ferro2022neural,
  author    = {Ferro, Quentin and Graillat, Stef and Hilaire, Thibault and J{\'e}z{\'e}quel, Fabienne and Lewandowski, Basile},
  title     = {Neural network precision tuning using stochastic arithmetic},
  booktitle = {Computer Arithmetic},
  series    = {Lecture Notes in Computer Science},
  volume    = {13253},
  pages     = {36--53},
  publisher = {Springer},
  year      = {2022},
  url       = {https://hal.science/hal-03682645},
}

@inproceedings{ioualalen2019precision,
  author    = {Ioualalen, A. and Martel, M.},
  title     = {Neural network precision tuning},
  booktitle = {Quantitative Evaluation of Systems},
  series    = {Lecture Notes in Computer Science},
  volume    = {11785},
  pages     = {129--143},
  publisher = {Springer},
  year      = {2019},
}

@misc{lauter2020framework,
  author        = {Lauter, Christoph and Volkova, Anastasia},
  title         = {A framework for semi-automatic precision and accuracy analysis for fast and rigorous deep learning},
  year          = {2020},
  url           = {https://arxiv.org/abs/2002.03869},
  eprint        = {2002.03869},
  archiveprefix = {arXiv},
  primaryclass  = {cs.LG},
}

@inproceedings{faraone2019montecarlo,
  author    = {Faraone, Julian and Leong, Philip},
  title     = {{Monte Carlo} deep neural network arithmetic},
  booktitle = {International Conference on Learning Representations},
  year      = {2019},
  url       = {https://openreview.net/forum?id=HyePberFvH},
}

@inproceedings{khalifa2024fixedpoint,
  author    = {Ben Khalifa, Dorra and Martel, Matthieu},
  title     = {Efficient implementation of neural networks usual layers on fixed-point architectures},
  booktitle = {Proceedings of the 25th International Workshop on Software and Compilers for Embedded Systems},
  year      = {2024},
  doi       = {10.1145/3652032.3657578},
}

@article{lohar2023sound,
  author  = {Lohar, Debasmita and Jeangoudoux, Clothilde and Volkova, Anastasia and Darulova, Eva},
  title   = {Sound mixed fixed-point quantization of neural networks},
  journal = {{ACM} Transactions on Embedded Computing Systems},
  volume  = {22},
  number  = {5s},
  year    = {2023},
  doi     = {10.1145/3609118},
}

@misc{gernigon2024adaqat,
  author        = {Gernigon, C{\'e}dric and Coggiola, Cl{\'e}ment and Filip, Silviu-Ioan and Sentieys, Olivier and Bruno, Micka{\"e}l},
  title         = {{AdaQAT}: Adaptive bit-width quantization-aware training},
  year          = {2024},
  url           = {https://arxiv.org/abs/2404.16876},
  eprint        = {2404.16876},
  archiveprefix = {arXiv},
  primaryclass  = {cs.LG},
}

@inproceedings{bengio2013estimating,
  author    = {Bengio, Yoshua and L{\'e}onard, Nicholas and Courville, Aaron},
  title     = {Estimating or propagating gradients through stochastic neurons for conditional computation},
  booktitle = {Advances in Neural Information Processing Systems},
  volume    = {26},
  year      = {2013},
}

@inproceedings{zhang2015character,
  author    = {Zhang, Xiang and Zhao, Junbo and LeCun, Yann},
  title     = {Character-level convolutional networks for text classification},
  booktitle = {Advances in Neural Information Processing Systems},
  volume    = {28},
  year      = {2015},
}

@misc{jezequel2024innerproducts,
  author = {J{\'e}z{\'e}quel, Fabienne and Mary, Theo},
  title  = {Probabilistic estimation of the accuracy of inner products and application to stochastic validation},
  year   = {2024},
  url    = {https://hal.science/hal-04554459},
}

\end{document}